\documentclass[11pt,reqno,english,letterpaper]{amsart}
\usepackage{amsfonts,amsmath,latexsym,verbatim,amscd,mathrsfs,color,array}
\usepackage[colorlinks=true]{hyperref}

\usepackage{amssymb,amsthm,graphicx,color}
\usepackage[hmargin=2.5cm, vmargin=2.5cm]{geometry}
\usepackage{float}
\usepackage{pdfsync}
\usepackage{epstopdf}

\newcommand{\ass}{\quad\mbox{as}\quad}

\newcommand{\ttt}{\tilde }

\newcommand{\R} {\mathbb R}

\newcommand{\cuad}{{\sqcap\kern-.68em\sqcup}}

\newcommand{\ve}{\varepsilon}

\newcommand{\be}{\begin{equation}}
	\newcommand{\ee}{\end{equation}}

\newtheorem{lemma}{Lemma}[section]
\newtheorem{propositio}{Proposition}[section]
\newtheorem{theorem}{Theorem}
\newtheorem{corollary}{Corollary}[section]
\newtheorem{remark}{Remark}[section]
\newcommand{\bremark}{\begin{remark} \em}
	\newcommand{\eremark}{\end{remark} }

\numberwithin{equation}{section}

\begin{document}
	
	\title{The double sphere solution in the liquid drop model}
	
	\author[M.~del Pino]{Manuel del Pino}
	\address{\noindent M.dP.:  Department of Mathematical Sciences University of Bath,
		Bath BA2 7AY, United Kingdom.}
	\email{mdp59@bath.ac.uk}
	
	\author[R.~Frank]{Rupert L. Frank}
	\address{\noindent R.L.F.: Mathematics Institute, University of Munich, Theresienstr.~39, 80333 Munich, Germany, and Munich Center for Quantum Science and Technology, Schellingstr.~4, 80799 Munich, Germany.
	}
	\email{r.frank@lmu.de}
	
	\author[M.~Musso]{Monica Musso}
	\address{\noindent M.M.:  Department of Mathematical Sciences University of Bath,
		Bath BA2 7AY, United Kingdom.}
	\email{mm2683@bath.ac.uk}
	
	\begin{abstract}
		We consider the problem of finding critical domains $\Omega\subset\R^3$ for the energy functional 
		$$
		\mathcal E (\Omega) = {\rm Per}\,(\Omega)   + \frac 12 \iint_{\Omega\times \Omega } \frac{dx\,dy}{|x-y|}
		$$
		under the volume constraint $|\Omega|=V$.
		We look for smooth, embedded, compact surfaces $\partial\Omega$ that solve this problem. We construct an axially symmetric, non-minimizing solution that, for a sufficiently small $V>0$, resembles the union of two balls with volume $V/2$ connected by a tiny, approximately catenoidal neck with a width of the order $V^{\frac 43}$. 
        
	\end{abstract}
	
	\maketitle

	\section{Introduction}
The liquid drop model, introduced by Gamow \cite{Gamow} and developed by Bohr and
Wheeler \cite{BohrWheeler1939} in their celebrated theory of nuclear fission, describes a nucleus as an incompressible
charged liquid drop. In its simplest form, all measurable sets $\Omega\subset\R^3$ are possible nuclear shapes and their volume $|\Omega|$ is proportional to the number of nucleons. The corresponding energy is
$$
\mathcal E(\Omega)
=
{\rm Per}(\Omega)
+
\frac12
\iint_{\Omega\times\Omega}
\frac{dx\,dy}{|x-y|},
$$
where the first term is the perimeter in the sense of De Giorgi, which, in the case of sufficiently nice boundary, coincides with the surface area. The second term models the Coulomb repulsion of the protons.

Of physical interest are critical points of the energy functional $\mathcal E(\Omega)$ under the volume constraint
$$
|\Omega|=V.
$$
Such critical points satisfy the Euler--Lagrange equation
$$
H_{\partial\Omega}(x)
+
\int_\Omega\frac{dy}{|x-y|}
=
\lambda,
\qquad x\in\partial\Omega,
$$
where $H_{\partial\Omega}$ denotes the mean curvature of $\partial\Omega$ and $\lambda$ is
the Lagrange multiplier associated with the volume constraint. Balls
provide the simplest family of solutions.

The variational theory of global minimizers has been extensively
developed over the past decades, for instance in the papers \cite{KnupferMuratov2013,KnupferMuratov2014,LuOtto2014,Julin2014,FrankLieb2015,FrankKillipNam2016,FrankNam2021,ChodoshRuohoniemi2025}. The final answer was very recently provided in a beautiful work by Chodosh and Gianocca \cite{ChodoshGianocca2026}, where it is shown that balls are the unique volume-constrained minimizers for $V \le V_*$, while no minimizers exist for $V> V_*$. Here $V_* = 5 (2-2^{2/3})/ (2^{2/3} -1)$. The theory of global
minimizers may therefore be regarded as essentially complete. This
naturally shifts the focus toward the existence and geometry of
non-minimizing critical points.

\medskip 
Among the known examples of non-minimizing stationary solutions are
the toroidal and double-toroidal solutions of Ren and Wei
\cite{RW11,RW14}, together with the compact stationary surfaces
constructed in \cite{DelPinoMussoZuniga2025}, obtained by gluing
methods.

Among stationary solutions, the Bohr--Wheeler branch occupies a
distinguished position because of its connection with nuclear fission.
Bohr and Wheeler formally argued that the spherical equilibrium
should lose stability at the critical volume $V=10$, giving rise to a
branch of axially symmetric equilibria. A rigorous proof
of this local bifurcation was obtained in
\cite{Frank2019}. More recently, the complete sequence of axially
symmetric bifurcations from the sphere was established in
\cite{DeRegibus}.

Bohr and Wheeler predicted that the branch bifurcating from the ball at $V=10$ extends all the way to arbitrarily small volumes. Along this branch, the solutions evolve from nearly spherical prolate drops into dumbbell-like configurations connected by an increasingly thin neck. Near $V=0$ the lobes are almost spherical. Throughout the decaudes, this prediction of Bohr and Wheeler has been numerically confirmed and made more precise; see, e.g., \cite{FrankelMetropolis1947,BusinaroGallone1955,CohenSwiatecki1962,CohenSwiatecki1963,Nix1969,ThomasDaviesSierk1985}. We highlight the recent numerical
computations of Xu and Du \cite{XuDu2023}. These computations suggest, in particular, that for small volumes the neck that connects the two almost spherical lobes is asymptotically catenoidal. Despite this numerical evidence we stress that proving the existence of a global branch from $V=10$ to $V=0$ is still an open problem.

The purpose of the present paper is to give the first rigorous
construction of the small-volume portion of the Bohr--Wheeler branch.
Our solutions exhibit precisely the geometry
predicted numerically: two almost spherical droplets joined by a single catenoidal
neck.

\begin{theorem}[Main Theorem]\label{thm:main}
There exists $V_0>0$ such that, for every
\[
0<V<V_0,
\]
the liquid drop energy functional admits a smooth embedded critical point
$\Omega_V$ of volume $V$.

The sets $\Omega_V$ are axially symmetric and even with respect to reflections in a plane orthogonal to the axis.

As $V\to0$, the boundary $\partial\Omega_V$ consists of two nearly
spherical lobes connected by a single catenoidal neck. More precisely,

\begin{enumerate}
\item
each lobe is a normal graph over the boundary of a sphere of volume
\[
\frac V2+o(V),
\]
and away from the neck the convergence is smooth;

\item
after rescaling by the neck scale, the connecting region converges in
$C^\infty_{\rm loc}$ to the standard catenoid;

\item
the neck is formed at the intermediate scale
\[
r\sim V^{4/3}.
\]
\end{enumerate}
\end{theorem}

\medskip

Our construction does not rely on a global continuation of the
Bohr--Wheeler branch. Instead, we work directly in the small-volume
regime. We first construct an approximate surface consisting of two
large spherical pieces connected through a suitable perturbation of a
catenoid, and then solve the full nonlocal Euler--Lagrange equation by
means of an inner--outer gluing procedure combined with a
Lyapunov--Schmidt reduction.

Note that, after rescaling a set $\Omega$ of volume $V$ to a set $E=V^{-1/3}\Omega$ of volume one, the energy
takes, up to an irrelevant multiplicative constant depending on $V$, the form
\[
{\rm Per}(E)
+
\frac V2
\iint_{E\times E}\frac{dx\,dy}{|x-y|}.
\]
Thus, in the small-volume regime that we are considering, the perimeter is the dominant term, while the Coulomb interaction appears as a perturbative correction. It should be stressed, however, that the perturbation is crucial for the existence of the sets $\Omega_V$.

It is convenient to perform the construction after another, different rescaling.
Introducing a small parameter $\varepsilon>0$, the spherical pieces have
radius
\[
R=\varepsilon^{-1}.
\]
The volume is not prescribed a priori but is selected by a final
solvability condition.
For all sufficiently small $\ve>0$, the gluing procedure produces a
smooth domain $\Omega_\ve$ together with constants $m_\ve>0$ and $\lambda_\ve$ such that
\[
H_{\partial\Omega_\ve}
+
m_\ve\ve^3
\int_{\Omega_\ve}\frac{dy}{|x-y|}
=
\lambda_\ve
\qquad\text{on }\partial\Omega_\ve.
\]
Moreover,
\[
|\Omega_\ve|
\sim
\frac{8\pi}{3\ve^3},
\qquad
m_\ve
\sim
\frac{9}{2\pi}\ve,
\qquad \ass \ve\to0.
\]
We now rescale $\Omega_\ve$ in two steps, first by $\ve$ and then by
$m_\ve^{1/3}$. We define
\[
\widetilde\Omega_\ve
:=
m_\ve^{1/3}\ve\,\Omega_\ve.
\]
Under this scaling, the boundary equation becomes
\[
H_{\partial\widetilde\Omega_\ve}
+
\int_{\widetilde\Omega_\ve}\frac{dy}{|x-y|}
=
\widetilde\lambda_\ve
\qquad\text{on }\partial\widetilde\Omega_\ve.
\]
At the same time, the volume transforms according to
\[
|\widetilde\Omega_\ve|
=
m_\ve\ve^3|\Omega_\ve| \sim 12 \ve ,  \quad \ass \ve \to 0.
\]

The analysis involves two main difficulties.

The first arises from the interaction of two very different geometric
scales. On the neck, the linearized mean-curvature operator is a perturbation of the
Jacobi operator of the catenoid, whose growing Jacobi field must be
controlled on a region whose size diverges as
$\varepsilon\to0$. On the spherical regions, by contrast, the
linearized mean-curvature operator possesses a one-dimensional kernel within the axially symmetric class,
generated by the vertical
translation mode. Projecting onto this mode produces the finite
dimensional balancing equation that determines the volume.

The second, and perhaps most delicate, difficulty is the singular
transition region connecting the two nearly spherical components. The
two profiles are matched at the intermediate scale
\[
    r\sim\ve^{-3/4}.
\]
Interpolating between them creates a localized mean-curvature error
that is too large to be included directly in the perturbative
remainder. We therefore introduce an explicit first correction on the
neck that cancels its leading part. The remaining error can then be
controlled in weighted spaces adapted to the catenoidal geometry.

\medskip
This explicit correction is one of the key new ingredients of the
construction. It reduces the transition error below the perturbative
threshold, thereby making the nonlinear gluing argument possible.

A fundamental feature of the construction is the role played by the non-local
Coulomb interaction in the final balancing condition. For the
corresponding purely local constant-mean-curvature problem, a compact
embedded surface consisting of two spherical pieces connected by a
catenoidal neck cannot exist because of the classical force-balancing
obstruction. In the liquid drop model, however, the Coulomb
interaction contributes at exactly the order needed to compensate this
obstruction. At leading order, projection onto the spherical
translation mode yields
\[
-\frac{4\pi^2}{9}m_\ve \varepsilon
+
2\pi\varepsilon^2
+
o(\varepsilon^2)
=
0,
\]
which immediately gives
\[
m_\ve 
\sim
\frac{9}{2\pi}\varepsilon \quad {\mbox {and hence}} \quad V_\ve \sim 12 \ve.
\]
Although perturbative throughout much of the construction, the
nonlocal term is therefore decisive in the finite-dimensional
reduction, making the existence of these compact double-lobe
configurations possible.

\medskip
Our approach belongs to the broad family of gluing methods that have proved
particularly effective in geometric analysis and nonlinear elliptic problems.
In the theory of constant-mean-curvature surfaces, gluing constructions were
developed in the seminal works of Kapouleas \cite{kapouleas1990,kapouleas1991} and subsequently in the work of
Mazzeo and Pacard \cite{MazzeoPacard,MazzeoPacard1999}, where elementary geometric pieces are joined
through suitably chosen neck regions and the resulting approximate
configuration is corrected to an exact solution. Related ideas have also been
used extensively in the construction of concentrating and multi-bump solutions
for nonlinear Schr\"odinger and other semilinear elliptic equations; see, among
others, the works \cite{Malchiodi, MussoPacardWei}. Our construction combines this
geometric gluing philosophy with an inner--outer gluing scheme: the catenoidal
neck and the spherical regions are treated at their natural, different scales,
and the corresponding corrections are coupled through their overlap region.
A closely related mechanism appears in the construction of overhanging water
waves \cite{DdPMW}, where a hairpin profile is glued to a disk and a half-space. In this
sense, the present catenoid--sphere construction can be viewed as the geometric
analogue, for the liquid drop problem, of the hairpin--disk--half-space
decomposition arising in the water-wave problem. We reiterate that the main additional feature
here is the nonlocal Coulomb interaction, which couples the different regions
and enters essentially in the final balancing condition.

\medskip
In a future work we will prove that, for all sufficiently small $V>0$, the sets $\Omega_V$ have Morse index two as volume-constrained critical points of the liquid drop energy functional $\mathcal E$. This has been suggested in the physics literature (see, e.g., \cite{CohenSwiatecki1963} and \cite[Figure 3]{Nix1969}). The two negative directions correspond to axially symmetric perturbations, one even and one odd under reflections in the symmetry plane. The even and odd perturbations corresponds, respectively, to neck pinching and volume transfer. The most subtle part of this analysis is to show that another axially symmetric, even perturbation given by pulling apart the lobes is stable. The proof of this latter fact relies heavily on the techniques that we introduce in the present paper.

\smallskip
The following picture emerges for the stability/instability of the Bohr--Wheeler branch $\Omega_V$. For every $0<V<10$ these sets are unstable against a certain axially symmetric, even perturbation. For $10-V\ll 1$, this is the perturbation responsible for the bifurcation, and for $V\ll 1$ this is the neck pinching instability mentioned before. At a certain parameter value $V\approx 3.96$, the Businaro--Gallone point, a second instability direction emerges, corresponding to an axially symmetric odd perturbation. For $V\ll 1$ this is the volume transfer instability mentioned before. Away from the Businaro--Gallone point, the sets $\Omega_V$ have nullity five, corresponding to infinitesimal translations in the three coordinate directions and to infinitesimal rotations of the symmetry axis.

\subsection{Scheme of the proof}

We briefly describe the organization of our proof.

In Section~\ref{approx-domain} we construct the approximate surface.
Starting from a suitable perturbation of the standard catenoid, we match
it with two spheres of radius $R=\varepsilon^{-1}$.
This determines both the intermediate matching scale
$r\sim\varepsilon^{-3/4}$ and the relative displacement of the
spherical components. The resulting surface $\Sigma_0^\varepsilon$
provides the starting point for the analysis.

In Sections~\ref{linear-neck} and \ref{linear-sphere} we develop the
linear theory on the neck and on the spherical regions. On the neck we
construct a right inverse for the catenoidal Jacobi operator in weighted
spaces and introduce the explicit first correction that removes the
leading transition error. On the spherical regions we solve the
linearized problem modulo the vertical translation Jacobi field.
Section~\ref{reduced-system} combines these ingredients into a coupled
inner--outer system. 

In Section~\ref{final1} we solve the nonlinear neck
equation by a contraction argument, obtaining a neck correction with the
continuity and estimates required for the outer problem.
In Section~\ref{final-sphere} we solve the projected nonlinear problem
on the spherical regions, reducing the construction to a single scalar
balancing condition. Finally, in Section~\ref{final-mass} we compute the leading-order
projection onto the spherical translation mode and solve the resulting
balancing equation. This determines the parameter
$m_\varepsilon$, removes the final obstruction, and yields an exact
solution of the Euler--Lagrange equation. After returning to the
original variables, a continuity argument shows that every sufficiently
small prescribed volume is attained.

\smallskip
The appendices collect the geometric expansions used throughout the
paper. Appendix~\ref{appe2} establishes the general expansion of the
mean curvature under normal perturbations, while
Appendix~\ref{app3} develops the corresponding formulas and nonlinear
estimates for the catenoid and the sphere.

\medskip
	\subsection{Notation}

Throughout the paper we use the following notation.
We denote by $S$ the unit sphere in $\R^3$ centered at the origin,
and by
\[
P_S=-e_3
\]
its south pole.
For $R>0$ and $d>0$, we write
\[
S_R((R+d)e_3),
\]
or simply $S_R$, for the sphere of radius $R$ centered at
$(R+d)e_3$. Its upper and lower hemispheres are denoted by
\[
S_R^+((R+d)e_3),
\qquad
S_R^-((R+d)e_3),
\]
or simply by $S_R^+$ and $S_R^-$.
We write $(r,\theta)$ for polar coordinates in the
$(x_1,x_2)$-plane,
\[
r=\sqrt{x_1^2+x_2^2},
\qquad
(x_1,x_2)=(r\cos\theta,r\sin\theta).
\]
Throughout the paper, $\chi_0$ denotes a fixed smooth cutoff function
satisfying
\begin{equation}\label{chi0}
0\le\chi_0\le1,
\qquad
\chi_0(s)=1
\quad\text{for } s\le1,
\qquad
\chi_0(s)=0
\quad\text{for } s\ge2.
\end{equation}

\subsection*{Acknowledgments}
R.L.F. acknowledges partial support through the German Research Foundation grants EXC-2111-390814868 and TRR 352–Project-ID 470903074.
The research of M.d.P. is supported by the Royal Society Research Professorship grant RP-R1-180114 and by the ERC/UKRI Horizon Europe grant ASYMEVOL, EP/Z000394/1.

	\section{Construction of the Approximate Domain}
	\label{approx-domain}
	
	In this section we construct the approximate surface around which the
	nonlinear analysis will be carried out. The construction reflects the
	geometry suggested by the numerical computations of Xu--Du
	\cite{XuDu2023}: two almost spherical components connected by a very thin
	neck. The two regions naturally live at different geometric scales.
	Away from the symmetry axis the surface is close to two spheres of radius
	\[
	R=\frac1\ve,
	\]
	whereas the neck has size one in the variables used below and is therefore
	naturally described by a catenoid.
	
	The construction proceeds in three steps. We first perturb a standard
	catenoid so that its mean curvature agrees, to leading order, with the
	mean curvature of the large spheres. We then determine the natural
	intermediate scale at which the modified catenoid can be matched to the
	spherical profile, and choose the vertical displacement of the spheres
	accordingly. Finally, the two profiles are joined by means of a smooth
	cutoff on a region of size \(r\sim\ve^{-3/4}\).
	
	More precisely, we construct an axially symmetric domain
	\(\Omega_0^\ve\) whose boundary connects two spheres of radius \(R\),
	centred respectively at
	\[
	(R+d)e_3
	\qquad\text{and}\qquad
	-(R+d)e_3,
	\]
	where the displacement \(d\) will be determined below. The resulting
	surface is invariant under rotations around the \(x_3\)-axis and under
	the reflection
	\[
	x_3\longmapsto -x_3.
	\]
	
	\subsection{The modified catenoid}
	
	We begin with the standard catenoid of neck-size one. Its upper half is
	parametrized by
	\begin{equation}\label{par-catenoid}
		X_0(r,\theta)
		=
		\left(
		r\cos\theta,\,
		r\sin\theta,\,
		F_0(r)
		\right),
		\qquad
		r\ge1,
		\qquad
		\theta\in[0,2\pi),
	\end{equation}
	where
	$$
		F_0(r)
		=
		\log\!\left(r+\sqrt{r^2-1}\right).
	$$
	The corresponding outward unit normal is
	\[
	\nu_0(r,\theta)
	=
	\frac1{\sqrt{1+(F_0')^2}}
	\left(
	F_0'(r)\cos\theta,\,
	F_0'(r)\sin\theta,\,
	-1
	\right).
	\]
	The catenoid is minimal, whereas the spheres of radius
	\(R=\ve^{-1}\) have mean curvature \(2\ve\). We therefore first modify
	the catenoid by a normal perturbation whose linearized mean curvature is
	equal to \(2\ve\).
	
	With the convention used throughout the paper,

    \[
H_{\Sigma_h}
=
H_\Sigma
-
J_\Sigma[h]
+
\mathcal Q_\Sigma[h],
\qquad
J_\Sigma
=
\Delta_\Sigma+|A|^2.
\]
Here $\Delta_\Sigma$ is the Laplace--Beltrami operator on $\Sigma$, with the analytic sign convention, and $|A|$ is the norm of the second fundamental form. The nonlinear remainder $\mathcal Q_\Sigma[h]$ is discussed in detail later, in Proposition~\ref{generalH}.

For axially symmetric functions on the catenoid,
\[
J_{\Sigma_0}[h]
=
\frac{r^2-1}{r^2}h''
+
\frac1r h'
+
\frac2{r^4}h.
\]
For a detailed derivation we refer to Appendix \ref{app3}, Subsection \ref{cat}, Proposition \ref{sscat}. We therefore solve
	\begin{equation}\label{111}
		J_{\Sigma_0}[h]
		=
		g,
		\qquad
		h(1)=0,
		\qquad
		\lim_{r\to1^+}\sqrt{r^2-1}\,h'(r)=0.
	\end{equation}
	The last condition is precisely the regularity condition corresponding
	to an even extension across the waist of the catenoid. Indeed, after
	writing \(r=\cosh s\), a radial function \(h(r)\) becomes an even
	function of \(s\), and regularity at \(s=0\) is equivalent to the last
	condition in \eqref{111}.
	
	The Jacobi operator has the two independent homogeneous solutions
	\[
	y_1(r)
	=
	\frac{\sqrt{r^2-1}}r,
	\qquad
	y_2(r)
	=
	1-
	\frac{\sqrt{r^2-1}}r
	\log\!\left(r+\sqrt{r^2-1}\right).
	\]
	Variation of parameters gives the solution
$$
		h(r)
		=
		\frac{\sqrt{r^2-1}}r
		\int_1^r
		\frac{\eta^2}{(\eta^2-1)^{3/2}}
		\left(
		\int_1^\eta s\,g(s)\,ds
		\right)d\eta .
	$$
	In order that the resulting normal graph have mean curvature
	\(2\ve\) at linear order, we take
	\[
	g\equiv-2\ve.
	\]
	We denote by \(h_0\) the corresponding solution. A direct computation
	gives
	\begin{equation}\label{h0}
		\begin{aligned}
			h_0(r)
			&=
			-\frac{\ve}{2}(r^2-1)
			\left(
			1+
			\frac{
				\log\!\left(r+\sqrt{r^2-1}\right)
			}{
				r\sqrt{r^2-1}
			}
			\right),
			\\[1ex]
			h_0'(r)
			&=
			-\ve r
			-\frac{\ve}{2r}
			-\frac{\ve}{2}
			\frac{
				\log\!\left(r+\sqrt{r^2-1}\right)
			}{
				r^2\sqrt{r^2-1}
			},
			\\[1ex]
			h_0''(r)
			&=
			-\ve
			+
			\frac{\ve}{2}
			\frac{r^2-2}{r^2(r^2-1)}
			+
			\frac{\ve}{2}
			\frac{
				(3r^2-2)
				\log\!\left(r+\sqrt{r^2-1}\right)
			}{
				r^3(r^2-1)^{3/2}
			}.
		\end{aligned}
	\end{equation}
	In particular, as \(r\to\infty\),
	\begin{equation}\label{h0-asymptotics}
		h_0(r)
		=
		-\frac{\ve}{2}r^2
		+
		O(\ve\log r),
		\qquad
		h_0'(r)
		=
		-\ve r
		+
		O\!\left(\frac{\ve}{r}\right),
		\qquad
		h_0''(r)
		=
		-\ve
		+
		O\!\left(\frac{\ve}{r^2}\right).
	\end{equation}
	The explicit formula also gives the behavior at the waist. As
	$r\to1^+$,
	$$
		\begin{aligned}
			h_0(r)
			&=
			-2\ve(r-1)
			+ \ve O((r-1)^2)
			, \quad 
			h_0'(r)
			=
			-2\ve
			+\ve  O(
			(r-1)), \quad \\
			h_0''(r)
			&=
			\frac{2\ve}{3}
			+
			O\!\left(\ve(r-1)\right).
		\end{aligned}
$$
	In particular,
	\[
	h_0(1)=0,
	\qquad
	\lim_{r\to1^+}\sqrt{r^2-1}\,h_0'(r)=0,
	\]
	so the normal graph extends smoothly and evenly across the waist of
	the catenoid.

	The modified catenoid is the normal graph of \(h_0\) over
	\(\Sigma_0\). We write it as
	$$
		\left(
		q(r)\cos\theta,\,
		q(r)\sin\theta,\,
		F(r)
		\right),
	$$
	where
	\begin{equation}\label{def-qF}
		q(r)
		=
		r+
		\frac{h_0(r)F_0'(r)}
		{\sqrt{1+(F_0')^2}},
		\qquad
		F(r)
		=
		F_0(r)
		-
		\frac{h_0(r)}
		{\sqrt{1+(F_0')^2}}.
	\end{equation}
	Since
	\[
	\frac{F_0'}{\sqrt{1+(F_0')^2}}
	=
	\frac1r,
	\]
	we have the useful exact identity
$$
		q(r)=r+\frac{h_0(r)}r.
	$$
	It follows from \eqref{h0-asymptotics} that, for \(r\gg1\),
$$
		|q-r|
		\lesssim
		\ve r,
		\qquad
		|q'-1|
		\lesssim
		\ve,
		\qquad
		|q''|
		\lesssim
		\frac{\ve}{r}.
$$
	By construction the linear contribution to the mean-curvature defect of
	the catenoid is cancelled:
	\[
	-J_{\Sigma_0}[h_0]=2\ve.
	\]
	Thus the modified catenoid has mean curvature \(2\ve\) up to
	higher-order terms. The precise uniform error estimates in the range
	used for the gluing construction will be established below in
	Proposition \ref{prop-error}.
	
	\subsection{Matching with the spherical profile}
	
	Consider the sphere of radius \(R=\ve^{-1}\) centred at
	\((R+d)e_3\). Its lower hemisphere can be written as the graph
	\begin{equation}\label{defG}
		G(r)
		=
		R+d-\sqrt{R^2-r^2},
		\qquad
		0\le r<R.
	\end{equation}
	For \(r\ll R\),
	\begin{equation}\label{G-exp}
		G(r)
		=
		d+\frac{\ve}{2}r^2
		+
		\frac{\ve^3}{8}r^4
		+
		O(\ve^5r^6),
\quad 
		G'(r)
		=
		\ve r
		+
		\frac{\ve^3}{2}r^3
		+
		O(\ve^5r^5).
	\end{equation}
	
	Using \eqref{h0-asymptotics} in \eqref{def-qF}, one obtains
	\begin{equation}\label{F-exp}
		F(r)
		=
		\log(2r)
		+
		\frac{\ve}{2}r^2
		+
		O(\ve\log r)
		+
		O(r^{-2}),
	\end{equation}
	and
$$
		F'(r)
		=
		\ve r
		+
		\frac1r
		+
		O\!\left(\frac{\ve}{r}\right)
		+
		O(r^{-3}).
$$
	
	The natural scale at which the modified catenoid and the sphere have
	comparable slopes is obtained by balancing the first terms in
	\(F'-G'\) not already shared by the two profiles:
	\[
	\frac1r
	\sim
	\frac{\ve^3}{2}r^3.
	\]
	This gives the reference scale
	\begin{equation}\label{defrbar}
		\bar r
		=
		2^{1/4}\ve^{-3/4}.
	\end{equation}
	The particular numerical value in \eqref{defrbar} is used only to fix
	the vertical displacement \(d\); the actual interpolation will be
	performed on a region
	\[
	r\sim\delta\ve^{-3/4},
	\]
	with \(\delta>0\) fixed and small.

\medskip
We choose the vertical displacement \(d\) so that the modified catenoid and the spherical profile match at leading order at the reference scale \(r=\bar r\). Using \eqref{G-exp}, \eqref{F-exp}, and
\(\bar r=2^{1/4}\ve^{-3/4}\), we define
\begin{equation}\label{defd}
d
=
-\frac34\log\ve
+
\frac54\log2
-
\frac14.
\end{equation}
With this choice one has
\[
F(\bar r)-G(\bar r)
=
O(\ve^{1/2}|\log\ve|),
\]
which is sufficient for the subsequent interpolation argument.


\medskip        
	The role of the scale \(\bar r\) is therefore to identify the natural
	intermediate regime and determine the displacement \(d\). We do not
	join the two profiles sharply at \(r=\bar r\). Instead, the
	interpolation is spread over a transition annulus where the derivatives
	of the cutoff supply the required small factors.

    \medskip  
	\subsection{Interpolation}
	
	Let \(\delta>0\) be fixed and sufficiently small and define
	\begin{equation}\label{defchi}
		\chi(r)
		=
		\chi_0
		\left(
		\frac{\ve^{3/4}}{\delta}r
		\right).
	\end{equation}
	The cut-off function $\chi_0$ is defined in \eqref{chi0}. Thus
	\[
	\chi(r)=1
	\quad\text{if}\quad
	r\le\delta\ve^{-3/4},
	\qquad
	\chi(r)=0
	\quad\text{if}\quad
	r\ge2\delta\ve^{-3/4}.
	\]
	We set
	\begin{equation}\label{defPQ}
		Q(r)
		=
		\chi(r)q(r)
		+
		(1-\chi(r))r,
		\qquad
		P(r)
		=
		\chi(r)F(r)
		+
		(1-\chi(r))G(r).
	\end{equation}
	
	In the transition region
	\[
	\delta\ve^{-3/4}
	<
	r
	<
	2\delta\ve^{-3/4}
	\]
	it is useful to write
	\[
	r=s\ve^{-3/4},
	\qquad
	\delta<s<2\delta.
	\]
	Using \eqref{defd} and the preceding expansions, one obtains
$$
		F(r)-G(r)
		=
		\log s
		-\frac14\log2
		+\frac14
		-\frac{s^4}{8}
		+
		O(\ve^{1/2}|\log\ve|),
$$
	uniformly for \(\delta<s<2\delta\). In particular,
\begin{equation}\label{FG-derivatives}
		|F-G|
		\le
		C_\delta, \quad
		|F'-G'|
		\le
		C_\delta\ve^{3/4},
		\qquad
		|F''-G''|
		\le
		C_\delta\ve^{3/2}.
	\end{equation}
	
	These estimates explain why the interpolation is possible even though
	the two profiles do not coincide pointwise throughout the transition
	annulus: the factors
	\[
	|\chi'|
	\lesssim
	\delta^{-1}\ve^{3/4},
	\qquad
	|\chi''|
	\lesssim
	\delta^{-2}\ve^{3/2}
	\]
	make the induced curvature defect perturbative, as will become clear later in \eqref{def-E02}.
	
	The upper half of the approximate surface is defined by
	\begin{equation}\label{defSigma0ve12}
		\Sigma_0^{\ve,+}
		=
		\Sigma_0^{\ve,+,1}
		\cup
		\Sigma_0^{\ve,+,2},
	\end{equation}
	where
	\[
	\Sigma_0^{\ve,+,1}
	=
	\left\{
	(Q(r)\cos\theta,Q(r)\sin\theta,P(r)):
	1<r\le R,\ 
	\theta\in[0,2\pi)
	\right\},
	\]
	and
	\[
	\Sigma_0^{\ve,+,2}
	=
	S_R^+
	=
	\left\{
	(r\cos\theta,r\sin\theta,\bar G(r)):
	0\le r\le R,\ 
	\theta\in[0,2\pi)
	\right\},
	\]
	with
	\begin{equation}
		\bar G(r)
		=
		R+d+\sqrt{R^2-r^2}.
	\end{equation}
	The lower half is obtained by reflection:
	\[
	\Sigma_0^{\ve,-}
	=
	\{
	(x_1,x_2,-x_3):
	(x_1,x_2,x_3)\in\Sigma_0^{\ve,+}
	\}.
	\]
	Finally,
	\begin{equation}\label{defSigma0ve}
		\Sigma_0^\ve
		=
		\Sigma_0^{\ve,+}
		\cup
		\Sigma_0^{\ve,-},
		\qquad
		\Omega_0^\ve
		\text{ denotes the enclosed region.}
	\end{equation}
	
	By construction, \(\Sigma_0^\ve\) is a smooth embedded surface of
	revolution. It agrees with the modified catenoid in the inner region,
	with the spherical pieces in the outer region, and interpolates smoothly
	between them on the scale \(r\sim\ve^{-3/4}\).
	
	\subsection{Residual of the approximate surface}
	
We now estimate the residual of the Euler--Lagrange equation
\[
H_{\Sigma_0^\ve}(x)
+
m\ve^3
\int_{\Omega_0^\ve}
\frac{dy}{|x-y|},
\qquad
x\in\Sigma_0^\ve.
\]
At this stage, the parameter $m>0$ is regarded as free. It will be determined later by the finite-dimensional solvability condition and will turn out to be of order $\ve$.
    
	The first ingredient is a precise expansion of the mean-curvature operator
	under normal perturbations. Although the approximate surface constructed above
	is rotationally symmetric, it is convenient to begin with a completely general
	result valid for arbitrary smooth embedded surfaces. This separates the
	geometric identities from the particular features of the present construction
	and allows the subsequent computations to be carried out in a systematic way.

	\begin{propositio}
		\label{generalH}
		Let $\Sigma\subset\mathbb R^3$ be a smooth oriented surface with metric $g$,
		second fundamental form $A$, unit normal $\nu$, and mean curvature
		$H=\operatorname{tr}_g A$, with the convention
		\[
		\nabla_i\nu=A_i{}^j e_j.
		\]
		For $h\in C^2(\Sigma)$ sufficiently small in the sense that
		\[
		|A||h|+|\nabla_\Sigma h|\le \delta
		\]
		for some $\delta>0$ sufficiently small, the mean curvature of the normal graph
		\[
		\Sigma_h
		=
		\{x+h(x)\nu(x):x\in\Sigma\}
		\]
		can be written as
		\[
		H_{\Sigma_h}
		=
		H
		+
		\Delta_\Sigma h
		+
		|A|^2 h
		+
		\mathcal Q_\Sigma[h],
		\]
		where $\mathcal Q_\Sigma[h]$ satisfies
		\begin{equation}\label{Q-basic}
			\begin{aligned}
				|\mathcal Q_\Sigma[h]|
				\le {}&
				C
				\Big(
				|A||h|
				+
				|\nabla_\Sigma h|^2
				\Big)
				|\nabla_\Sigma^2h|
				\\
				&+
				C
				\Big(
				|A|^3|h|^2
				+
				|A||\nabla_\Sigma h|^2
				+
				|h||\nabla_\Sigma A||\nabla_\Sigma h|
				\Big).
			\end{aligned}
		\end{equation}
		Here and below all contractions are taken with respect to the metric $g$.
		
		Moreover, let $h_1,h_2\in C^2(\Sigma)$ satisfy \[ |A||h_i|+|\nabla_\Sigma h_i|\le \delta, \qquad i=1,2, \] for $\delta>0$ sufficiently small, and set \[ w=h_1-h_2. \] Then one has the pointwise estimate \begin{equation}\label{Q-difference} \begin{aligned} &|\mathcal Q_\Sigma[h_1]-\mathcal Q_\Sigma[h_2]| \le {} C \Big( |A|(|h_1|+|h_2|) + |\nabla_\Sigma h_1|^2 + |\nabla_\Sigma h_2|^2 \Big) |\nabla_\Sigma^2 w| \\ &+ C|A| \Big( |\nabla_\Sigma^2 h_1| + |\nabla_\Sigma^2 h_2| \Big) |w| + C \Big( |\nabla_\Sigma h_1|
				+ |\nabla_\Sigma h_2| \Big) \Big( |\nabla_\Sigma^2 h_1| + |\nabla_\Sigma^2 h_2| \Big) |\nabla_\Sigma w| \\ &+ C|A|^3 \Big( |h_1|+|h_2| \Big) |w| + C|A| \Big( |\nabla_\Sigma h_1| + |\nabla_\Sigma h_2| \Big) |\nabla_\Sigma w| \\ &+ C|\nabla_\Sigma A| \Big[ \Big( |\nabla_\Sigma h_1| + |\nabla_\Sigma h_2| \Big) |w| + \Big( |h_1|+|h_2| \Big) |\nabla_\Sigma w| \Big]. \end{aligned} \end{equation} In particular, the difference estimate is quadratic in the sense that every term on the right-hand side contains one factor measuring the size of $h_1$ or $h_2$ and one factor measuring the size of $h_1-h_2$.
	\end{propositio}

	The proof follows by differentiating the geometric quantities associated with a
	normal graph perturbation and collecting the resulting first- and second-order
	terms. Since this argument is completely general and independent of the
	specific geometry considered in the present paper, we postpone it to
	Appendix~\ref{appe2}.

    \begin{remark}\label{rem0}
For the nonlinear analysis it is convenient to separate the part of the
quadratic remainder containing second derivatives of the perturbation from the
genuinely lower-order nonlinear terms. Accordingly, we write
\[
\mathcal Q_\Sigma[h]
=
\mathcal Q_1[h,\nabla_\Sigma h]\,
\nabla_\Sigma^2h
+
\mathcal Q_2[h,\nabla_\Sigma h].
\]

Here the decomposition is understood in a symbolic sense. The quantity
\[
\mathcal Q_1[h,\nabla_\Sigma h]\nabla_\Sigma^2 h
\]
stands for a finite linear combination of contractions between the tensor
$\nabla_\Sigma^2 h$ and coefficients depending smoothly on
$h$ and $\nabla_\Sigma h$. The order of the factors is therefore irrelevant.

Estimate \eqref{Q-basic} immediately yields
\[
|\mathcal Q_1|
\lesssim
|A||h|
+
|\nabla_\Sigma h|^2,
\]
and
\[
|\mathcal Q_2|
\lesssim
|A|^3|h|^2
+
|A||\nabla_\Sigma h|^2
+
|h||\nabla_\Sigma A||\nabla_\Sigma h|.
\]

This decomposition will repeatedly be used in the nonlinear estimates, since
the highest-order derivatives appear only linearly.
\end{remark}

    
	We now specialize these general formulas to surfaces of revolution. Their
	special structure allows explicit expressions for both the mean curvature and
	the Jacobi operator, which constitute the basic analytical tools used
	throughout the remainder of the paper.
	
	Let $\Sigma$ be a smooth surface of revolution parametrized by
	\begin{equation}\label{defX}
		x(r,\theta)
		=
		\left(
			Q(r)\cos\theta , 
			Q(r)\sin\theta , 
			P(r)
		\right).
	\end{equation}
	
	A straightforward computation shows that the second fundamental form is given
	by
	\[
	A_{rr}
	=
	\frac{P''Q'-P'Q''}{\sqrt{Q'^2+P'^2}},
	\qquad
	A_{r\theta}=0,
	\qquad
	A_{\theta\theta}
	=
	\frac{QP'}{\sqrt{Q'^2+P'^2}}.
	\]
	
	Consequently, the mean curvature takes the explicit form
	\begin{equation}\label{meancurvature}
		H[P,Q](r)
		=
		\frac{
			P''Q'-P'Q''
		}{
			(Q'^2+P'^2)^{3/2}
		}
		+
		\frac{
			P'
		}{
			Q\sqrt{Q'^2+P'^2}
		}.
	\end{equation}
	
	A detailed derivation of this formula is given in
	Appendix~\ref{app3}. It will be used repeatedly in the sequel, both for the
	construction of the approximate surface and for the analysis of the residual
	error.

	For a function $h:\Sigma\to\mathbb R$, we shall write, with a slight abuse of
	notation,
	\[
	h=h(r,\theta)=h(x(r,\theta)),
	\]
	where $x(r,\theta)$ denotes the parametrization \eqref{defX}. A normal
	perturbation of the surface is therefore represented by
	\begin{equation}\label{def-np}
		\left(
			Q(r)\cos\theta , 
			Q(r)\sin\theta , 
			P(r)
		\right)
		+
		h(r,\theta)\nu(r,\theta).
	\end{equation}
	
	Since the unit normal vector can be expressed explicitly in terms of the
	profile curve, the perturbed surface is again a surface of revolution and may
	be written in the form
	\[
	\left(
		(Q+l)\cos\theta, 
		(Q+l)\sin\theta, 
		P+k
	\right),
	\]
	where
$$
		l=
		\frac{P'h}{\sqrt{Q'^2+P'^2}},
		\qquad
		k=
		-\frac{Q'h}{\sqrt{Q'^2+P'^2}}.
$$
	Thus every sufficiently small normal deformation is completely determined by
	the scalar function $h$, a fact that considerably simplifies the linearization
	of the mean-curvature operator.
	
	Restricting ourselves to axially symmetric perturbations, the Jacobi operator
	takes the explicit form
	\begin{equation}\label{Jacobi}
		\begin{aligned}
			J_\Sigma[h]
			&=
			\frac{1}{Q\sqrt{Q'^2+P'^2}}
			\frac{d}{dr}
			\left(
			\frac{Q}{\sqrt{Q'^2+P'^2}}
			h'(r)
			\right)
			\\
			&\qquad
			+
			\left[
			\frac{(P''Q'-P'Q'')^2}
			{(Q'^2+P'^2)^3}
			+
			\frac{P'^2}
			{Q^2(Q'^2+P'^2)}
			\right]
			h(r).
		\end{aligned}
	\end{equation}
	
	\begin{proof}[Proof of \eqref{Jacobi}]
		Let
		\[
		Y=
		\bigl[D_rx,\,
		D_\theta x\bigr]
		\]
		and let
		\[
		g=Y^TY
		\]
		be the induced metric on $\Sigma$. The Jacobi operator is given by the standard
		formula
		\[
		J_\Sigma[h]
		=
		\frac1{\sqrt{\det g}}
		\partial_j
		\left(
		g^{ij}
		\sqrt{\det g}
		\,\partial_i h
		\right)
		+
		|A|^2h,
		\]
		where $A$ denotes the second fundamental form.
		
		Since $h$ depends only on the radial variable, the angular derivatives vanish
		identically, and the Laplace--Beltrami operator reduces immediately to the
		first term in \eqref{Jacobi}. The potential term follows from the identity
		\[
		|A|^2=\operatorname{tr}(S^2),
		\]
		where the shape operator is
		\[
		S
		=
		-
		g^{-1}
		Y^T
		\bigl[D_r\nu,\,
		D_\theta\nu\bigr].
		\]
		Substituting the explicit expressions for the metric and the second
		fundamental form yields formula \eqref{Jacobi}.
	\end{proof}
	\medskip
	
	\medskip
	
	The following lemma provides a precise estimate for the quadratic term \({\mathcal Q}_\Sigma [h]\) in the expansion of the mean curvature (see Proposition \ref{generalH}, corresponding to normal perturbations of the surface \(\Sigma_0^{\varepsilon,+,1}\), defined as in \eqref{defSigma0ve12}. These estimates rely crucially on the expansions of \(P\) and \(Q\), and the result in Proposition \ref{generalH}.
	
	The estimates degenerate as \(r\to R\). This reflects the degeneration of the coordinate system \((r,\theta)\) near the point where the lower and upper hemispheres meet, rather than any geometric singularity of the sphere itself. For this reason, we restrict the analysis to the region lying just before this transition point.
	
	The estimates for \(\mathcal Q\) take different forms in the various regions of the surface, reflecting the change in its geometry. Near \(r=1\), the surface is very close to a catenoid, and it remains well approximated by a catenoid up to radial distances of order \(\varepsilon^{-1/2}\). Beyond this scale, the surface begins to deform into the lower hemisphere of the sphere of radius \(R\) centred at \((R+d)e_3\). Recall that, for
	\[
	r>2\delta \varepsilon^{-3/4},
	\]
	the surface coincides exactly with the lower hemisphere.
	
	\begin{lemma}\label{expansion-curvature-0}
		Let $\Sigma_0^{\ve, +,1}$ be the surface parametrized as in \eqref{defSigma0ve12}, with $P$ and $Q$ defined as in \eqref{defPQ}. Let $h = h(r)$ be a $C^{2, \alpha}$-small normal perturbation of $\Sigma_0^{\ve, +,1}$ as in \eqref{def-np}. Then: for all $r \in (1,2)$ we have
	$$
		\begin{aligned}
			| {\mathcal Q}_{\Sigma_0^{\ve, +,1}} [h] (r)
			| &\lesssim  (|h| + (r-1) |h'|^2 ) \, (r-1) \, (|h''| + |h'|) \\
			&+ |h|^2 + (r-1) |h'|^2 + (r-1) |h| |h'|;
		\end{aligned}
$$
		for any $r \in [2, \ve^{-\frac12})$ we have
$$
		\begin{aligned}
			| {\mathcal Q}_{\Sigma_0^{\ve, +,1}} [h] (r)| &\lesssim  ( \frac{|h|}{1+ r^2} + |h'|^2 )  \, ( |h''| + \frac{|h'|}{1+r} )\\
			&+ \frac{|h|^2}{1+ r^6} + \frac{|h'|^2}{1+ r^2} + \frac{|h| \, |h'|}{1+ r^3} ;
		\end{aligned}
$$
		and any $r \in (\ve^{-\frac12}, \ve^{-1})$
$$
		\begin{aligned}
			| {\mathcal Q}_{\Sigma_0^{\ve, +,1}} [h] (r)| &\lesssim  (\ve |h| + (1-\ve^2 r^2)^2 |h'|^2) \, ( (1-\ve^2 r^2) |h''| + \ve^\frac12 (1-\ve^2 r^2) |h'|)\\
			&+ \ve^3 |h|^2 + \ve (1-\ve^2 r^2)^2 + \ve^\frac{3}{2} (1-\ve^2 r^2 ) |h| |h'|. 
		\end{aligned}
$$
		Here ${\mathcal Q}_{\Sigma_0^{\ve, +,1}}$ is the quadratic term obtained in Proposition \ref{generalH} for the surface $\Sigma_0^{\ve, +,1}$.
		We also recall that $R=\frac{1}{\ve}$.    
    \end{lemma}

	\begin{proof}
		We use Proposition \ref{generalH} together with the general estimate \eqref{Q-basic}, combined with the computations of $|A|$ and \(|\nabla_\Sigma A|\) for the surface \(\Sigma_0^{\varepsilon,+,1}\).
		The resulting estimates reflect the different geometric regimes of the surface. Near \(r=1\), the surface is very close to a catenoid and remains well approximated by a catenoid up to radial distances of order \(\varepsilon^{-1/2}\). Beyond this scale, the surface gradually deforms into the lower hemisphere of the sphere of radius \(R\) centred at \((R+d)e_3\). Exploiting these asymptotic descriptionsand the computations in Appendice \ref{app3}, we obtain the following estimates
		$$
		|A| \lesssim \left\{ \begin{matrix} \frac{1}{r^2} & \quad r< \ve^{-\frac{1}{2}}\\
			\ve & \quad r> \ve^{-{1\over 2}}\end{matrix}\right., \quad
		|\nabla_{\Sigma_0^{\ve,+,1}} A| \lesssim \left\{ \begin{matrix} \sqrt{r-1} & \quad 1<r<2 \\ {1\over r^3} & \quad 2< r< \ve^{-{1\over 2}}\\
			\ve^{3\over 2} & \quad r> \ve^{-{1\over 2}}\end{matrix}\right. .
		$$
		Besides
		$$
		|\nabla_{\Sigma_0^{\ve,+,1}} h | \lesssim \left\{ \begin{matrix} \sqrt{r-1} |h'| & \quad 1<r<2 \\ |h'| & \quad 2< r< \ve^{-{1\over 2}}\\
			(1-\ve^2 r^2) |h'|  & \quad r> \ve^{-{1\over 2}}\end{matrix}\right. , \quad 
		$$
		and
		$$
		|\nabla_{\Sigma_0^{\ve,+,1}}^2 h | \lesssim \left\{ \begin{matrix} (r-1) ( |h''| + |h'|)  & \quad 1<r<2 \\ |h''| + |{h' \over r}| & \quad 2< r< \ve^{-{1\over 2}}\\
			(1-\ve^2 r^2) \left( |h''| + |{h' \over r}|  \right)  & \quad r> \ve^{-{1\over 2}}\end{matrix}\right. .
		$$
		For the details we refer to Appendix \ref{app3}. With these estimates we get
		\begin{align*}
			&\big(|A||h|+|\nabla_\Sigma h|^2\big)|\nabla_\Sigma^2 h| \\&\lesssim \left\{ \begin{matrix}(|h| + (r-1) |h'|^2 ) \, (r-1) \, (|h''| + |h'|)  & \quad 1<r<2 \\ ( {|h| \over 1+ r^2} + |h'|^2 )  \, ( |h''| + {|h'| \over 1+r} ) & \quad 2< r< \ve^{-{1\over 2}}\\
				(\ve |h| + (1-\ve^2 r^2)^2 |h'|^2) \, ( (1-\ve^2 r^2) |h''| + \ve^{1\over 2} (1-\ve^2 r^2) |h'|)  & \quad r> \ve^{-{1\over 2}}\end{matrix}\right.
		\end{align*}
		and
		\begin{align*}
			&|A|^3 |h|^2 + |A| |\nabla_\Sigma h|^2 + |h| |\nabla_\Sigma A| |\nabla_\Sigma h| 
			\\&\lesssim \left\{ \begin{matrix} |h|^2 + (r-1) |h'|^2 + (r-1) |h| |h'|  & \quad 1<r<2 \\ {|h|^2 \over 1+ r^6} + {|h'|^2 \over 1+ r^2} + {|h| \, |h'| \over 1+ r^3} & \quad 2< r< \ve^{-{1\over 2}}\\
				\ve^3 |h|^2 + \ve (1-\ve^2 r^2)^2 + \ve^{3\over 2} (1-\ve^2 r^2 ) |h| |h'|  & \quad r> \ve^{-{1\over 2}}\end{matrix}\right. .
		\end{align*}
		The conclusion follows from Proposition \ref{generalH} and estimate \eqref{Q-basic}.
	\end{proof}

	We now proceed with the study of the residual of the Euler--Lagrange equation
	\[
	H_{\Sigma_0^\ve}(x)
	+
	m\ve^3
	\int_{\Omega_0^\ve}
	\frac{dy}{|x-y|},
	\qquad
	x\in\Sigma_0^\ve.
	\]
	
	\medskip
	\begin{propositio}\label{prop-error}
		For every \(x\in\Sigma_0^\ve\),
		\begin{equation}\label{comp-H0}
			\begin{aligned}
				H_{\Sigma_0^\ve}(x)
				=
				2\ve
				&+
				O(\ve^{3/2})
				{\bf1}_{(0,\,2\delta\ve^{-3/4})}
				\\
				&+
				O(\ve^2|\log\delta|)
				{\bf1}_{(\delta\ve^{-3/4},\,2\delta\ve^{-3/4})}
				\\
				&+
				2\chi'(F'-G')
				+
				\chi''(F-G)
				+
				\frac{\chi'(F-G)}r.
			\end{aligned}
		\end{equation}
		Moreover,
		\begin{equation}\label{comp-H00}
			2\chi'(F'-G')
			+
			\chi''(F-G)
			+
			\frac{\chi'(F-G)}r
			=
			O\!\left(
			\delta^{-2}|\log\delta|\,\ve^{3/2}
			\right)
			{\bf1}_{(\delta\ve^{-3/4},\,2\delta\ve^{-3/4})}.
		\end{equation}
		
		Furthermore,
		\begin{equation}\label{exp-N0}
			m\ve^3
			\int_{\Omega_0^\ve}
			\frac{dy}{|x-y|}
			=
			\frac{4\pi m}{3|x-P|}
			+
			\frac{4\pi m}{3|x-\bar P|}
			+
			\Theta(x),
		\end{equation}
		where
		\[
		P=(R+d)e_3,
		\qquad
		\bar P=-(R+d)e_3,
		\]
		and
		\[
		|\Theta(x)|
		\lesssim
		\begin{cases}
			m\ve^3|x|\,|\log(\ve|x|)|,
			&
			|x|<5\delta\ve^{-3/4},
			\\[1ex]
			\dfrac{m\delta^2\ve^{3/2}}{|x|},
			&
			|x|\ge5\delta\ve^{-3/4}.
		\end{cases}
		\]
	\end{propositio}
	
	\begin{proof}
		We  rewrite the first term in the mean-curvature formula \eqref{meancurvature} in a form that
		clearly separates the contributions coming from the modified catenoid, the
		sphere, and the transition region. Using \eqref{FG-derivatives}, we obtain
		\begin{equation}\label{first-part}
			\begin{aligned}
				\frac{
					P''Q'-P'Q''
				}{
					(Q'^2+P'^2)^{3/2}
				}
				&=
				\chi
				\frac{
					F''q'-F'q''
				}{
					(Q'^2+P'^2)^{3/2}
				}
				+
				(1-\chi)
				\frac{
					G''
				}{
					(Q'^2+P'^2)^{3/2}
				}
				\\
				&\qquad
				+
				\frac{
					2\chi'(F'-G')
					+\chi''(F-G)
				}{
					(Q'^2+P'^2)^{3/2}
				}
				+
				\frac{\pi_1}{(Q'^2+P'^2)^{3/2}}.
			\end{aligned}
		\end{equation}
		We recall that the definition of $\chi$ is in \eqref{defchi}. To compare the first two terms with the corresponding mean-curvature
		expressions of the modified catenoid and the sphere, we introduce the
		quantities
	$$
			\begin{aligned}
				P'
				&=
				F'+a_1,
				\qquad
				a_1
				:=
				(1-\chi)(G'-F')
				+
				\chi'(F-G),
				\\
				Q'
				&=
				q'+a_2,
				\qquad
				a_2
				:=
				\big((\chi-1)(q-r)\big)'.
			\end{aligned}
$$
		
		With this notation,
		
		\begin{align*}
			\chi
			\frac{F''q'-F'q''}
			{(Q'^2+P'^2)^{3/2}}
			&=
			\chi
			\frac{F''q'-F'q''}
			{(q'^2+F'^2)^{3/2}}
			+
			\chi
			\frac{F''q'-F'q''}
			{(q'^2+F'^2)^{3/2}}
			\left[
			\left(
			1+\dfrac{a_3}{q'^2+F'^2}
			\right)^{-{3/2}}
			-1
			\right],
		\end{align*}
		where
		\[
		a_3
		=
		2a_1F'
		+
		a_1^2
		+
		2a_2q'
		+
		a_2^2.
		\]
		
		Similarly,
		
		\begin{align*}
			\frac{\chi F'}
			{Q\sqrt{Q'^2+P'^2}}
			&=
			\frac{\chi F'}
			{q\sqrt{q'^2+F'^2}}
			+
			\frac{\chi F'}
			{q\sqrt{q'^2+F'^2}}
			\left[
			\left(
			1+\dfrac{(\chi-1)(q-r)}{q}
			\right)
			\left(
			1+\dfrac{a_3}{q'^2+F'^2}
			\right)^{-{1/2}}
			-1
			\right].
		\end{align*}
		
		The correction terms generated by these expansions are supported entirely in
		the transition region. More precisely,
		
		\begin{align*}
			\chi
			\frac{F''q'-F'q''}
			{(Q'^2+P'^2)^{3/2}}
			&=
			\chi
			\frac{F''q'-F'q''}
			{(q'^2+F'^2)^{3/2}}
			+
			\pi_2,
			\\
			\frac{\chi F'}
			{Q\sqrt{Q'^2+P'^2}}
			&=
			\frac{\chi F'}
			{q\sqrt{q'^2+F'^2}}
			+
			\pi_3,
		\end{align*}
		
		where
		
		\begin{equation}\label{pi23}
			\pi_i(r)
			=
			\left\{
			\begin{array}{ll}
				O(\ve^2),
				&
				r\in
				(\delta\ve^{-3/4},
				\,2\delta\ve^{-3/4}),
				\\[1ex]
				0,
				&
				\text{otherwise},
			\end{array}
			\right.
			\qquad
			i=2,3.
		\end{equation}
		
		An entirely analogous computation yields
		
		\begin{align*}
			(1-\chi)
			\frac{G''}
			{(Q'^2+P'^2)^{3/2}}
			&=
			(1-\chi)
			\frac{G''}
			{(1+G'^2)^{3/2}}
			+
			\pi_4,
			\\
			(1-\chi)
			\frac{G'}
			{Q\sqrt{Q'^2+P'^2}}
			&=
			(1-\chi)
			\frac{G'}
			{\sqrt{1+G'^2}}
			+
			\pi_5,
		\end{align*}
		
		where the remainders $\pi_4$ and $\pi_5$ satisfy the same estimate
		\eqref{pi23}.
		
		Combining the previous expansions with \eqref{first-part}, we obtain
		\begin{align*}
			H[P,Q](r)
			&=
			\chi\,H[F,q]
			+
			(1-\chi)\,H[G,r]
			\\
			&\qquad
			+
			\frac{
				2\chi'(F'-G')
				+
				\chi''(F-G)
			}{
				(Q'^2+P'^2)^{3/2}
			}
			+
			\frac{\chi'(F-G)}
			{Q\sqrt{Q'^2+P'^2}}
			+
			\pi_6,
		\end{align*}
		where $\pi_6$ satisfies the same estimate as in \eqref{pi23}.
		
		It remains to simplify the two transition terms. Observe that throughout the
		support of $\chi'$ and $\chi''$ we have
		\[
		Q
		=
		r\bigl(1+O(\ve)\bigr),
		\qquad
		Q'
		=
		1+O(\ve),
		\qquad
		P'
		=
		O(\ve^{1/4}).
		\]
		Consequently,
		\[
		(Q'^2+P'^2)^{1/2}
		=
		1+O(\ve),
		\]
		and therefore
		\[
		(Q'^2+P'^2)^{-3/2}
		=
		1+O(\ve).
		\]
		
		Substituting these estimates into the previous identity gives
		\begin{align*}
			&
			\frac{
				2\chi'(F'-G')
				+
				\chi''(F-G)
			}{
				(Q'^2+P'^2)^{3/2}
			}
			+
			\frac{\chi'(F-G)}
			{Q\sqrt{Q'^2+P'^2}}
			\\
			&\qquad
			=
			2\chi'(F'-G')
			+
			\chi''(F-G)
			+
			\frac{\chi'(F-G)}{r}
			+
			\pi,
		\end{align*}
		where
		\begin{equation}\label{pi-f}
			\pi(r)
			=
			\left\{
			\begin{array}{ll}
				O(\ve^2|\log\delta|),
				&
				r\in
				(\delta\ve^{-3/4},
				\,2\delta\ve^{-3/4}),
				\\[1ex]
				0,
				&
				\text{otherwise}.
			\end{array}
			\right.
		\end{equation}
		
		Hence,
		\begin{align*}
			H[P,Q](r)
			&=
			\chi\,H[F,q]
			+
			(1-\chi)\,H[G,r]
			\\
			&\qquad
			+
			2\chi'(F'-G')
			+
			\chi''(F-G)
			+
			\frac{\chi'(F-G)}{r}
			+
			\pi(r).
		\end{align*}
		
		The spherical contribution is explicit,
		\[
		H[G,r]
		=
		2\ve.
		\]
		
		On the other hand, since
		\[
		H_{\Sigma_{0,h_0}}
		=
		-
		J_{\Sigma_0}[h_0]
		+
		\mathcal Q_{\Sigma_0}[h_0],
		\]
		and
		\[
		-J_{\Sigma_0}[h_0]
		=
		2\ve,
		\]
		we conclude that
		\[
		H[F,q]
		=
		2\ve
		+
		\mathcal Q_{\Sigma_0}[h_0].
		\]
		
		Finally, applying Lemma~\ref{expansion-curvature-0} to the explicit function
		$h_0$, together with the estimates in \eqref{h0}, yields
		\[
		\mathcal Q_{\Sigma_0}[h_0]
		=
		O(\ve^{3/2}),
		\]
		uniformly in the catenoidal region. Therefore,
		\[
		H[F,q]
		=
		2\ve
		+
		O(\ve^{3/2}),
		\]
		and substitution into the previous expansion immediately gives
		\eqref{comp-H0}. Estimate \eqref{comp-H00} follows directly from
		\eqref{pi-f}. This
		completes the proof.
\end{proof}
        
		\begin{proof}[Proof of \eqref{exp-N0}]
			
			Following the notation introduced in \eqref{defSigma0ve}, we decompose the
			domain into its upper and lower halves,
			\[
			\Omega_0^\ve
			=
			\Omega_{0}^{\ve,+}
			\cup
			\Omega_{0}^{\ve,-},
			\qquad
			\Omega_{0}^{\ve,\pm}
			=
			\Omega_{0}^\ve
			\cap
			\{
			(r\cos\theta,r\sin\theta,x_3):
			\pm x_3>0
			\}.
			\]
			Since $\Omega_0^\ve$ is symmetric with respect to reflection across the plane
			$\{x_3=0\}$, for every
			\[
			x=(r\cos\theta,r\sin\theta,z)
			\in
			\Sigma_0^\ve=\partial\Omega_0^\ve
			\]
			we may write
			\begin{align*}
				m\ve^3
				\int_{\Omega_0^\ve}
				\frac{dy}{|x-y|}
				&=
				m\ve^3
				\int_{\Omega_0^{\ve,+}}
				\frac{dy}{|x-y|}
				+
				m\ve^3
				\int_{\Omega_0^{\ve,-}}
				\frac{dy}{|x-y|}
				=
				N(x)
				+
				N(\bar x),
			\end{align*}
			where
			\[
			N(x)
			=
			m\ve^3
			\int_{\Omega_0^{\ve,+}}
			\frac{dy}{|x-y|},
			\qquad
			\bar x
			=
			(r\cos\theta,r\sin\theta,-z).
			\]
			We further decompose
$$
				\begin{aligned}
					N(x)
					&=
					N_1(x)
					+
					N_2(x),
					\quad
					N_1(x)
					=
					m\ve^3
					\int_{B_R(P)}
					\frac{dy}{|x-y|},
				\end{aligned}
$$
			where $B_R(P)$ denotes the ball of radius
			$
			R=\frac1\ve
			$
			centered at
			$
			P=(R+d)e_3.
			$
From Newton's theorem (see, e.g., Theorem 9.7 in \cite{LiebLoss}) we get that 
\[
N_1(x) ={4 \pi \over 3} { m \over |x-P|}
\qquad\text{for all}\ x\in\R^3 \setminus B_R(P).
\]
            For completeness we add a brief proof of this fact here.  
			The function
			\(
			x\longmapsto N_1(x)
			=
			m\ve^3
			\int_{B_R(P)}
			\frac{dy}{|x-y|}
			\)
			is harmonic in
			$\mathbb R^3\setminus B_R(P)$
			and vanishes at infinity.
			To determine it explicitly, let
			\(
			x\in\partial B_R(P).
			\)
			Writing
			\[
			x=(R+d)e_3+R\widetilde x,
			\qquad
			\widetilde x\in\partial B_1(0),
			\]
			we obtain
			\begin{align*}
				m\ve^3
				\int_{B_R(P)}
				\frac{dy}{|x-y|}
				&=
				m\ve^3
				\int_{B_R(0)}
				\frac{dy}{|R\widetilde x-y|}
				=
				m\ve
				\int_{B_1(0)}
				\frac{dz}{|\widetilde x-z|}
				=
				m\ve\,\beta,  \quad 
			\beta
			:=
			\int_{B_1(0)}
			\frac{dz}{|z-e_3|}. \end{align*}
			We claim that			\begin{equation}\label{beta}
				\beta=\frac{4\pi}{3}.
			\end{equation}
			The verification of \eqref{beta} is postponed until the end of the proof.	Since
			\(
			x\longmapsto
			\frac{m\beta}{|x-P|}
			\)
			is harmonic outside $B_R(P)$, decays at infinity, and agrees with
			$N_1$ on the boundary of the ball, the uniqueness of the exterior Dirichlet
			problem immediately yields
$$
				N_1(x)
				=
				\frac{m\beta}{|x-P|},
				\qquad
				x\in\Sigma_0^{\ve,+},
$$			
            as claimed.
It remains to estimate
			\[
			N_2(x)
			=
			m\ve^3
			\int_{C^+}
			\frac{dy}{|x-y|},
			\]
			where the correction region is decomposed as
			\begin{align*}
				C^+
				&=
				C_1^+
				\cup
				C_2^+,
				\\
				C_1^+
				&=
				\{
				(r\cos\theta,r\sin\theta,z):
				0\le r<1,\;
				0<z<G(r)
				\},
				\\
				C_2^+
				&=
				\{
				(r\cos\theta,r\sin\theta,z):
				1\le r<2\delta\ve^{-3/4},\;
				P(r)\le z<G(r)
				\}.
			\end{align*}
			Let
			\[
			R_0=\delta\ve^{-3/4}.
			\]
			We first consider the case
			\[
			|x|\ge5R_0.
			\]
			Then
			\[
			|x|
			\ge
			2|y|,
			\qquad
			y\in C^+,
			\]
			and therefore
			\begin{align*}
				N_2(x)
				&\le
				\frac{m\ve^3}{|x|}
				\int_{C^+}dy
				\lesssim
				\frac{m\ve^3}{|x|}
				\int_1^{2R_0}
				(G(r)-P(r))\,r\,dr.
			\end{align*}
			Using the estimate
			\[
			G(r)-P(r)
			\lesssim
			|\log(r/R_0)|,
			\]
			we conclude that
			
			\[
			N_2(x)
			\lesssim
			\frac{m\ve^3}{|x|}
			\int_1^{2R_0}
			|\log(r/R_0)|\,r\,dr
			\lesssim
			\frac{m\delta^2}{|x|}
			\ve^{3/2}.
			\]
			We next consider the intermediate regime, namely
			\[
			10<|x|<5R_0,
			\qquad
			R_0=\delta\ve^{-3/4}.
			\]
			Writing
$x=(x',x_3),$
			we estimate the contribution of the transition region
			$C_2^+$.
			
			Using the definition of $C_2^+$ and integrating first in the vertical
			direction, we obtain
			\[
			\begin{aligned}
				m\ve^3
				\int_{C_2^+}
				\frac{dy}{|x-y|}
				&=
				m\ve^3
				\int_{|y'|<2R_0}
				dy'
				\int_{P(y')}^{G(y')}
				\frac{dy_3}
				{\bigl(|y'-x'|^2+(y_3-P(x'))^2\bigr)^{1/2}}
				\\
				&=
				m\ve^3
				\int_{|y'|<2R_0}
				dy'
				\int_{\frac{P(y')-P(x')}{|y'-x'|}}
				^{\frac{G(y')-P(x')}{|y'-x'|}}
				\frac{dt}{(1+t^2)^{1/2}} .
			\end{aligned}
			\]
Since
			\[
			1+t
			\le
			2(1+t^2)^{1/2},
			\]
			it follows that
			
			\[
			\begin{aligned}
				m\ve^3
				\int_{C_2^+}
				\frac{dy}{|x-y|}
				\le
				2m\ve^3
				\int_{|y'|<2R_0}
				\log\frac{1+a}{1+b}\,dy',
			\end{aligned}
			\]
			where
			\[
			a=
			\frac{G(y')-P(x')}{|y'-x'|},
			\qquad
			b=
			\frac{P(y')-P(x')}{|y'-x'|}.
			\]
			
			Using the elementary inequality
			
			\[
			\log\frac{1+a}{1+b}
			\le
			\frac{a-b}{1+b},
			\]
			
			we obtain
			
			\[
			\log\frac{1+a}{1+b}
			\le
			\frac{G(y')-P(y')}
			{|y'-x'|+P(y')-P(x')}.
			\]
			For
			\[
			|x'|,\ |y'|\ge10,
			\]
			the asymptotic expansion of the transition profile gives
			\[
			G(y')-P(y')
			\lesssim
			\Bigl|
			\log\!\bigl(
			\delta^{-1}\ve^{3/4}|y'|
			\bigr)
			\Bigr|,
			\]
			and therefore
			\[
			m\ve^3
			\int_{C_2^+}
			\frac{dy}{|x-y|}
			\lesssim
			m\ve^3
			\int_{|y'|<5R_0}
			\frac{
				\bigl|
				\log(
				\delta^{-1}\ve^{3/4}|y'|)
				\bigr|
			}
			{|y'-x'|+P(y')-P(x')}
			\,dy'.
			\]
			
			Furthermore,
			
			\[
			|y'-x'|
			+
			P(y')-P(x')
			\ge
			\gamma|y'-x'|,
			\]
			
			for some universal constant
			$\gamma>0$.
			Hence
			
			\[
			m\ve^3
			\int_{C_2^+}
			\frac{dy}{|x-y|}
			\lesssim
			m\ve^3\,I,
			\quad {\mbox {where}} \quad 
			I
			=
			\int_{|y'|<5R_0}
			\frac{
				\bigl|
				\log(
				\delta^{-1}\ve^{3/4}|y'|)
				\bigr|
			}
			{|y'-x'|}
			\,dy'.
			\]
			
			To estimate this integral we split
			
			\[
			I=I_1+I_2,
			\]
			
			where
			
			\[
			I_1
			=
			\int_{|y'-x'|<|x'|/2}
			\frac{
				\bigl|
				\log(
				\delta^{-1}\ve^{3/4}|y'|)
				\bigr|
			}
			{|y'-x'|}
			\,dy',
			\quad 
			I_2
			=
			\int_{\substack{|y'|<5R_0\\|y'-x'|\ge|x'|/2}}
			\frac{
				\bigl|
				\log(
				\delta^{-1}\ve^{3/4}|y'|)
				\bigr|
			}
			{|y'-x'|}
			\,dy'.
			\]
			We estimate the three terms separately.
			
For the first one,
			\[
			\begin{aligned}
				I_1
				&\le
				\frac{1}{|x'|}
				\left|
				\log\!\left(
				\frac{|x'|}{R_0}
				\right)
				\right|
				\int_{|y'|<|x'|/2}dy'
				\lesssim
				|x'|
				\left|
				\log\!\left(
				\frac{|x'|}{R_0}
				\right)
				\right|.
			\end{aligned}
			\]
			Next,
			\[
			\begin{aligned}
				I_2
				&\le
				\left|
				\log\!\left(
				\frac{|x'|}{R_0}
				\right)
				\right|
				\int_{|x'-y'|<4|x'|}
				\frac{dy'}{|x'-y'|}
				\lesssim
				|x'|
				\left|
				\log\!\left(
				\frac{|x'|}{R_0}
				\right)
				\right|.
			\end{aligned}
			\]
			Finally,
			\[
			\begin{aligned}
				I_3
				&\le
				\int_{3|x'|<|y'|<5R_0}
				\log\!\left(
				\frac{|y'|}{R_0}
				\right)
				\frac{dy'}{|y'|}
				\lesssim
				R_0
				\int_0^5
				|\log t|\,dt.
			\end{aligned}
			\]
			Collecting the three estimates gives
			\[
			m\ve^3
			\int_{C_2^+}
			\frac{dy}{|x-y|}
			\lesssim
			m\ve^3
			|x|
			\left|
			\log
			\frac{|x|}{R_0}
			\right|,
			\qquad
			10<|x|<5R_0.
			\]
			The same estimate clearly remains valid when $|x|<10,$
			and exactly the same argument applies to the contribution of the region
			$C_1^+$.
			We have therefore proved that, for every
			$|x|>2,$
			\[
			N_{\Omega_0^\ve}(x)
			\lesssim
			\left\{
			\begin{aligned}
				&
				m\ve^3
				|x|
				\left|
				\log\frac{|x|}{R_0}
				\right|,
				&&|x|<5R_0,
				\\[1ex]
				&
				\frac{m\delta^2}{|x|}
				\ve^{3/2},
				&&|x|\ge5R_0.
			\end{aligned}
			\right.
			\]
			Consequently,
			\[
			N_{\Omega_0^\ve}(x)
			=
			\frac{m\beta}{|x-P|}
			+
			\frac{m\beta}{|x-\bar P|}
			+
			\Theta(x),
			\]
			where
			\[
			|\Theta(x)|
			\lesssim
			\left\{
			\begin{aligned}
				&
				m\ve^3
				|x|
				\left|
				\log\frac{|x|}{R_0}
				\right|,
				&&|x|<5R_0,
				\\[1ex]
				&
				\frac{m\delta^2\ve^{3/2}}{|x|},
				&&|x|\ge5R_0.
			\end{aligned}
			\right.
			\]
			
			It remains only to establish identity
			\eqref{beta}.
			Using cylindrical coordinates
			$
			y=(\rho\cos\theta,\rho\sin\theta,z),
			$
			with
			$
			\rho^2+z^2\le1,
			$
			we obtain
			$
			dy
			=
			\rho\,d\rho\,d\theta\,dz,
			$
			while
			\[
			|y-e_3|
			=
			\sqrt{\rho^2+(z-1)^2}.
			\]
			Hence
			\[
			\begin{aligned}
				\beta
				&=
				\int_{B_1(0)}
				\frac{dy}{|y-e_3|}
				=
				2\pi
				\int_{-1}^{1}
				\int_{0}^{\sqrt{1-z^2}}
				\frac{\rho}
				{\sqrt{\rho^2+(z-1)^2}}
				\,d\rho\,dz.
			\end{aligned}
			\]
			Since
			\(
			\int
			\frac{\rho}
			{\sqrt{\rho^2+a^2}}
			\,d\rho
			=
			\sqrt{\rho^2+a^2},
			\)
			the inner integral is equal to
			\[
			\int_{0}^{\sqrt{1-z^2}}
				\frac{\rho}
				{\sqrt{\rho^2+(z-1)^2}}
				\,d\rho =
			\sqrt{2(1-z)}
			-(1-z).
			\]
			Therefore,
			\[
			\beta
			=
			2\pi
			\int_{-1}^{1}
			\left(
			\sqrt{2(1-z)}
			-
			(1-z)
			\right)
			dz=\frac{4\pi}{3},
			\]
			which proves \eqref{beta} and completes the proof.
		\end{proof}

		For later use, we decompose the residual as
		\begin{equation}\label{E}
			H_{\Sigma_0^\ve}(x)
			+
			m\ve^3
			\int_{\Omega_0^\ve}
			\frac{dy}{|x-y|}
			=
			2\ve
			+
			E_{01}
			+
			E_{02}
			+
			E_{03},
		\end{equation}
		where
	$$
			E_{01}
			=
			O(\ve^{3/2})
			{\bf1}_{(0,\,2\delta\ve^{-3/4})},
$$
		\begin{equation}\label{def-E02}
			\begin{aligned}
				E_{02}
				={}&
				\chi''(F-G)
				+
				\chi'
				\left(
				2(F-G)'
				+
				\frac{F-G}{r}
				\right)
				+
				O(\ve^2|\log\delta|)
				{\bf1}_{(\delta\ve^{-3/4},\,2\delta\ve^{-3/4})},
			\end{aligned}
		\end{equation}
		and
		$$
			E_{03}
			=
			\frac{4\pi m}{3|x-P|}
			+
			\frac{4\pi m}{3|x-\bar P|}
			+
			\Theta(x).
		$$
		More precisely, the error $E_{01}$ retains the spatial information
        \begin{equation}\label{E01-pointwise}
			|E_{01}(r)|
			\leq
			C\ve^3 r^2,
			\qquad
			1<r<2\delta\ve^{-3/4},
		\end{equation}
		up to terms of strictly smaller order. In particular,
		\[
		\|E_{01}\|_{L^\infty}
		\leq
		C_\delta\ve^{3/2}.
		\]
		The rough $L^\infty$ estimate will be sufficient for the neck
		fixed-point argument, whereas the refined estimate
		we will get later in Lemma \ref{E01-refined} 
        will be used in the final projection onto the
		spherical Jacobi field.
		
		\medskip
		The three pieces play different roles. The term \(E_{01}\) is a
		localized higher-order error. The term \(E_{02}\) is the principal
		transition error generated by differentiating the interpolation
		cutoff, and will be removed by the explicit correction \(h_C^0\).
		Finally, \(E_{03}\) contains the leading Coulomb interaction between
		the two spherical components and is responsible, through the final
		solvability condition, for determining the parameter \(m\).

		\section{The gluing scheme and its reduction}
		\label{scheme}
		
		We now formulate the nonlinear problem as a coupled inner--outer
		system on the catenoidal and spherical regions of the approximate
		surface.
		
		Recall that
		\[
		\Sigma_0^\ve
		=
		\Sigma_0^{\ve,+}
		\cup
		\Sigma_0^{\ve,-},
		\]
		and, on the upper half,
		\[
		\Sigma_0^{\ve,+}
		=
		\Sigma_0^{\ve,+,1}
		\cup
		\Sigma_0^{\ve,+,2},
		\]
		where \(\Sigma_0^{\ve,+,1}\) denotes the modified catenoidal and
		transition region and
		\[
		\Sigma_0^{\ve,+,2}
		=
		S_R^+
		\]
		is the upper hemisphere of $S_R$.
		
		Given a sufficiently small axially symmetric function
		\[
		h:\Sigma_0^\ve\longrightarrow\R,
		\]
		even with respect to \(x_3\), we consider the normal graph
		\[
		\Sigma_{0,h}^\ve
		=
		\{ \hat x= 
		x+h(x)\nu_{\Sigma_0^\ve}(x):
		x\in\Sigma_0^\ve
		\}.
		\]
		It is enough to solve the problem on \(\Sigma_0^{\ve,+}\) and then
		extend the perturbation evenly to the lower half.
		
		We use the notation
		\begin{equation}\label{def-Newtonian}
			N_\Omega(x)
			:=
			m\ve^3
			\int_\Omega
			\frac{dy}{|x-y|}.
		\end{equation}
		The Euler--Lagrange equation we wish to solve is therefore
		\begin{equation}\label{EL-compact}
			H_{\Sigma_{0,h}^\ve}(\hat x)
			+
			N_{\Omega_{0,h}^\ve}(\hat x)
			=
			\lambda_\ve[h],
			\qquad
			\hat x\in\Sigma_{0,h}^\ve.
		\end{equation}
We normalize the constant \(\lambda_\ve[h]\) by evaluating the equation
		at the north pole $\widehat N = (2R+d) e_3$:
		\begin{equation}\label{deflaep}
			\lambda_\ve[h]
			=
			H_{\Sigma_{0,h}^\ve}(\widehat N)
			+
			N_{\Omega_{0,h}^\ve}(\widehat N).
		\end{equation}
		We set
		\begin{equation}\label{barlaep}
			\bar\lambda_\ve[h]
			=
			\lambda_\ve[h]-2\ve.
		\end{equation}
		Notice that \(\bar\lambda_\ve[h]\) is spatially constant, but depends
		on the perturbation \(h\) and on the parameter \(m\). For simplicity, its dependence on $m$ is not reflected in the notation. Its dependence on $h$ is non-local and, in particular, its dependence on the neck correction enters through the Coulomb potential. We shall keep this dependence inside the neck
		fixed-point map.
		
		\subsection{Cutoffs and decomposition of the perturbation}
		
		We introduce the neck cutoff
		\begin{equation}\label{defchiC}
			\chi_C(x)
			=
			\begin{cases}
				\chi_0(2 \ve^{1-b}r),
				&
				x\in\Sigma_0^{\ve,+,1},
				\\
				0,
				&
				x\in\Sigma_0^{\ve,+,2},
			\end{cases}
		\end{equation}
		where
		\[
		0<b<\frac14.
		\]
		Thus
		\[
		\chi_C=1
		\quad\text{for}\quad
		r\le\ve^{-1+b},
		\]
		and \(\chi_C\) decreases to zero on the scale
		$r\sim\ve^{-1+b}$.
		In particular,
		\begin{equation}\label{chiC-derivatives}
			|\nabla\chi_C|
			\lesssim
			\ve^{1-b},
			\qquad
			|D^2\chi_C|
			\lesssim
			\ve^{2(1-b)}.
		\end{equation}
		We also introduce a spherical cutoff \(\chi_{S_\ve}\) satisfying
		\[
		\chi_{S_\ve}=0
		\quad\text{for}\quad
		r<\frac\delta4\ve^{-3/4},
\quad 
		\chi_{S_\ve}=1
		\quad\text{for}\quad
		r>2\delta\ve^{-3/4}
		\]
		on \(\Sigma_0^{\ve,+,1}\), and
		\[
		\chi_{S_\ve}=1
		\quad\text{on}\quad
		\Sigma_0^{\ve,+,2}.
		\]
		Its derivatives satisfy
		\begin{equation}\label{chiS-derivatives}
			|\nabla\chi_{S_\ve}|
			\lesssim
			\delta^{-1}\ve^{3/4},
			\qquad
			|D^2\chi_{S_\ve}|
			\lesssim
			\delta^{-2}\ve^{3/2}.
		\end{equation}
		We denote
		\[
		S_\ve
		=
		\operatorname{supp}\chi_{S_\ve}.
		\]
		We seek a perturbation of the form
		\begin{equation}\label{defh}
			h
			=
			\chi_C h_C
			+
			\chi_{S_\ve}h_{S_\ve}.
		\end{equation}
		Later, we further decompose
		$$
			h_C
			=
			h_C^0+h_C^1,
	$$
		where \(h_C^0\) is the explicit first correction cancelling the
		transition error \(E_{02}\), while \(h_C^1\) is the remaining neck
		unknown.
		
		\subsection{Perturbative form of the equation}
		
		By the mean-curvature expansion,
		\[
		H_{\Sigma_{0,h}^\ve}
		(\hat x) =
		H_{\Sigma_0^\ve} (x)
		-
		J_{\Sigma_0^\ve}[h] (x)
		+
		\mathcal Q_{\Sigma_0^\ve}[h] (x), \quad \hat x = x+h(x)\nu_{\Sigma_0^\ve}(x), \quad 
		x\in\Sigma_0^\ve .
		\]
		Defining \(E_0\) by
		\[
		E_0(x)
		=
		H_{\Sigma_0^\ve}(x)
		+
		N_{\Omega_0^\ve}(x),
		\]
		equation \eqref{EL-compact} is equivalent to
	$$
			\begin{aligned}
				(E_0 (x) -2\ve)
				-
				J_{\Sigma_0^\ve}[h] (x)
				+
				\mathcal Q_{\Sigma_0^\ve}[h](x)
				+
				\Big(
				N_{\Omega_{0,h}^\ve} (\hat x)
				-
				N_{\Omega_0^\ve}
				(x) \Big)
				=
				\bar\lambda_\ve[h].
			\end{aligned}
$$
		Equivalently,
		\begin{equation}\label{ecuacion-rearranged}
			J_{\Sigma_0^\ve}[h](x)
			=
			(E_0 (x)-2\ve)
			+
			\mathcal Q_{\Sigma_0^\ve}[h] (x)
			+
			\Big(
			N_{\Omega_{0,h}^\ve}
			(\hat x) -
			N_{\Omega_0^\ve}
			(x) \Big)
			-
			\bar\lambda_\ve[h].
		\end{equation}
		Using \eqref{E},
		\[
		E_0-2\ve
		=
		E_{01}+E_{02}+E_{03}.
		\]
		We next separate the part of the quadratic mean-curvature remainder
		containing second derivatives of the unknown. By
		Proposition \ref{generalH} and Remark \ref{rem0},
		\[
		\mathcal Q_{\Sigma_0^\ve}[h]
		=
		\mathcal Q_1[h,\nabla h] \, D^2h
		+
		\mathcal Q_2[h,\nabla h].
		\]
		Since
		\[
		h=\chi_Ch_C+\chi_{S_\ve}h_{S_\ve},
		\]
		we have
		\begin{align*}
			D^2h
			={}&
			\chi_C D^2h_C
			+
			\chi_{S_\ve}D^2h_{S_\ve}
			+
			2\,\operatorname{sym}
			(\nabla\chi_C\otimes\nabla h_C)
			+
			h_C D^2\chi_C
			\nonumber\\
			&+
			2\,\operatorname{sym}
			(\nabla\chi_{S_\ve}\otimes\nabla h_{S_\ve})
			+
			h_{S_\ve}D^2\chi_{S_\ve}.
		\end{align*}
		Here, for tangent vector fields $X,Y$, $X\otimes Y$ denotes their
		tensor product,
		\[
		(X\otimes Y)_{ij}=X_iY_j,
		\quad 
		\operatorname{sym}(X\otimes Y)
		=
		\frac12\bigl(X\otimes Y+Y\otimes X\bigr).
		\]
		We therefore define the lower-order nonlinear operator
		\begin{equation}\label{calN}
			\begin{aligned}
				&\mathcal N_\ve[h]
				:={}
				\mathcal Q_2[h,\nabla h]
				+
				\Big(
				N_{\Omega_{0,h}^\ve} (\hat x) 
				-
				N_{\Omega_0^\ve}
				(x) \Big) 
				\\
				&+
				\mathcal Q_1[h,\nabla h]
				\Big[
				2\,\operatorname{sym}
				(\nabla\chi_C\otimes\nabla h_C)
				+
				h_CD^2\chi_C
				+
				2\,\operatorname{sym}
				(\nabla\chi_{S_\ve}\otimes\nabla h_{S_\ve})
				+
				h_{S_\ve}D^2\chi_{S_\ve}
				\Big].
			\end{aligned}
		\end{equation}
        By construction, $\mathcal N_\ve$ collects all lower-order
perturbative terms. In particular, it contains no second derivatives
of $h_C$ or $h_{S_\ve}$, and also includes the lower-order linear
contribution \eqref{transport-N} induced by transporting the Coulomb
potential under the identification of $S_\ve$ with $S_R$. With this notation,
		\begin{equation}\label{Q-decomp-exact}
			\mathcal Q_{\Sigma_0^\ve}[h]
			=
			\Big(
			\chi_C D^2h_C
			+
			\chi_{S_\ve}D^2h_{S_\ve}
			\Big)
			\mathcal Q_1[h,\nabla h]
			+
			\mathcal N_\ve[h]
			-
			\Big(
			N_{\Omega_{0,h}^\ve}
			(\hat x) -
			N_{\Omega_0^\ve}
			(x) \Big).
		\end{equation}
		
		\subsection{The inner--outer system}
		
		Using
		\[
		J[\chi h]
		=
		\chi J[h]
		+
		2\nabla\chi\cdot\nabla h
		+
		h\Delta\chi,
		\]
		equation \eqref{ecuacion-rearranged} can be split into a neck equation
		and a spherical equation.
		
		A sufficient coupled system is
		\begin{equation}\label{C}
			\begin{aligned}
				J_{\Sigma_0^\ve}[h_C]
				={}&
				E_{01}
				+
				E_{02}
				+
				E_{03}
				+
				D^2_{\Sigma_0^\ve}h_C\,
				\mathcal Q_1[h,\nabla h]
				+
				\mathcal N_\ve[h]
				-
				\bar\lambda_\ve[h]
				\\
				&-
				2\nabla h_{S_\ve}\cdot\nabla\chi_{S_\ve}
				-
				h_{S_\ve}\Delta\chi_{S_\ve},
				\quad
				x\in\Sigma_0^{\ve,+,1},
				\qquad
				r<\ve^{-1+b},
			\end{aligned}
		\end{equation}
		and
		\begin{equation}\label{Sve}
			\begin{aligned}
				J_{\Sigma_0^\ve}[h_{S_\ve}]
				={}&
				D^2_{\Sigma_0^\ve}h_{S_\ve}\,
				\mathcal Q_1[h,\nabla h]
				+
				(1-\chi_C)
				\Big(
				E_{03}
				+
				\mathcal N_\ve[h]
				-
				\bar\lambda_\ve[h]
				\Big)
				\\
				&-
				2\nabla h_C\cdot\nabla\chi_C
				-
				h_C\Delta\chi_C,
				\qquad
				x\in S_\ve.
			\end{aligned}
		\end{equation}
		If \((h_C,h_{S_\ve})\) solves \eqref{C}--\eqref{Sve}, then
		\[
		h
		=
		\chi_Ch_C+\chi_{S_\ve}h_{S_\ve}
		\]
		solves \eqref{ecuacion-rearranged}. We refer to
		\eqref{C}--\eqref{Sve} as the \emph{gluing system}.
		
		The two equations are not independent, but they are only weakly coupled. The first is solved using the
		inverse of the Jacobi operator on the neck. The second is transported
		to the model sphere \(S_R\) and solved modulo the one-dimensional
		vertical translation mode.

\medskip
        
        \noindent
		{\bf Identification of \(S_\ve\) with \(S_R\) minus a small cap}. \ \ 
		For \(x\in S_\ve\) we define
		\(\widetilde x=\Phi(x)\in S_R\) by
		\begin{equation}\label{tttx}
			\widetilde x
			=
			\begin{cases}
				(r\cos\theta,r\sin\theta,G(r)),
				&
				x=(Q(r)\cos\theta,Q(r)\sin\theta,P(r)),
				\\[1mm]
				x,
				&
				x\in\Sigma_0^{\ve,+,2}.
			\end{cases}
		\end{equation}
		Thus \(\Phi\) preserves the \((r,\theta)\) coordinates.
		Let
		\[
		X(r,\theta)
		=
		(Q(r)\cos\theta,Q(r)\sin\theta,P(r)),
		\]
		and
		\[
		\widetilde X(r,\theta)
		=
		(r\cos\theta,r\sin\theta,G(r)).
		\]
		Then
		\[
		d\Phi(X_r)=\widetilde X_r,
		\qquad
		d\Phi(X_\theta)=\widetilde X_\theta.
		\]
		The estimates obtained from \(P,Q\) give
		$$
			|d\Phi-I|
			\le
			C_\delta\ve^{3/4},
$$
		and the difference is supported in
		\begin{equation}\label{Phi-support}
			\frac\delta4\ve^{-3/4}
			<
			r
			<
			2\delta\ve^{-3/4}.
		\end{equation}
		Differentiating once more yields
$$
			|D(d\Phi)|
			\le
			C_\delta\ve^{3/2}.
$$	
		We now write
		\[
		h_{S_\ve}(x)
		=
		h_{S_R}(\widetilde x),
		\qquad
		\widetilde x=\Phi(x).
		\]
		After identifying tangent spaces by the common \((r,\theta)\)
		coordinates, the chain rule gives
	$$
			\begin{aligned}
				\nabla_{S_\ve}h_{S_\ve}(x)
				&=
				(I+\widetilde B_\ve(x))
				\nabla_{S_R}h_{S_R}(\widetilde x),
				\\
				D^2_{S_\ve}h_{S_\ve}(x)
				&=
				D^2_{S_R}h_{S_R}(\widetilde x)
				+
				\mathcal B_\ve(x)
				\,
				D^2_{S_R}h_{S_R}(\widetilde x)
				+
				C_\ve(x)
				\,
				\nabla_{S_R}h_{S_R}(\widetilde x),
			\end{aligned}
	$$
		where
		$$
			|\widetilde B_\ve|
			+
			|\mathcal B_\ve|
			\le
			C_\delta\ve^{3/4},
			\qquad
			|C_\ve|
			\le
			C_\delta\ve^{3/2},
		$$
		and all these tensors vanish outside the region
		\eqref{Phi-support}.
		
		The corresponding Jacobi operator can therefore be written as
		\begin{equation}\label{Jacobi-identification}
			J_{\Sigma_0^\ve}[h_{S_\ve}](x)
			=
			J_{S_R}[h_{S_R}](\widetilde x)
			+
			B_{S_R}[h_{S_R}](\widetilde x),
		\end{equation}
		where \(B_{S_R}\) is a lower-order perturbation whose coefficients are
		supported in \eqref{Phi-support} and satisfy the estimates established
		in the spherical linear theory. 
The same identification also affects the evaluation of the Coulomb
potential. Formally, transporting the potential under $\Phi$
produces an additional zeroth-order contribution of the form

\begin{equation}\label{transport-N}
DN(x)[h\nu]
=
\partial_\nu N(x)\,h,
\qquad
\partial_\nu N(x)=O(m\varepsilon^2).
\end{equation}
This lower-order contribution will be incorporated together with the
remaining perturbative terms in the spherical equation.

\medskip \noindent
        {\bf Removal of the leading neck error.}\ \ 
		In Section \ref{linear-neck} we will construct the correction \(h_C^0\) such that \(	J_{\Sigma_0^\ve}[h_C^0]-E_{02}\) is of lower order. We set
		\[
		h_C=h_C^0+h_C^1.
		\]
		Then the neck equation \eqref{C} becomes
		\begin{equation}\label{C1}
			\begin{aligned}
				J_{\Sigma_0^\ve}[h_C^1]
				={}&
				E_{01}
				+
				E_{03}
				+
				D^2_{\Sigma_0^\ve}h_C\,
				\mathcal Q_1[h,\nabla h]
				+
				\mathcal N_\ve[h]
				-
				\bar\lambda_\ve[h]
				\\
				&+
				\left( J_{\Sigma_0^\ve}[h_C^0]-E_{02} \right)
				-
				2\nabla h_{S_\ve}\cdot\nabla\chi_{S_\ve}
				-
				h_{S_\ve}\Delta\chi_{S_\ve}.
			\end{aligned}
		\end{equation}
For fixed \(m\) and \(h_{S_R}\), equation \eqref{C1} will be solved by
		a contraction argument in the neck norm. The important point is that
		\(\bar\lambda_\ve[h]\) is kept inside this fixed-point map. Although it
		is constant in space, it depends on \(m\), on \(h_{S_R}\), and, through
		the Coulomb potential, on \(h_C^1\) itself.

        \medskip
        
		\noindent
        {\bf The spherical equation.}\ \ 
		Using \eqref{Jacobi-identification}, the spherical equation \eqref{Sve} can be
		written as
		$$
			J_{S_R}[h_{S_R}]
			+
			B_{S_R}[h_{S_R}]
			=
			H[m,h_{S_R},h_C^1]
			\qquad
			\text{on }S_R,
	$$
		where
		\begin{equation}\label{Hdef}
			\begin{aligned}
				H[m,h_{S_R},h_C^1]
				={}&
				(1-\chi_C)
				\Big(
				E_{03}
				+
				\mathcal N_\ve[h]
				-
				\bar\lambda_\ve[h]
				\Big)
				+
				E_{04}
				\\
				&+
				\Big[
				D^2_{S_R}h_{S_R}
				+
				\mathcal B_\ve \, D^2_{S_R}h_{S_R}
				+
				C_\ve \, \nabla_{S_R}h_{S_R}
				\Big]
				\mathcal Q_1[h,\nabla h]
				\\
				&-
				2\nabla_{S_R}h_C^1\cdot\nabla_{S_R}\chi_C
				-
				h_C^1\Delta_{S_R}\chi_C,
			\end{aligned}
		\end{equation}
		and
		\begin{equation}\label{def-E04}
			E_{04}
			=
			-2\nabla_{S_R}h_C^0\cdot\nabla_{S_R}\chi_C
			-
			h_C^0\Delta_{S_R}\chi_C.
		\end{equation}
		
		The cutoff \(\chi_C\) varies only in a region where the approximate
		surface already coincides exactly with \(S_R\). Therefore all
		derivatives in \eqref{def-E04} and in the corresponding terms involving
		\(h_C^1\) may be taken intrinsically on \(S_R\).

        \medskip
		\noindent
        {\bf Projected spherical problem and the parameter $m$.}\ \ 
		In the axially symmetric class, the Jacobi operator on the sphere has a
		one-dimensional kernel generated by the vertical translation mode
		\[
		Z(\widetilde x)
		=
		\frac{\widetilde x_3-(R+d)}R.
		\]
		We therefore solve the projected problem
		\begin{equation}\label{Sve2}
			\begin{aligned}
				J_{S_R}[h_{S_R}]
				+
				B_{S_R}[h_{S_R}]
				=
				H[m,h_{S_R},h_C^1]
				+
				{\rm c}\,
				\frac{\widetilde x_3-(R+d)}{R^2},
			\end{aligned}
		\end{equation}
        together with the normalization used in the spherical linear theory. Here both $h_{S_R}$ and the constant ${\rm c}$ are to be determined.
		
		For every fixed \(m\) in the range
		\[
		A\ve\le m\le A^{-1}\ve,
		\]
		the strategy is as follows. Given \(h_{S_R}\), solve the neck equation
		\eqref{C1} to obtain
		\[
		h_C^1=h_C^1[m,h_{S_R}].
		\]
		Substitute this function into \eqref{Sve2} and solve the projected
		spherical equation for
		\[
		h_{S_R}=h_{S_R}[m]
		\]
		and the scalar projection coefficient
		\[
		{\rm c}={\rm c}[m].
		\]
		The remaining scalar condition is
		$$
			{\rm c}[m]=0.
$$
		It is this condition that determines the parameter \(m=m_\ve\).

        Let us give an intuition why the range $A\ve\le m\le A^{-1}\ve$ is the relevant one. At leading order, the projection onto the vertical Jacobi field is generated by the two terms
		\[
		E_{03}
		\qquad\text{and}\qquad
		E_{04}.
		\]
		The calculations carried out below give
	$$
			\int_S
			E_{03}
			\bigl(
			R\omega+(R+d)e_3
			\bigr)
			\omega_3\,d\sigma
			=
			-\frac{4\pi^2}{9}m\ve
			+
			O(m\ve^2|\log\ve|),
	$$
		and
$$
			\int_S
			E_{04}
			\bigl(
			R\omega+(R+d)e_3
			\bigr)
			\omega_3\,d\sigma
			=
			2\pi\ve^2
			\Big(
			1+O(\delta^4)+o_\ve(1)
			\Big).
$$
		Consequently, the leading-order balance is
		\[
		-\frac{4\pi^2}{9}m\ve
		+
		2\pi\ve^2
		+
		O(\delta^4\ve^2)
		+
		o(\ve^2)
		=
		0,
		\]
		which gives
		$$
			m
			=
			\frac{9}{2\pi}\ve
			\Big(
			1+O(\delta^4)+o_\ve(1)
			\Big).
	$$
		The preceding computation is only used here to identify the natural
		range of the parameter \(m\). The exact equation corresponding to
		\({\rm c}=0\) also contains the projection of the perturbation
		\(B_{S_R}[h_{S_R}]\); this lower-order contribution will be retained in
		the final balancing argument.
		
		Finally, the estimates for \(E_{03}\) and \(E_{04}\) show that the
		natural size of the spherical correction is
		\[
		\|h_{S_R}\|_{**}
		=
		O\left(
		\ve^{1-b\beta-2b}|\log\ve|
		\right).
		\]
		The norm $\| \cdot \|_{**}$ is introduced later in \eqref{norm-h-sphere}. Thus, in the nonlinear reduction, we work with
$$
			A\ve
			\le
			m
			\le
			A^{-1}\ve
$$
		and
        \begin{equation}\label{con2-preview}
			\|h_{S_R}\|_{**}
			\le
			A^{-1}
			\ve^{1-b\beta-2b}
			|\log\ve|.
		\end{equation}
		
		The rest of the proof consists of making this reduction rigorous:
		first solving \eqref{C1} for \(h_C^1\) with \(m,h_{S_R}\) fixed,
		then solving the projected spherical problem, and finally choosing
		\(m\) so that \({\rm c}[m]=0\).
		
		\section{The linear problem in the neck and the first correction}
		\label{linear-neck}
		
		We now develop the linear theory required to solve the neck equation
		\eqref{C1}. The Jacobi operator on the neck interpolates between the
		Jacobi operator of the catenoid near the waist and that of the large
		sphere near the outer end. Our first goal is to construct a right inverse
		that is uniform in $\ve$ in a weighted norm adapted to the two ends.
		
		Later, we shall need more precise information than the basic weighted
		estimate alone. In particular, a spatially constant right hand side
		produces a solution growing quadratically in $r$, whereas a right hand
		side supported in the catenoid--sphere transition region produces only
		logarithmic growth once this support has been crossed. We therefore
		establish both estimates in this section.
		
		\subsection{The Jacobi operator on the neck}
		
		Recall that
		\[
		\Sigma_0^{\ve,+,1}
		=
		\left\{
		(Q(r)\cos\theta,Q(r)\sin\theta,P(r)):
		1<r<R
		\right\},
		\]
		where $P,Q$ are defined in \eqref{defPQ}. For an axially symmetric
		function $h=h(r)$, the Jacobi operator $J_{\Sigma_0^{\ve,+,1}}[h]$ takes the form
		$$
			J(P,Q)[h]
			=
			\frac1{Q'^2+P'^2}
			\left(
			h''+a(r)h'+b(r)h
			\right),
		$$
		where
		$$
			a(r)
			=
			\frac{Q'}{Q}
			-
			\frac{
				Q'(Q'Q''+P'P'')
			}{
				Q'^2+P'^2
			}, \quad 
			b(r)
			=
			\frac{
				(P''Q'-P'Q'')^2
			}{
				(Q'^2+P'^2)^2
			}
			+
			\frac{P'^2}{Q^2}.
	$$
		In the pure catenoidal region, corresponding to $Q=r$ and $P=F_0$,
		we recover
	$$
			J_{\Sigma_0}[h]
			=
			\frac{r^2-1}{r^2}h''
			+
			\frac1r h'
			+
			\frac2{r^4}h.
		$$
		See Appendix \ref{app3}, Subsection \ref{cat}. On the other hand, when $Q=r$ and $P=G$, we obtain the Jacobi
		operator on the lower hemisphere of the sphere of radius $R$:
		\begin{equation}\label{defJR}
			J_{S_R}[h]
			=
			\frac{R^2-r^2}{R^2}h''
			+
			\frac{R^2-2r^2}{R^2r}h'
			+
			\frac2{R^2}h.
		\end{equation}
		
		We solve
		\begin{equation}\label{prob0}
			\begin{aligned}
				J(P,Q)[h](r)
				&=
				g(r),
				\qquad
				1<r<\ve^{-1+b},
				\\
				h(1)&=0,
				\quad 
				\lim_{r\to1^+}\sqrt{r^2-1}\,h'(r) =0.
			\end{aligned}
		\end{equation}
		As before, the two conditions at $r=1$ guarantee that the corresponding
		normal graph extends evenly and smoothly across the waist.
		
		We use the norm
		\begin{equation}\label{norma*}
			\begin{aligned}
				\|h\|_*
				:={}&
				\left\|
				\frac{h}{r(r-1)}
				\right\|_{L^\infty(1,\ve^{-1+b})}
				+
				\left\|
				r^{-1}h'
				\right\|_{L^\infty(1,\ve^{-1+b})}
				+
				\left\|
				\min\{r-1,1\}\,h''
				\right\|_{L^\infty(1,\ve^{-1+b})}.
			\end{aligned}
		\end{equation}
		In particular, for $r\ge2$,
		$$
			|h(r)|
			\le
			r^2\|h\|_*,
			\qquad
			|h'(r)|
			\le
			r\|h\|_*,
			\qquad
			|h''(r)|
			\le
			\|h\|_*.
	$$
		
		The following is the basic linear estimate.
		
		\begin{lemma}\label{lemma-linear-neck}
			Let $0<b<\frac14$
			be sufficiently small. There exist constants $\ve_0>0$ and
			$\mathcal C>0$ such that, for every $0<\ve<\ve_0$ and every
$g\in L^\infty(1,\ve^{-1+b})$
			problem \eqref{prob0} has a unique solution
			\[
			h=\mathcal T_C[g]
			\]
			in the class $\|h\|_*<\infty$. The operator $\mathcal T_C$ is linear
			and satisfies
			\begin{equation}\label{esti-neck}
				\|\mathcal T_C[g]\|_*
				\le
				\mathcal C
				\|g\|_{L^\infty(1,\ve^{-1+b})}.
			\end{equation}
		\end{lemma}
		
		\begin{proof}
			Multiplying \eqref{prob0} by $Q'^2+P'^2$, we write the equation as
			\begin{equation}\label{prob1}
				\bar J(Q,P)[h]
				=
				(Q'^2+P'^2)g,
			\end{equation}
			where
			\[
			\bar J(Q,P)[h]
			=
			h''+ah'+bh.
			\]
			
			Let $a_0,b_0$ denote the coefficients of the catenoidal operator, and
			$a_{0,R},b_{0,R}$ those of the spherical operator. From the estimates
			for $P$ and $Q$ established in the construction of the approximate
			surface, one has
			\begin{equation}\label{a-exp}
				a(r)
				=
				\begin{cases}
					a_0(r)
					+
					\dfrac{O(\ve)+O(\ve^2r^2)}r,
					&
					1<r<\ve^{-1/2},
					\\[3mm]
					a_{0,R}(r)
					+
					\dfrac{O(\ve^2r^2)}r\,\chi(r),
					&
					\ve^{-1/2}<r<\ve^{-1+b},
				\end{cases}
			\end{equation}
			and
			\begin{equation}\label{b-exp}
				b(r)
				=
				\begin{cases}
					b_0(r)
					+
					\dfrac{O(\ve)+O(\ve^2r^2)}{r^2},
					&
					1<r<\ve^{-1/2},
					\\[3mm]
					b_{0,R}(r)
					+
					\dfrac{O(\ve^2r^2)}{r^2}\,\chi(r),
					&
					\ve^{-1/2}<r<\ve^{-1+b}.
				\end{cases}
			\end{equation}
            Note that in these estimates a new length scale $r \sim \ve^{-1/2}$ emerges. This length scale will be relevant in the remainder of this proof and elsewhere in this section.
                        
			We introduce an approximate positive Jacobi field that interpolates
			between a Jacobi field of the catenoid and the vertical Jacobi field of
			the sphere:
			$$
				z_1(r)
				=
				\chi_0(\ve^{1/2}r)y_1(r)
				+
				\bigl(1-\chi_0(\ve^{1/2}r)\bigr)y_R(r),
		$$
			where
			\[
			y_1(r)
			=
			\frac{\sqrt{r^2-1}}r,
			\qquad
			y_R(r)
			=
			\sqrt{1-\frac{r^2}{R^2}},
			\]
            and $\chi_0$ is given in \eqref{chi0}.
			Thus
			\[
			J_{\Sigma_0}[y_1]=0,
			\qquad
			J_{S_R}[y_R]=0.
			\]
			Moreover, the two Jacobi fields have the same leading size in the
			intermediate region $r\sim\ve^{-1/2}$.
			We write
			\[
			h=\psi z_1.
			\]
			Substitution into \eqref{prob1} gives
			\begin{equation}\label{eq1-neck}
				\psi''
				+
				\left(
				2\frac{z_1'}{z_1}+a
				\right)\psi'
				+
				\frac{\bar J(Q,P)[z_1]}{z_1}\psi
				=
				\frac{Q'^2+P'^2}{z_1}g.
			\end{equation}
			
			For the construction of an approximate inverse, we first omit the
			zero-order defect involving $\bar J(Q,P)[z_1]$. Write
			$$
				a=a_0+a_1,
			$$
			where
			\[
			a_0(r)
			=
			\frac{\sqrt{1+(F_0')^2}}r
			\frac{d}{dr}
			\left(
			\frac{r}{\sqrt{1+(F_0')^2}}
			\right)
			=
			\frac{r}{r^2-1}.
			\]
			For $r\ge2$, \eqref{a-exp} gives
			\begin{equation}\label{a1-large}
				a_1(r)
				=
				\frac{O(\ve+\ve^2r^2)}r
				=
				\frac{O(\ve^{2b})}{r},
				\qquad
				2\le r<\ve^{-1+b}.
			\end{equation}
			
			The integrating factor for the equation obtained from
			\eqref{eq1-neck} by dropping the zero-order defect is
			$$
				\mu_1(r)
				=
				z_1^2(r)
				\frac{r}{\sqrt{1+(F_0'(r))^2}}
				\exp\left(
				\int_1^r a_1(s)\,ds
				\right).
			$$
			Consequently we define
			\begin{equation}\label{def-tildeJ}
				\begin{aligned}
					\widetilde J^{-1}[g](r)
					:={}&
					z_1(r)
					\int_1^r
					\frac{
						\sqrt{1+(F_0'(\eta))^2}
					}{
						\eta z_1^2(\eta)
						\exp\bigl(\int_1^\eta a_1\bigr)
					}
					\\
					&\qquad\times
					\left[
					\int_1^\eta
					\frac{
						s z_1(s)
						\exp\bigl(\int_1^s a_1\bigr)
						(Q'(s)^2+P'(s)^2)
					}{
						\sqrt{1+(F_0'(s))^2}
					}
					g(s)\,ds
					\right]
					d\eta .
				\end{aligned}
			\end{equation}
			By construction,
			\begin{equation}\label{tildeJ-equation}
				\bar J(Q,P)
				\bigl[
				\widetilde J^{-1}[g]
				\bigr]
				-
				\frac{\bar J(Q,P)[z_1]}{z_1}
				\widetilde J^{-1}[g]
				=
				(Q'^2+P'^2)g.
			\end{equation}
The estimates for $z_1$ and $a_1$ imply
			\begin{equation}\label{tildeJ-est}
				\|\widetilde J^{-1}[g]\|_*
				\le
				C\|g\|_\infty.
			\end{equation}
			Near $r=1$ this follows from
			\[
			z_1(r)\sim\sqrt{r-1},
			\]
			together with the regularity built into the double integral in
			\eqref{def-tildeJ}. For $r\ge2$, $z_1$ is bounded above and below,
			and the outer kernel in \eqref{def-tildeJ} is of size $O(r^{-1})$,
			whereas the inner weight is of size $O(r)$. This gives precisely the
			growth allowed by \eqref{norma*}. Differentiating the representation
			once and twice gives the corresponding estimates for $h'$ and $h''$.
			
			We now restore the zero-order defect. Define
		$$
				\mathcal B[h]
				=
				\widetilde J^{-1}
				\left[
				(Q'^2+P'^2)^{-1}
				\frac{\bar J(Q,P)[z_1]}{z_1}h
				\right].
			$$
			Then the original equation is equivalent to
			$$
				(I+\mathcal B)h
				=
				\widetilde J^{-1}[g].
			$$
We claim
			\begin{equation}\label{B-exp}
				\left\|
				(Q'^2+P'^2)^{-1}
				\frac{\bar J(Q,P)[z_1]}{z_1}h
				\right\|_\infty
				\le
				C\ve^{1/4}\|h\|_*.
			\end{equation}
			Let us recall why this is true.
			Set
			\[
			\chi_1(r)=\chi_0(\ve^{1/2}r).
			\]
			Since
			\[
			z_1
			=
			\chi_1y_1+(1-\chi_1)y_R,
			\]
			the product rule gives
			\begin{align*}
				\bar J(Q,P)[z_1]
				={}&
				\chi_1\bar J(Q,P)[y_1]
				+
				(1-\chi_1)\bar J(Q,P)[y_R]
				+
				\ve\chi_0''(\ve^{1/2}r)(y_1-y_R)
				\nonumber\\
				&+
				2\ve^{1/2}\chi_0'(\ve^{1/2}r)(y_1'-y_R')
				+
				a(r)\ve^{1/2}\chi_0'(\ve^{1/2}r)(y_1-y_R).
			\end{align*}
			Using \eqref{a-exp}--\eqref{b-exp}, one obtains
			\begin{equation}\label{est-barJz1}
				\bar J(Q,P)[z_1]
				=
				\begin{cases}
					a_1(r)y_1'(r)+b_1(r)y_1(r),
					&
					1<r<\ve^{-1/2},
					\\[1mm]
					O(\ve^2)
					\bigl(
					1+
					|\chi_0''(\ve^{1/2}r)|
					+
					|\chi_0'(\ve^{1/2}r)|
					\bigr),
					&
					\ve^{-1/2}<r<2\ve^{-1/2},
					\\[1mm]
					O(\ve^2),
					&
					2\ve^{-1/2}<r<2\delta\ve^{-3/4},
					\\[1mm]
					0,
					&
					r>2\delta\ve^{-3/4},
				\end{cases}
			\end{equation}
			where $b_1=b-b_0$ in the first region.
			The last line is particularly important: once
			\[
			r>2\delta\ve^{-3/4},
			\]
			the approximate surface coincides with the sphere and
			\[
			z_1=y_R
			\]
			is an exact spherical Jacobi field. Therefore
			\begin{equation}\label{defect-support}
				\bar J(Q,P)[z_1]=0
				\qquad
				\text{for }
				r>2\delta\ve^{-3/4}.
			\end{equation}
			
			Combining \eqref{est-barJz1} with the definition of the
			$\|\cdot\|_*$-norm gives \eqref{B-exp}. More explicitly, for
			$r<2$ the degeneracy of $z_1$ is compensated by the factor
			$r(r-1)$ in the norm; for
$2<r<\ve^{-1/2}$
			one uses
			\[
			|h(r)|\le Cr^2\|h\|_*;
			\]
			for $\ve^{-1/2}<r<2\ve^{-1/2}$
			the estimate is $O(\ve)\|h\|_*$; and for
$2\ve^{-1/2}<r<2\delta\ve^{-3/4}$
			it is
			\[
			O(\ve^2r^2)\|h\|_*
			\le
			C_\delta\ve^{1/2}\|h\|_*.
			\]
			Thus \eqref{B-exp} follows after decreasing $\ve_0$ if necessary.
			
			By \eqref{tildeJ-est} and \eqref{B-exp},
			\[
			\|\mathcal B\|_{\mathcal L(X)}
			\le
			C\ve^{1/4},
			\qquad
			X=\{h:\|h\|_*<\infty\}.
			\]
			Hence, for $\ve$ sufficiently small,
			\(
			\|\mathcal B\|_{\mathcal L(X)}<\frac12,
			\)
			and
			\(
			(I+\mathcal B)^{-1}
			=
			\sum_{k=0}^{\infty}(-\mathcal B)^k.
			\)
			We therefore define
			$$
				\mathcal T_C
				=
				(I+\mathcal B)^{-1}\widetilde J^{-1}.
			$$
			Estimate \eqref{esti-neck} follows immediately from
			\eqref{tildeJ-est}. Uniqueness in the class $\|h\|_*<\infty$ follows
			from the same invertibility argument.
		\end{proof}
		
		\medskip
		
		The preceding lemma gives the uniform estimate required for the
		fixed-point argument used to solve the neck equation \eqref{C1}. It does not, however, distinguish between
		different types of right hand sides $g$ in \eqref{prob0}. This distinction becomes important near the
		outer end of the neck.
		
		A constant right hand side $g$ produces a quadratically growing particular solution,
		as already illustrated by the function $h_0$ constructed in
		Section~\ref{approx-domain}. By contrast, if the right hand side $g$ is supported in
		the catenoid--sphere transition region, its weighted radial mass becomes
		constant once the support is crossed, and the solution then grows only
		logarithmically.
		We make this precise in the following lemma.
		
		Set
		\begin{equation}\label{def-Lve}
			L_\ve
			:=
			2\delta\ve^{-3/4}.
		\end{equation}
		
		\begin{lemma}[Refined estimates for the neck inverse]
			\label{lemma-neck-refined}
			Let $\mathcal T_C$ be the inverse from
			Lemma~\ref{lemma-linear-neck}.
			
			\smallskip
			
			\noindent
			{\rm (i)}
			Set
			\[
			\zeta_\ve
			=
			\mathcal T_C[1].
			\]
			Then
			$$
				\|\zeta_\ve\|_*
				\le
				C,
			$$
			and for
			\(
			2<r<\ve^{-1+b}
			\)
			one has
			\begin{equation}\label{zeta-pointwise}
				|\zeta_\ve(r)|
				\le
				Cr^2,
				\qquad
				|\zeta_\ve'(r)|
				\le
				Cr,
				\qquad
				|\zeta_\ve''(r)|
				\le
				C.
			\end{equation}
			Consequently, for every constant $\gamma\in\R$,
			\begin{equation}\label{constant-neck-pointwise}
				\begin{aligned}
					|\mathcal T_C[\gamma](r)|
					&\le
					C|\gamma|r^2,
					\quad
					|(\mathcal T_C[\gamma])'(r)|
					\le
					C|\gamma|r,
					\quad
					|(\mathcal T_C[\gamma])''(r)|
					\le
					C|\gamma|.
				\end{aligned}
			\end{equation}
			
			\smallskip
			
			\noindent
			{\rm (ii)}
			Assume
			\(
			g_0\in L^\infty(1,\ve^{-1+b})
			\)
			and $\operatorname{supp}g_0
				\subset
				(1,L_\ve).$
			If
			\(
			u_0=\mathcal T_C[g_0],
			\)
			then
$$
				\|u_0\|_*
				\le
				C\|g_0\|_\infty,
$$
			and for
			\[
			2L_\ve<r<\ve^{-1+b}
			\]
			we have
			\begin{equation}\label{u0-refined}
				\begin{aligned}
					|u_0(r)|
					&\le
					C L_\ve^2\|g_0\|_\infty
					\left(
					1+\log\frac r{L_\ve}
					\right),
					\quad
					|u_0'(r)|
					\le
					C L_\ve^2\|g_0\|_\infty\frac1r,
					\\
					|u_0''(r)|
					&\le
					C L_\ve^2\|g_0\|_\infty\frac1{r^2}.
				\end{aligned}
			\end{equation}
			In particular, if
			\begin{equation}\label{g0-eps32}
				\|g_0\|_\infty
				\le
				C\ve^{3/2},
			\end{equation}
			then for
			\(
			2L_\ve<r<\ve^{-1+b}
			\)
			\begin{equation}\label{u0-refined-special}
				\begin{aligned}
					|u_0(r)|
					&\le
					C_\delta
					\left(
					1+
					\left|
					\log(\ve^{3/4}r)
					\right|
					\right),
					\quad
					|u_0'(r)|
					\le
					\frac{C_\delta}{r},
					\quad
					|u_0''(r)|
					\le
					\frac{C_\delta}{r^2}.
				\end{aligned}
			\end{equation}

		\end{lemma}
		
		\begin{proof}
			Part {\rm (i)} is an immediate consequence of
			Lemma~\ref{lemma-linear-neck}. Indeed,
			\[
			\|\zeta_\ve\|_*
			=
			\|\mathcal T_C[1]\|_*
			\le C.
			\]
			For $r\ge2$, the definition of the norm then gives
			\[
			|\zeta_\ve(r)|
			\le Cr(r-1)
			\le Cr^2,
			\quad
			|\zeta_\ve'(r)|
			\le Cr,
			\qquad
			|\zeta_\ve''(r)|
			\le C,
			\]
			which proves \eqref{zeta-pointwise}. By linearity,
			\[
			\mathcal T_C[\gamma]
			=
			\gamma\zeta_\ve,
			\]
			and \eqref{constant-neck-pointwise} follows.
			
			We turn to part {\rm (ii)}. The exact inverse satisfies
			\[
			u_0
			=
			(I+\mathcal B)^{-1}
			\widetilde J^{-1}[g_0].
			\]
			Equivalently, from the definition of $\mathcal B$,
			$$
				u_0
				=
				\widetilde J^{-1}[\widehat g_0],
			$$
			where
			$$
				\widehat g_0
				=
				g_0
				-
				V_\ve u_0,
				\qquad
				V_\ve
				=
				(Q'^2+P'^2)^{-1}
				\frac{\bar J(Q,P)[z_1]}{z_1}.
			$$
			By \eqref{defect-support},
			\[
			V_\ve(r)=0
			\qquad
			\text{for }r>L_\ve.
			\]
			Since $g_0$ also vanishes there,
			\begin{equation}\label{ghat-support}
				\operatorname{supp}\widehat g_0
				\subset
				(1,L_\ve).
			\end{equation}
			Moreover, \eqref{B-exp} and \eqref{esti-neck} imply
			\[
			\|V_\ve u_0\|_\infty
			\le
			C\ve^{1/4}\|u_0\|_*
			\le
			C\ve^{1/4}\|g_0\|_\infty.
			\]
			Hence
			\begin{equation}\label{ghat-bound}
				\|\widehat g_0\|_\infty
				\le
				C\|g_0\|_\infty.
			\end{equation}
			We now use the representation \eqref{def-tildeJ}. Write
			$$
				K_\ve(\eta)
				=
				\frac{
					\sqrt{1+(F_0'(\eta))^2}
				}{
					\eta z_1^2(\eta)
					\exp(\int_1^\eta a_1)
				},
			\quad 
				W_\ve(s)
				=
				\frac{
					s z_1(s)
					\exp(\int_1^s a_1)
					(Q'(s)^2+P'(s)^2)
				}{
					\sqrt{1+(F_0'(s))^2}
				}.
			$$
			Then
			\begin{equation}\label{u0-representation-refined}
				u_0(r)
				=
				z_1(r)
				\int_1^r
				K_\ve(\eta)
				\left[
				\int_1^\eta
				W_\ve(s)\widehat g_0(s)\,ds
				\right]d\eta.
			\end{equation}
			For $r>L_\ve$ (see \eqref{def-Lve}), the approximate surface already coincides exactly
			with the sphere. In particular,
			\[
			z_1(r)=y_R(r)
			=
			\sqrt{1-\frac{r^2}{R^2}},
			\]
			and, since
$r\le\ve^{-1+b},$
			we have
			\[
			\ve^2r^2\le\ve^{2b}=o(1).
			\]
			Consequently,
			$$
				|z_1(r)|\le C,
				\qquad
				|z_1'(r)|\le C\ve^2r,
				\qquad
				|z_1''(r)|\le C\ve^2.
			$$
			The estimates for $a_1$ also give
			$$
				|K_\ve(r)|
				\le
				\frac Cr,
				\qquad
				|W_\ve(r)|
				\le
				Cr.
			$$
			By \eqref{ghat-support}, for $\eta>L_\ve$,
			\[
			\int_1^\eta
			W_\ve(s)\widehat g_0(s)\,ds
			=
			\int_1^{L_\ve}
			W_\ve(s)\widehat g_0(s)\,ds
			=:
			M_0.
			\]
			Using \eqref{ghat-bound} and the bounds on $W_\ve$ gives
			$$
				|M_0|
				\le
				C
				\int_1^{L_\ve}
				s|\widehat g_0(s)|\,ds
				\le
				CL_\ve^2\|g_0\|_\infty.
			$$
			We split the outer integral in \eqref{u0-representation-refined} at
			$L_\ve$. For $r>2L_\ve$,
			\[
			\begin{aligned}
				|u_0(r)|
				&\le
				CL_\ve^2\|g_0\|_\infty
				+
				C|M_0|
				\int_{L_\ve}^r\frac{d\eta}{\eta}
				\le
				CL_\ve^2\|g_0\|_\infty
				\left(
				1+\log\frac r{L_\ve}
				\right).
			\end{aligned}
			\]
			This proves the first estimate in \eqref{u0-refined}.
			Differentiating \eqref{u0-representation-refined}, for
			$r>L_\ve$ we have
			\begin{align*}
				u_0'(r)
				={}&
				z_1'(r)
				\int_1^r
				K_\ve(\eta)
				\left[
				\int_1^\eta
				W_\ve(s)\widehat g_0(s)\,ds
				\right]d\eta
				+
				z_1(r)K_\ve(r)M_0.
			\end{align*}
			The second term is bounded by $C|M_0|/r$. For the first term,
			\[
			\begin{aligned}
				&\left|
				z_1'(r)
				\int_1^r
				K_\ve(\eta)
				\left[
				\int_1^\eta
				W_\ve(s)\widehat g_0(s)\,ds
				\right]d\eta
				\right|
				\, \le
				C\ve^2r
				|M_0|
				\left(
				1+\log\frac r{L_\ve}
				\right).
			\end{aligned}
			\]
			Since
			\[
			\ve^2r^2
			\left(
			1+|\log\ve|
			\right)
			\le
			C\ve^{2b}
			\left(
			1+|\log\ve|
			\right)
			=o(1),
			\]
			this is also bounded by $C|M_0|/r$. Therefore
			\[
			|u_0'(r)|
			\le
			CL_\ve^2\|g_0\|_\infty\frac1r.
			\]
			
			Finally, for $r>L_\ve$ both $\widehat g_0$ and $V_\ve$ vanish.
			Hence $u_0$ satisfies the homogeneous spherical Jacobi equation in
			this region. Using \eqref{defJR} together with the bounds already
			obtained for $u_0$ and $u_0'$, we get
			\[
			|u_0''(r)|
			\le
			C\frac{|M_0|}{r^2}
			+
			C\ve^2|M_0|
			\left(
			1+\log\frac r{L_\ve}
			\right),
			\]
			and the second term is again absorbed because
			\[
			\ve^2r^2(1+|\log\ve|)=o(1).
			\]
			Thus
			\[
			|u_0''(r)|
			\le
			CL_\ve^2\|g_0\|_\infty\frac1{r^2}.
			\]
			This proves \eqref{u0-refined}.
			
			If \eqref{g0-eps32} holds, then
			\[
			L_\ve^2\|g_0\|_\infty
			\le
			C\delta^2\ve^{-3/2}\ve^{3/2}
			\le
			C_\delta,
			\]
			and \eqref{u0-refined-special} follows.
			
			Part {\rm (iii)} is simply a consequence of linearity:
			\[
			\mathcal T_C[\gamma+g_0]
			=
			\gamma\mathcal T_C[1]
			+
			\mathcal T_C[g_0].
			\]
		\end{proof}
		
		\begin{remark}\label{remark-neck-refined}
			The preceding lemma clarifies the different behaviours that will occur
			in the nonlinear neck problem.
			
			A constant right hand side produces, in general, a component with
			quadratic growth. This is already visible in the explicit first
			correction $h_0$, for which
			\[
			J_{\Sigma_0}[h_0]=-2\ve,
			\qquad
			h_0(r)
			=
			-\frac{\ve}{2}r^2
			+
			O(\ve\log r).
			\]
			Accordingly, the spatially constant quantity
			$\bar\lambda_\ve[h]$ appearing in \eqref{C1} cannot be expected to
			generate a logarithmically growing correction.
			
			By contrast, a right hand side of size $O(\ve^{3/2})$ supported in
			$r\lesssim\ve^{-3/4}$ produces, after this support has been crossed,
			a solution satisfying
			\[
			|u(r)|\lesssim1+|\log(\ve^{3/4}r)|,
			\qquad
			|u'(r)|\lesssim r^{-1},
			\qquad
			|u''(r)|\lesssim r^{-2}.
			\]
			This improved information will later be used to estimate the cutoff
			terms generated by $h_C^1$ in the spherical equation.
		\end{remark}

		\subsection{The first neck correction \(h_C^0\)}
		\label{subsec-hC0}
		
		We now construct the first correction in the neck. Its purpose is to
		cancel the leading transition error \(E_{02}\) generated by the
		interpolation between the modified catenoid and the spherical profile.
		
		Recall from \eqref{def-E02} that
		$$
			\begin{aligned}
				E_{02}(r)
				={}&
				\left(
				\chi''(r)+\frac1r\chi'(r)
				\right)
				\bigl(F(r)-G(r)\bigr)
				+
				2\chi'(r)
				\bigl(F'(r)-G'(r)\bigr)
				+
				\mathcal R_{02}(r),
			\end{aligned}
		$$
		where
		\begin{equation}\label{R02-est}
			|\mathcal R_{02}(r)|
			\le
			C\ve^2|\log\delta|\,
			{\bf1}_{(r_0,\,2r_0)}(r),
			\qquad
			r_0:=\delta\ve^{-3/4}.
		\end{equation}
		In particular,
		$$
			\operatorname{supp}E_{02}
			\subset
			[r_0,2r_0].
		$$
		
		The general refined estimate of
		Lemma~\ref{lemma-neck-refined} already shows that a right hand side of this
		type generates at most logarithmic growth outside its support.
		For the subsequent balancing argument, however, we need considerably
		more information: we must determine the coefficient of this logarithm.
		For this reason we use directly the integral representation introduced
		in the proof of Lemma~\ref{lemma-linear-neck}.
		
		We seek an axially symmetric function
		\[
		h_C^0:
		\Sigma_0^{\ve,+,1}\longrightarrow\R
		\]
		satisfying
		\[
		h_C^0(1)=0,
		\qquad
		\lim_{r\to1^+}
		\sqrt{r^2-1}\,(h_C^0)'(r)=0,
		\]
		and such that
		\[
		J_{\Sigma_0^\ve}[h_C^0]
		=
		E_{02}
		+
		\text{a lower-order error}.
		\]
		Recall the approximate positive Jacobi field
		$$
			z_1(r)
			=
			y_1(r)\chi_0(\ve^{1/2}r)
			+
			y_R(r)
			\bigl(1-\chi_0(\ve^{1/2}r)\bigr),
		$$
		where
		\[
		y_1(r)
		=
		\frac{\sqrt{r^2-1}}r,
		\qquad
		y_R(r)
		=
		\sqrt{1-\frac{r^2}{R^2}}.
		\]
		We define
		\begin{equation}\label{def-hC0}
			h_C^0(r)
			=
			z_1(r)\psi_C^0(r),
		\end{equation}
		where
		\begin{align}
			\psi_C^0(r)
			={}&
			\int_1^r
			\frac{
				\sqrt{1+(F_0'(\eta))^2}
			}{
				\eta z_1^2(\eta)
				\exp\!\left(\int_1^\eta a_1(\tau)\,d\tau\right)
			}
			\nonumber\\
			&\qquad\times
			\left[
			\int_1^\eta
			\frac{
				s z_1(s)
				\exp\!\left(\int_1^s a_1(\tau)\,d\tau\right)
				(Q'(s)^2+P'(s)^2)
			}{
				\sqrt{1+(F_0'(s))^2}
			}
			E_{02}(s)\,ds
			\right]d\eta .
			\label{def-psiC0}
		\end{align}
		Thus \(h_C^0\) is obtained by applying the approximate inverse
		\(\widetilde J^{-1}\) to \(E_{02}\). We use the approximate rather
		than the exact inverse at this stage because the representation
		\eqref{def-psiC0} allows us to compute explicitly the logarithmic
		coefficient of \(h_C^0\).
		
		By the computation leading to \eqref{tildeJ-equation},
		\begin{equation}\label{hC0-equation}
			J_{\Sigma_0^\ve}[h_C^0]
			=
			E_{02}
			+
			\frac1{Q'^2+P'^2}
			\frac{\bar J(Q,P)[z_1]}{z_1}
			h_C^0.
		\end{equation}
		Equivalently,
		\begin{equation}\label{hC0-equation-psi}
			J_{\Sigma_0^\ve}[h_C^0]
			=
			E_{02}
			+
			\frac{\bar J(Q,P)[z_1]}
			{Q'^2+P'^2}
			\psi_C^0.
		\end{equation}
		The second term will be shown to be of lower order.
		
		The next lemma contains the precise information that will be needed in
		the spherical balancing argument.
		
		\begin{lemma}\label{lemma-h0C}
			Let \(0<b<\frac14\) be fixed and sufficiently small, and let
			\(\delta>0\) be fixed and sufficiently small. Then the function
			\(h_C^0\) defined in \eqref{def-hC0}--\eqref{def-psiC0} has the
			following properties.
			First,
			\begin{equation}\label{hC0-zero}
				h_C^0(r)=0,
				\qquad
				1<r\le r_0 :=\delta\ve^{-3/4}.
			\end{equation}
			In the transition region
			\[
			r_0<r<2r_0
			\]
			one has
			\begin{equation}\label{hC0-transition}
				|h_C^0(r)|
				\le C_\delta,
				\qquad
				|(h_C^0)'(r)|
				\le
				\frac{C_\delta}{r_0},
				\qquad
				|(h_C^0)''(r)|
				\le
				\frac{C_\delta}{r_0^2}.
			\end{equation}
			Moreover, uniformly for
			\[
			2r_0<r<\ve^{-1+b},
			\]
			one has
			\begin{equation}\label{hC0-derivative-asymptotic}
				r(h_C^0)'(r)
				=
				-\left(
				1+O(\delta^4)+O(\ve^b)
				\right),
			\end{equation}
			and consequently
			\begin{equation}\label{hC0-asymptotic}
				h_C^0(r)
				=
				-\left(
				1+O(\delta^4)+O(\ve^b)
				\right)
				\log\frac r{r_0}
				+
				O_\delta(1).
			\end{equation}
			In addition,
			\begin{equation}\label{hC0-second-outer}
				|(h_C^0)''(r)|
				\le
				\frac{C_\delta}{r^2},
				\qquad
				2r_0<r<\ve^{-1+b},
			\end{equation}
			and
			\begin{equation}\label{norm-hC0}
				\|h_C^0\|_*
				\le
				C_\delta
				\ve^{3/2}|\log\ve|.
			\end{equation}
			Finally,
			\begin{equation}\label{hC0-error-pointwise}
				J_{\Sigma_0^\ve}[h_C^0]-E_{02}
				=
				0
				\qquad
				\text{for }
				r\notin(r_0,2r_0),
			\end{equation}
			while
		$$
				\left|
				J_{\Sigma_0^\ve}[h_C^0]-E_{02}
				\right|
				\le
				C_\delta
				\ve^2|\log\delta|
				\qquad
				\text{for }
                r \in (r_0,2r_0).
			$$
			In particular,
			\begin{equation}\label{e-00-2}
				\left\|
				J_{\Sigma_0^\ve}[h_C^0]-E_{02}
				\right\|_{L^\infty(1,\ve^{-1+b})}
				\le
				C_\delta
				\ve^2|\log\delta|.
			\end{equation}
		\end{lemma}
		
		\begin{proof}
			Since \(E_{02}\) vanishes on \((1,r_0)\), the definition
			\eqref{def-psiC0} immediately gives
			\[
			\psi_C^0(r)=0,
			\qquad
			h_C^0(r)=0,
			\qquad
			1<r\le r_0.
			\]
			This proves \eqref{hC0-zero}.
			
			For \(r>r_0\), formula \eqref{def-psiC0} becomes
			\begin{align}
				h_C^0(r)
				={}&
				z_1(r)
				\int_{r_0}^r
				K_\ve(\eta)
				\left[
				\int_{r_0}^\eta
				W_\ve(s)E_{02}(s)\,ds
				\right]d\eta,
				\label{hC0-explicit}
			\end{align}
			where, as in the proof of
			Lemma~\ref{lemma-neck-refined},
			\[
			K_\ve(\eta)
			=
			\frac{
				\sqrt{1+(F_0'(\eta))^2}
			}{
				\eta z_1^2(\eta)
				\exp(\int_1^\eta a_1)
			},
			\quad 
			W_\ve(s)
			=
			\frac{
				s z_1(s)
				\exp(\int_1^s a_1)
				(Q'(s)^2+P'(s)^2)
			}{
				\sqrt{1+(F_0'(s))^2}
			}.
			\]
			We first record the size of these coefficients in the range
			\[
			r_0<r<\ve^{-1+b}.
			\]
			Since
			\[
			r_0=\delta\ve^{-3/4}\gg\ve^{-1/2},
			\]
			we are already in the spherical part of the interpolating Jacobi field,
			and therefore
			\[
			z_1(r)
			=
			y_R(r)
			=
			\sqrt{1-\ve^2r^2}.
			\]
			As
			\(
			\ve^2r^2\le\ve^{2b},
			\)
			we obtain
			\begin{equation}\label{z1-hC0-est}
				z_1(r)
				=
				1+O(\ve^{2b}),
				\qquad
				z_1'(r)
				=
				O(\ve^2r).
			\end{equation}
			Furthermore,
			\[
			1+(F_0'(r))^2
			=
			1+O(r^{-2})
			=
			1+O_\delta(\ve^{3/2}),
			\]
			and
			\[
			Q'(r)^2+P'(r)^2
			=
			1+O(\ve^{2b}).
			\]
			Finally, from \eqref{a1-large},
			\[
			\exp\left(\int_1^r a_1(s)\,ds\right)
			=
			1+O(\ve^{2b}|\log\ve|).
			\]
			Since \(b>0\) is fixed, after decreasing \(b\), if necessary, all these
			errors may be absorbed into \(O(\ve^b)\). Hence, uniformly in the
			range under consideration,
			\begin{equation}\label{KW-hC0}
				K_\ve(r)
				=
				\frac1r
				\bigl(1+O(\ve^b)\bigr),
				\qquad
				W_\ve(r)
				=
				r\bigl(1+O(\ve^b)\bigr).
			\end{equation}
			The decisive point is now to compute the weighted radial mass of
			\(E_{02}\). We write
			\[
			E_{02}
			=
			\bar E_{02}+\mathcal R_{02},
			\]
			where
			$$
				\bar E_{02}
				=
				\left(
				\chi''+\frac1r\chi'
				\right)(F-G)
				+
				2\chi'(F'-G').
			$$
			Set
			\(
			D=F-G
			\) (see \eqref{defG} and \eqref{def-qF})
			Then
			\[
			r\bar E_{02}
			=
			r\chi''D+\chi'D+2r\chi'D'.
			\]
			Since
			\[
			(r\chi'D)'
			=
			r\chi''D+\chi'D+r\chi'D',
			\]
			we have the exact identity
			$$
				r\bar E_{02}
				=
				(r\chi'D)'
				+
				r\chi'D'.
			$$
			Integrating from \(r_0\) to \(\eta\), and using
			\(\chi'(r_0)=0\), gives
			\begin{equation}\label{ibp-E02}
				\int_{r_0}^{\eta}
				s\bar E_{02}(s)\,ds
				=
				\eta\chi'(\eta)D(\eta)
				+
				\int_{r_0}^{\eta}
				s\chi'(s)D'(s)\,ds.
			\end{equation}
			We next use the specific matching of the catenoidal and spherical
			profiles. On the support of \(\chi'\), namely
			\[
			r_0<s<2r_0,
			\]
			the expansions obtained in the construction of the approximate surface
			give
			$$
				s\bigl(F'(s)-G'(s)\bigr)
				=
				1+O(\delta^4)+O(\ve^{1/2})
			$$
			see \eqref{FG-derivatives}. Therefore,
			\begin{align}
				\int_{r_0}^{\eta}
				s\chi'(s)
				\bigl(F'(s)-G'(s)\bigr)\,ds
				={}&
				\int_{r_0}^{\eta}
				\chi'(s)
				\left(
				1+O(\delta^4)+O(\ve^{1/2})
				\right)ds
				\nonumber\\
				={}&
				\bigl(\chi(\eta)-1\bigr)
				\left(
				1+O(\delta^4)+O(\ve^{1/2})
				\right),
				\label{main-inner-integral}
			\end{align}
			uniformly for \(r_0<\eta<2r_0\).
			The first term on the right-hand side of \eqref{ibp-E02},
			\[
			\eta\chi'(\eta)D(\eta),
			\]
			is supported in \((r_0,2r_0)\) and is uniformly bounded there.
			More importantly, once \(\eta\ge2r_0\) it vanishes identically.
			Since \(\chi(2r_0)=0\), it follows from
			\eqref{ibp-E02}--\eqref{main-inner-integral} that
			$$
				\int_{r_0}^{\eta}
				s\bar E_{02}(s)\,ds
				=
				-1+O(\delta^4)+O(\ve^{1/2}),
				\qquad
				\eta\ge2r_0.
			$$
			We now include the remainder \(\mathcal R_{02}\). From
			\eqref{R02-est},
			\[
			\begin{aligned}
				\left|
				\int_{r_0}^{2r_0}
				s\mathcal R_{02}(s)\,ds
				\right|
				\le
				C\ve^2|\log\delta|\,r_0^2 =
				C\delta^2\ve^{1/2}|\log\delta|.
			\end{aligned}
			\]
			For fixed \(\delta\), this term can be absorbed into \(O(\ve^b)\),
			since \(b<\frac14\). Thus
			$$
				\int_{r_0}^{\eta}
				sE_{02}(s)\,ds
				=
				-1+O(\delta^4)+O(\ve^b),
				\qquad
				\eta\ge2r_0.
			$$
			The weights \(W_\ve\) differ from \(s\) only by a relative
			\(O(\ve^b)\)-error. Hence the same computation yields
			\begin{equation}\label{weighted-moment-E02}
				\int_{r_0}^{\eta}
				W_\ve(s)E_{02}(s)\,ds
				=
				-1+O(\delta^4)+O(\ve^b),
				\qquad
				\eta\ge2r_0.
			\end{equation}
			For \(r>2r_0\), differentiating \eqref{hC0-explicit} gives
			\begin{align*}
				(h_C^0)'(r)
				={}&
				z_1'(r)
				\int_{r_0}^{r}
				K_\ve(\eta)
				\left[
				\int_{r_0}^{\eta}
				W_\ve(s)E_{02}(s)\,ds
				\right]d\eta
				+
				z_1(r)K_\ve(r)
				\int_{r_0}^{r}
				W_\ve(s)E_{02}(s)\,ds.
			\end{align*}
			Using \eqref{z1-hC0-est}, \eqref{KW-hC0}, and
			\eqref{weighted-moment-E02}, the second term equals
			\[
			-\frac1r
			\left(
			1+O(\delta^4)+O(\ve^b)
			\right).
			\]
			The first term is smaller. Indeed,
			\[
			|z_1'(r)|
			\le
			C\ve^2r,
			\]
			whereas the integral multiplying \(z_1'\) is
			\(O(1+\log(r/r_0))\). Hence
			\[
			\left|
			z_1'(r)
			\int_{r_0}^{r}\cdots d\eta
			\right|
			\le
			C\ve^2r
			\left(
			1+\left|\log\frac r{r_0}\right|
			\right).
			\]
			Since
			\(
			r<\ve^{-1+b},
			\)
			we have
			\[
			\ve^2r^2|\log\ve|
			\le
			C\ve^{2b}|\log\ve|
			=o(\ve^b),
			\]
			and therefore
			\[
			(h_C^0)'(r)
			=
			-\frac1r
			\left(
			1+O(\delta^4)+O(\ve^b)
			\right).
			\]
			This proves \eqref{hC0-derivative-asymptotic}.
			
			Integrating from \(2r_0\) to \(r\), we obtain
			\[
			h_C^0(r)
			=
			h_C^0(2r_0)
			-
			\left(
			1+O(\delta^4)+O(\ve^b)
			\right)
			\log\frac r{2r_0}.
			\]
			The explicit representation on the transition annulus gives
			\[
			|h_C^0(2r_0)|
			\le
			C_\delta.
			\]
			Absorbing the harmless constant \(\log2\) into the bounded remainder,
			we obtain
			\[
			h_C^0(r)
			=
			-\left(
			1+O(\delta^4)+O(\ve^b)
			\right)
			\log\frac r{r_0}
			+
			O_\delta(1),
			\]
			which proves \eqref{hC0-asymptotic}.
			
			We now justify the estimates in the transition region. On
			\((r_0,2r_0)\), the support has length comparable with \(r_0\), while
			\[
			|E_{02}(r)|
			\le
			C_\delta r_0^{-2}.
			\]
			Indeed,
			\[
			r_0^{-2}
			=
			\delta^{-2}\ve^{3/2},
			\]
			which is precisely the scale of the principal transition error.
			Using the representation \eqref{hC0-explicit} and the bounds
			\[
			|K_\ve(r)|\le\frac{C_\delta}{r},
			\qquad
			|W_\ve(r)|\le C_\delta r,
			\]
			we obtain
			\[
			|h_C^0(r)|\le C_\delta,
			\qquad
			|(h_C^0)'(r)|\le C_\delta r_0^{-1}.
			\]
			The equation \eqref{hC0-equation} then gives
			\[
			|(h_C^0)''(r)|
			\le
			C_\delta r_0^{-2}.
			\]
			This proves \eqref{hC0-transition}.
			
			For \(r>2r_0\), differentiating
			\eqref{hC0-derivative-asymptotic}, or equivalently using the homogeneous
			spherical Jacobi equation satisfied there, yields
			\[
			|(h_C^0)''(r)|
			\le
			C_\delta r^{-2},
			\]
			which is \eqref{hC0-second-outer}.
			
			We next estimate the weighted norm. Since \(h_C^0=0\) for
			\(r\le r_0\) and \(r_0\gg1\), the preceding estimates give
			\[
			\frac{|h_C^0(r)|}{r(r-1)}
			\le
			C_\delta
			\frac{
				1+|\log(r/r_0)|
			}{
				r^2
			}
			\le
			C_\delta
			\frac{|\log\ve|}{r_0^2}.
			\]
			Since
			\[
			r_0^{-2}
			=
			\delta^{-2}\ve^{3/2},
			\]
			we obtain
			\[
			\left\|
			\frac{h_C^0}{r(r-1)}
			\right\|_\infty
			\le
			C_\delta
			\ve^{3/2}|\log\ve|.
			\]
			Likewise,
			\[
			\left|
			\frac{(h_C^0)'(r)}r
			\right|
			\le
			\frac{C_\delta}{r^2}
			\le
			C_\delta\ve^{3/2},
			\]
			and
			\[
			|(h_C^0)''(r)|
			\le
			C_\delta\ve^{3/2}
			\]
			at the smallest relevant radius \(r\sim r_0\). The same estimates hold
			in the transition region by \eqref{hC0-transition}. Hence
			\[
			\|h_C^0\|_*
			\le
			C_\delta
			\ve^{3/2}|\log\ve|,
			\]
			which proves \eqref{norm-hC0}.
			
			It remains to estimate the defect in the equation. From
			\eqref{hC0-equation-psi},
			$$
				J_{\Sigma_0^\ve}[h_C^0]-E_{02}
				=
				\frac{\bar J(Q,P)[z_1]}
				{Q'^2+P'^2}
				\psi_C^0.
$$
			For \(r<r_0\), both \(E_{02}\) and \(h_C^0\) vanish, and hence so does
			the defect. For \(r>2r_0\), the approximate surface is exactly
			spherical and \(z_1=y_R\) is an exact spherical Jacobi field. Therefore
			\[
			\bar J(Q,P)[z_1]=0,
			\qquad
			r>2r_0.
			\]
			Thus the defect is supported entirely in the transition region
			\((r_0,2r_0)\).
			
			There,
			\[
			|\bar J(Q,P)[z_1]|
			\le
			C_\delta\ve^2,
			\qquad
			|\psi_C^0|
			\le
			C_\delta|\log\delta|,
			\]
			and \(Q'^2+P'^2\) is uniformly bounded above and below. Consequently,
			\[
			\left|
			J_{\Sigma_0^\ve}[h_C^0]-E_{02}
			\right|
			\le
			C_\delta\ve^2|\log\delta|.
			\]
			This proves \eqref{hC0-error-pointwise}--\eqref{e-00-2} and completes
			the proof.
		\end{proof}
		
		\begin{remark}\label{remark-hC0-moment}
			The estimate of Lemma~\ref{lemma-h0C} is stronger than the general
			localized estimate in Lemma \ref{lemma-neck-refined}. The
			reason is the special structure of \(E_{02}\). Its weighted radial
			moment is not merely bounded but satisfies
			\[
			\int_{r_0}^{2r_0}
			sE_{02}(s)\,ds
			=
			-1+O(\delta^4)+O(\ve^b).
			\]
			It is precisely this identity that determines the leading coefficient
			of the logarithmic tail:
			\[
			(h_C^0)'(r)
			\sim
			-\frac1r.
			\]
			This coefficient will later produce the leading contribution
			\(2\pi\ve^2\) in the projection of the spherical cutoff error
			\(E_{04}\) onto the vertical Jacobi field.
		\end{remark}

		\section{The linear problem on the sphere}
		\label{linear-sphere}
		
		We now develop the linear theory required to solve the spherical equation \eqref{Sve2}. The purpose
		of this section is twofold. First, we compare the Jacobi operator on
		the portion \(S_\ve\) of the approximate surface with the Jacobi
		operator on the exact sphere \(S_R\). Second, we solve the resulting
		equation on \(S_R\) modulo the one-dimensional obstruction generated,
		in the axially symmetric class, by vertical translations.
		
		Recall that
		\[
		S_R
		=
		\partial B_R((R+d)e_3),
		\qquad
		R=\ve^{-1},
		\]
		and
		\[
		S_\ve
		=
		\left(
		\Sigma_0^{\ve,+,1}
		\cap
		\left\{
		r>\frac{\delta}{4}\ve^{-3/4}
		\right\}
		\right)
		\cup
		\Sigma_0^{\ve,+,2}.
		\]
		Thus \(S_\ve\) contains the catenoid--sphere transition region together
		with the entire spherical part of the approximate surface.
		
		As described in Section~\ref{scheme}, we identify \(S_\ve\) with the subset of
		\(S_R\) obtained by removing a small spherical cap by means of the map
		\[
		\Phi:S_\ve\longrightarrow S_R,
		\qquad
		\Phi(x)=\widetilde x,
		\]
		where (see \eqref{tttx})
		$$
			\widetilde x
			=
			\begin{cases}
				(r\cos\theta,r\sin\theta,G(r)),
				&
				x=(Q(r)\cos\theta,Q(r)\sin\theta,P(r)),
				\\[1mm]
				x,
				&
				x\in\Sigma_0^{\ve,+,2}.
			\end{cases}
	$$
		The map \(\Phi\) is exactly the identity outside the transition region
		and satisfies
		$$
			|d\Phi-I|
			\le
			C_\delta\ve^{3/4},
			\qquad
			|D(d\Phi)|
			\le
			C_\delta\ve^{3/2}.
		$$
		Given a function
		\[
		h_{S_R}:S_R\longrightarrow\R,
		\]
		we define the corresponding function on \(S_\ve\) by
		$$
			h_{S_\ve}(x)
			=
			h_{S_R}(\Phi(x)).
		$$
		
		\subsection{Comparison of the two Jacobi operators}
		
		We first describe the difference between the Jacobi operator on
		\(S_\ve\) and the Jacobi operator of the exact sphere.
		
		On the sphere $
		S_R=\partial B_R((R+d)e_3)
		$
		the Jacobi operator is
$$
			J_{S_R}[h]
			=
			\Delta_{S_R}h+\frac2{R^2}h.
$$
		On its lower hemisphere, for an axially symmetric function
		\(f=f(r)\), this takes the form
$$
			J_{S_R}[f]
			=
			\frac{R^2-r^2}{R^2}f''
			+
			\frac{R^2-2r^2}{R^2r}f'
			+
			\frac2{R^2}f.
$$
		See Appendix \ref{app3}, Subsection \ref{sphere}. On the other hand, if
		\[
		x=(Q(r)\cos\theta,Q(r)\sin\theta,P(r))
		\in S_\ve
		\quad
		{\mbox {and}}
		\quad
		f(r)
		=
		h_{S_\ve}(x)
		=
		h_{S_R}(\widetilde x),
		\]
		then
		\[
		J_{\Sigma_0^\ve}[h_{S_\ve}](x)
		=
		J(Q,P)[f](r).
		\]
		We therefore write
	$$
			J_{\Sigma_0^\ve}[h_{S_\ve}](x)
			=
			J_{S_R}[h_{S_R}](\widetilde x)
			+
			B_{S_R}[h_{S_R}](\widetilde x).
		$$
        The operator \(B_{S_R}\) is supported entirely in the region where
		the approximate surface differs from \(S_R\):
		$$
			\operatorname{supp}B_{S_R}
			\subset
			\left\{
			\frac{\delta}{4}\ve^{-3/4}
			<
			r
			<
			2\delta\ve^{-3/4}
			\right\}.
		$$
		In radial coordinates it has the form
		\begin{equation}\label{AA}
			B_{S_R}[h_{S_R}](\widetilde x)
			=
			A_2(r)f''(r)
			+
			A_1(r)f'(r)
			+
			A_0(r)f(r),
		\end{equation}
		with
		\begin{equation}\label{Ai-est}
			|A_2(r)|
			\le
			C_\delta\ve,
			\qquad
			|A_1(r)|
			\le
			C_\delta\ve^{7/4},
			\qquad
			|A_0(r)|
			\le
			C_\delta\ve^{5/2}.
		\end{equation}
		We now see how these coordinate estimates translate into intrinsic
		estimates on \(S_R\). For fixed \(\theta\), let
		\[
		\gamma(r)
		=
		(r\cos\theta,r\sin\theta,G(r)).
		\]
		Then
		\[
		f'(r)
		=
		\left\langle
		D_{S_R}h_{S_R},
		\dot\gamma(r)
		\right\rangle,
		\]
		and
		\[
		f''(r)
		=
		D^2_{S_R}h_{S_R}
		(\dot\gamma,\dot\gamma)
		+
		\left\langle
		D_{S_R}h_{S_R},
		\nabla_{\dot\gamma}\dot\gamma
		\right\rangle.
		\]
		Since
		\[
		|\dot\gamma(r)|
		=
		\frac{R}{\sqrt{R^2-r^2}},
		\]
		we have
		$$
			|f'(r)|
			\le
			\frac{R}{\sqrt{R^2-r^2}}
			|D_{S_R}h_{S_R}|.
		$$
		Likewise,
		$$
			|f''(r)|
			\le
			C
			\frac{R^2}{R^2-r^2}
			|D^2_{S_R}h_{S_R}|
			+
			C
			\frac{R^2}{(R^2-r^2)^{3/2}}
			|D_{S_R}h_{S_R}|.
		$$
		On the support of \(B_{S_R}\),
		\[
		r\le
		2\delta\ve^{-3/4},
		\qquad
		R=\ve^{-1},
		\]
		and therefore
		\[
		\frac{r^2}{R^2}
		\le
		C_\delta\ve^{1/2}.
		\]
		Consequently,
		\[
		\frac{R}{\sqrt{R^2-r^2}}
		=
		1+O_\delta(\ve^{1/2}),
		\]
		and in particular, for \(\ve\) sufficiently small,
		\begin{equation}\label{f-intrinsic-simple}
			|f'|
			\le
			C|D_{S_R}h_{S_R}|,
			\qquad
			|f''|
			\le
			C|D^2_{S_R}h_{S_R}|
			+
			C\ve|D_{S_R}h_{S_R}|.
		\end{equation}
		Combining \eqref{AA}, \eqref{Ai-est}, and
		\eqref{f-intrinsic-simple}, we obtain
		\begin{equation}\label{BSR-pointwise}
			|B_{S_R}[h_{S_R}]|
			\le
			C_\delta
			\left(
			\ve|D^2_{S_R}h_{S_R}|
			+
			\ve^{7/4}|D_{S_R}h_{S_R}|
			+
			\ve^{5/2}|h_{S_R}|
			\right).
		\end{equation}
		The same computation, applied to H\"older quotients of the
		coefficients, gives
		$$
			\begin{aligned}
				\|B_{S_R}[h_{S_R}]\|_{C^{0,\beta}_R(S_R)}
				\le
				C_\delta\Big(
				&
				\ve
				\|D^2_{S_R}h_{S_R}\|_{C^{0,\beta}_R(S_R)}
				\\ &+
				\ve^{7/4}
				\|D_{S_R}h_{S_R}\|_{L^\infty(S_R)}
				+
				\ve^{5/2}
				\|h_{S_R}\|_{L^\infty(S_R)}
				\Big),
			\end{aligned}
		$$
		where we use the scale-invariant norm
		$$
			\|q\|_{C^{0,\beta}_R(S_R)}
			:=
			\|q\|_{L^\infty(S_R)}
			+
			R^\beta[q]_{\beta,S_R}.
		$$
		Thus \(B_{S_R}\) is a small perturbation of the spherical Jacobi
		operator, localized near the catenoid--sphere transition.
		
		\subsection{Scaling to the unit sphere}
		
		It is convenient to translate and rescale \(S_R\) to the unit sphere $S$.
		We write
		\be \label{sphere-scaling}
			\widetilde x
			=
			(R+d)e_3+R\omega,
			\qquad
			\omega\in S,
		\ee
		and set
		\begin{equation}\label{h-scaling}
			h_{S_R}(\widetilde x)
			=
			R\,h_S(\omega).
		\end{equation}
		The factor \(R\) in \eqref{h-scaling} is natural for a normal
		displacement.
		Under this scaling,
		\[
		J_{S_R}[h_{S_R}](\widetilde x)
		=
		\frac1R
		J_S[h_S](\omega),
		\]
		where
		\begin{equation}\label{JacobiS}
			J_S[h_S]
			=
			\Delta_Sh_S+2h_S.
		\end{equation}
		Hence an equation
		\[
		J_{S_R}[h_{S_R}]
		+
		B_{S_R}[h_{S_R}]
		=
		g_{S_R}
		\]
		is equivalent to
	$$
			J_S[h_S]
			+
			B_S[h_S]
			=
			R\,g_{S_R}\big((R+d)e_3+R\omega\big),
	$$
		where
		$$
			B_S[h_S](\omega)
			=
			R\,
			B_{S_R}[h_{S_R}]
			\big((R+d)e_3+R\omega\big).
		$$
		The scaling relations are
		\[
		|h_{S_R}|=R|h_S|,
		\qquad
		|D_{S_R}h_{S_R}|=|D_Sh_S|,
		\qquad
		|D^2_{S_R}h_{S_R}|
		=
		R^{-1}|D_S^2h_S|.
		\]
		Using \eqref{BSR-pointwise}, we obtain
		\begin{equation}\label{BS}
			|B_S[h_S]|
			\le
			C_\delta
			\left(
			\ve|D_S^2h_S|
			+
			\ve^{3/4}|D_Sh_S|
			+
			\ve^{1/2}|h_S|
			\right).
		\end{equation}
		Consequently,
		\begin{equation}\label{BSS}
			\|B_S[h_S]\|_{C^{0,\beta}(S)}
			\le
			C_\delta\ve^{1/2}
			\|h_S\|_{C^{2,\beta}(S)}.
		\end{equation}
		The last estimate is deliberately written with the largest of the
		three coefficients. The sharper pointwise powers in \eqref{BS} will
		occasionally be useful later, but for invertibility the
		\(O(\ve^{1/2})\) bound is sufficient.
		We now introduce the norms that will be used on \(S_R\). For
		\(h_{S_R}\in C^{2,\beta}(S_R)\), set
		\begin{equation}\label{norm-h-sphere}
			\begin{aligned}
				\|h_{S_R}\|_{**}
				:=
				R^{-1}\Big(
				&
				\|h_{S_R}\|_{L^\infty(S_R)}
				+
				R\|D_{S_R}h_{S_R}\|_{L^\infty(S_R)}
				\\
				&+
				R^2
				\|D^2_{S_R}h_{S_R}\|_{L^\infty(S_R)}
				+
				R^{2+\beta}
				[D^2_{S_R}h_{S_R}]_{\beta,S_R}
				\Big).
			\end{aligned}
		\end{equation}
		For \(g_{S_R}\in C^{0,\beta}(S_R)\), define
		\begin{equation}\label{norm-g-sphere}
			\|g_{S_R}\|
			:=
			R\|g_{S_R}\|_{L^\infty(S_R)}
			+
			R^{1+\beta}
			[g_{S_R}]_{\beta,S_R}.
		\end{equation}
		If
		\[
		h_{S_R}((R+d)e_3+R\omega)=Rh_S(\omega),
		\]
		then
		\begin{equation}\label{norm-scaling-h}
			\|h_{S_R}\|_{**}
			\approx
			\|h_S\|_{C^{2,\beta}(S)}.
		\end{equation}
		Likewise, if
		\[
		g_S(\omega)
		=
		g_{S_R}((R+d)e_3+R\omega),
		\]
		then
		\begin{equation}\label{norm-scaling-g}
			\|g_{S_R}\|
			=
			\|R g_S\|_{C^{0,\beta}(S)}.
		\end{equation}
		Thus the norms \eqref{norm-h-sphere} and
		\eqref{norm-g-sphere} are precisely the norms for which the
		rescaled spherical problem is uniform as \(R\to\infty\).
		
		\subsection{The Jacobi operator on the unit sphere}
		
		We next recall the solvability properties of
		\[
		J_S=\Delta_S+2
		\]
		in the axially symmetric class.
		
		A function \(u:S\to\R\) is called axially symmetric if it depends only
		on \(\omega_3\). The kernel of \(J_S\) (see \eqref{JacobiS}) on the full sphere is
		\[
		\ker J_S
		=
		\operatorname{span}
		\{\omega_1,\omega_2,\omega_3\}.
		\]
		Indeed, these are the first spherical harmonics and satisfy
		\[
		-\Delta_S\omega_j=2\omega_j.
		\]
		Restricting to axially symmetric functions leaves only the vertical
		translation mode:
		\begin{equation}\label{kernel-axis}
			\ker J_S
			\cap
			\{u=u(\omega_3)\}
			=
			\operatorname{span}\{\omega_3\}.
		\end{equation}
		
		Let
		\[
		P_S=(0,0,-1)
		\]
		denote the south pole of \(S\). We use the value of the solution at
		\(P_S\) to fix the freedom of adding multiples of \(\omega_3\).
		
		\begin{lemma}\label{unit-sphere-lemma}
			Let \(\beta\in(0,1)\), and let
			\[
			\bar g=\bar g(\omega_3)
			\in C^{0,\beta}(S)
			\]
			be axially symmetric. Assume
		$$
				\int_S
				\bar g\,\omega_3\,d\sigma
				=
				0.
			$$
			Then there exists a unique axially symmetric solution
			\(
			\bar h\in C^{2,\beta}(S)
			\)
			of
			$$
				J_S[\bar h]
				=
				\bar g
				\qquad\text{on }S
			$$
			satisfying
			\begin{equation}\label{unit-normalization}
				\bar h(P_S)=0.
			\end{equation}
			Moreover,
			\begin{equation}\label{unit-estimate}
				\|\bar h\|_{C^{2,\beta}(S)}
				\le
				C
				\|\bar g\|_{C^{0,\beta}(S)}.
			\end{equation}
		\end{lemma}
		
		\begin{proof}
			Since \(J_S\) is self-adjoint, the Fredholm alternative and
			\eqref{kernel-axis} show that the axially symmetric equation
			\[
			J_S\bar h=\bar g
			\]
			is solvable if and only if
			\[
			\int_S\bar g\,\omega_3\,d\sigma=0.
			\]
			Any two axially symmetric solutions differ by a multiple of
			\(\omega_3\). Since
			\(
			\omega_3(P_S)=-1,
			\)
			for every solution there exists a unique multiple of \(\omega_3\)
			which can be added so as to impose
			\[
			\bar h(P_S)=0.
			\]
			This proves existence and uniqueness under the normalization
			\eqref{unit-normalization}.
			
			It remains to prove the uniform estimate. Suppose, by contradiction,
			that \eqref{unit-estimate} fails. Then there exist axially symmetric
			functions \(\bar h_k,\bar g_k\) satisfying
			\[
			J_S\bar h_k=\bar g_k,
			\qquad
			\bar h_k(P_S)=0,
			\qquad
			\int_S\bar g_k\omega_3=0,
			\]
			such that
			\[
			\|\bar h_k\|_{C^{2,\beta}(S)}=1,
			\qquad
			\|\bar g_k\|_{C^{0,\beta}(S)}\to 0.
			\]
			By Schauder estimates and compactness, after passing to a subsequence,
			\(\bar h_k\) converges in \(C^2(S)\) to an axially symmetric function
			\(\bar h_\infty\) satisfying
			\[
			J_S\bar h_\infty=0,
			\qquad
			\bar h_\infty(P_S)=0.
			\]
			Hence
			\(
			\bar h_\infty=a\omega_3
			\)
			for some \(a\), and evaluation at \(P_S\) gives \(a=0\). Thus
			\(\bar h_\infty=0\), contradicting the normalization of
			\(\bar h_k\), together with the Schauder estimate. This proves
			\eqref{unit-estimate}.
		\end{proof}
		For an arbitrary right hand side we introduce a projection coefficient.
		
		\begin{corollary}[Projected inverse on the unit sphere]
			\label{projected-unit-sphere}
			Let
			\(
			\bar g=\bar g(\omega_3)
			\in C^{0,\beta}(S)
			\)
			be axially symmetric. Define
			\begin{equation}\label{c-unit-def}
				{\rm c}[\bar g]
				=
				-
				\frac{
					\displaystyle
					\int_S\bar g\,\omega_3\,d\sigma
				}{
					\displaystyle
					\int_S\omega_3^2\,d\sigma
				}.
			\end{equation}
			Then there exists a unique axially symmetric
			\(
			\bar h\in C^{2,\beta}(S)
			\)
			satisfying
			$$
				J_S[\bar h]
				=
				\bar g
				+
				{\rm c}[\bar g]\,\omega_3,
				\qquad
				\bar h(P_S)=0.
			$$
			Moreover,
			\begin{equation}\label{projected-unit-estimate}
				\|\bar h\|_{C^{2,\beta}(S)}
				+
				|{\rm c}[\bar g]|
				\le
				C
				\|\bar g\|_{C^{0,\beta}(S)}.
			\end{equation}
			We denote this linear solution operator by
			\begin{equation}\label{TS-def}
				\mathcal T_S[\bar g]
				=
				(\bar h,{\rm c}[\bar g]).
			\end{equation}
		\end{corollary}
		
		\begin{proof}
			By \eqref{c-unit-def},
			\[
			\int_S
			\bigl(
			\bar g+{\rm c}[\bar g]\omega_3
			\bigr)\omega_3\,d\sigma
			=
			0.
			\]
			Lemma~\ref{unit-sphere-lemma} therefore gives the unique normalized
			solution. Furthermore,
			\[
			|{\rm c}[\bar g]|
			\le
			C\|\bar g\|_{L^\infty(S)},
			\]
			and \eqref{projected-unit-estimate} follows from
			\eqref{unit-estimate}.
		\end{proof}
		
		We shall also use the elementary fact that every smooth axially
		symmetric function \(h=h(\omega_3)\) satisfies
		$$
			D_Sh(P_S)=0.
		$$
		Indeed, its Euclidean extension depends locally only on the third
		coordinate, and its Euclidean gradient at \(P_S\) is parallel to
		\(e_3\), whereas
		\(
		T_{P_S}S=e_3^\perp.
		\)
		
		\subsection{The projected problem on \(S_R\)}
		
		We now incorporate the perturbation \(B_{S_R}\). Given an axially
		symmetric
		\[
		g_{S_R}\in C^{0,\beta}(S_R),
		\]
		we consider
		\begin{equation}\label{JJP}
			\begin{aligned}
				J_{S_R}[h_{S_R}](\widetilde x)
				+
				B_{S_R}[h_{S_R}](\widetilde x)
				=
				g_{S_R}(\widetilde x)
				+
				{\rm c}\,
				\frac{\widetilde x_3-(R+d)}{R^2},
				\qquad
				\widetilde x\in S_R
			\end{aligned}
		\end{equation}
		Notice that
		\[
		\frac{\widetilde x_3-(R+d)}R
		=
		\omega_3,
		\]
(see \eqref{sphere-scaling})		so the last term in \eqref{JJP} becomes
		\({\rm c}\omega_3\) after multiplication by \(R\) and rescaling to
		the unit sphere.
		
		\begin{lemma}[Linear projected problem on \(S_R\)]
			\label{lemma-linear-sphere}
			Let \(\beta\in(0,1)\). For \(\delta>0\) fixed sufficiently small,
			there exist
			\(
			\ve_0>0,\) 
			\(
			C>0,
			\)
			such that, for every \(0<\ve<\ve_0\) and every axially symmetric
			\[
			g_{S_R}\in C^{0,\beta}(S_R),
			\]
			problem \eqref{JJP} has a unique axially symmetric solution
			\[
			(h_{S_R},{\rm c})
			\in
			C^{2,\beta}(S_R)\times\R
			\]
			satisfying the normalization
		$$
				h_{S_R}\bigl((R+d)e_3+RP_S\bigr)
				=
				0.
			$$
			Equivalently, if
			\[
			h_{S_R}((R+d)e_3+R\omega)
			=
			Rh_S(\omega),
			\]
			then
			\(
			h_S(P_S)=0.
			\)
			Moreover,
			\begin{equation}\label{esti-sphere}
				\|h_{S_R}\|_{**}
				+
				|{\rm c}|
				\le
				C
				\|g_{S_R}\|.
			\end{equation}
			We refer to \eqref{norm-h-sphere} and \eqref{norm-g-sphere} for the definitions of $\| \cdot \|_{**}$  and $\| \cdot \|$ respectively. The solution depends linearly on \(g_{S_R}\).
		\end{lemma}
		
		\begin{proof}
			Define
			\[
			g_S(\omega)
			=
			g_{S_R}((R+d)e_3+R\omega).
			\]
			Under the scaling
			\[
			h_{S_R}((R+d)e_3+R\omega)
			=
			Rh_S(\omega),
			\]
			equation \eqref{JJP} becomes
			\begin{equation}\label{JJP-unit}
				J_S[h_S]
				+
				B_S[h_S]
				=
				R g_S
				+
				{\rm c}\,\omega_3.
			\end{equation}
			Applying the projected inverse
			\(\mathcal T_S\) of Corollary~\ref{projected-unit-sphere}, we see that
			\eqref{JJP-unit} is equivalent to the fixed-point equation
			$$
				(h_S,{\rm c})
				=
				\mathcal T_S
				\left[
				R g_S-B_S[h_S]
				\right]
			$$
            where $\mathcal T_S$ is given in \eqref{TS-def}.
			Here the value of \({\rm c}\) is the projection coefficient produced
			by \(\mathcal T_S\); in particular, it is not prescribed in advance.
			Let
			\(
			C_0
			\)
			be the constant in
			\eqref{projected-unit-estimate}. From
			\eqref{BSS},
			\begin{align}
				\|h_S\|_{C^{2,\beta}(S)}
				+
				|{\rm c}|
				&\le
				C_0
				\left(
				\|R g_S\|_{C^{0,\beta}(S)}
				+
				\|B_S[h_S]\|_{C^{0,\beta}(S)}
				\right)
				\nonumber\\
				&\le
				C_0
				\|R g_S\|_{C^{0,\beta}(S)}
				+
				C_\delta\ve^{1/2}
				\|h_S\|_{C^{2,\beta}(S)}.
				\label{linear-sphere-apriori}
			\end{align}
			Choose \(\ve_0\) so small that
			\[
			C_\delta\ve^{1/2}
			\le
			\frac14
			\qquad
			\text{for }
			0<\ve<\ve_0.
			\]
			Set
			\[
			M
			=
			2C_0
			\|R g_S\|_{C^{0,\beta}(S)}
			\]
			and consider
			\[
			\mathcal X_M
			=
			\left\{
			h\in C^{2,\beta}(S):
			h=h(\omega_3),\
			h(P_S)=0,\
			\|h\|_{C^{2,\beta}(S)}
			\le M
			\right\}.
			\]
			If
			\[
			\mathcal A(h)
			=
			\pi_1
			\mathcal T_S
			\left[
			R g_S-B_S[h]
			\right],
			\]
			where \(\pi_1\) denotes the first component of \(\mathcal T_S\), then
			\eqref{linear-sphere-apriori} shows that
			\[
			\|\mathcal A(h)\|_{C^{2,\beta}(S)}
			\le
			\frac M2+\frac M4<M,
			\]
			so \(\mathcal A\) maps \(\mathcal X_M\) into itself.
			For \(h_1,h_2\in\mathcal X_M\),
			\begin{align*}
				\|\mathcal A(h_1)-\mathcal A(h_2)\|_{C^{2,\beta}(S)}
				&\le
				C_0
				\|B_S[h_1-h_2]\|_{C^{0,\beta}(S)}
				\le
				C_\delta\ve^{1/2}
				\|h_1-h_2\|_{C^{2,\beta}(S)}.
			\end{align*}
			After decreasing \(\ve_0\), if necessary, this factor is strictly
			smaller than one. Hence \(\mathcal A\) is a contraction.
			The Contraction Mapping Theorem yields a unique normalized solution
			\(h_S\). The corresponding coefficient \({\rm c}\) is uniquely
			determined by the second component of
			\(\mathcal T_S[Rg_S-B_S[h_S]]\).
			
			Finally,
			\[
			\|h_S\|_{C^{2,\beta}(S)}
			+
			|{\rm c}|
			\le
			C
			\|R g_S\|_{C^{0,\beta}(S)}.
			\]
			Using \eqref{norm-scaling-h} and \eqref{norm-scaling-g} gives
			\[
			\|h_{S_R}\|_{**}
			+
			|{\rm c}|
			\le
			C
			\|g_{S_R}\|,
			\]
			which proves \eqref{esti-sphere}.
		\end{proof}
		
		\begin{remark}
			The coefficient \({\rm c}\) introduced in
			Lemma~\ref{lemma-linear-sphere} is an auxiliary projection
			parameter. It should not be confused with the Lagrange multiplier
			\(\lambda_\ve[h]\) in the Euler--Lagrange equation.
			
			For the nonlinear problem, \(m\) will first be regarded as fixed. We
			shall solve the coupled neck and projected spherical problems and
			obtain
			\[
			h_C^1=h_C^1[m],
			\qquad
			h_{S_R}=h_{S_R}[m],
			\qquad
			{\rm c}={\rm c}[m].
			\]
			The final scalar equation
			\[
			{\rm c}[m]=0
			\]
			will determine the parameter \(m=m_\ve\).
			
			Observe also that, because \(B_S\) is present, the exact condition
			\({\rm c}=0\) is not simply
			\[
			\int_S g\,\omega_3\,d\sigma=0.
			\]
			Indeed, from \eqref{JJP-unit}, multiplication by \(\omega_3\) and
			integration over \(S\) give
			$$
				{\rm c}
				\int_S\omega_3^2\,d\sigma
				=
				\int_S
				B_S[h_S]\omega_3\,d\sigma
				-
				\int_S
				R g_S\,\omega_3\,d\sigma.
			$$
			Consequently,
			\[
			{\rm c}=0
			\]
			is equivalent to
			$$
				\int_S
				R g_S\,\omega_3\,d\sigma
				=
				\int_S
				B_S[h_S]\omega_3\,d\sigma.
		$$
			This exact identity will be used in the final balancing argument.
		\end{remark}

		\section{The reduced inner--outer system and the admissible set}
		\label{reduced-system}
		
		We now collect the equations that will be solved in the remaining
		part of the proof and specify the range of the parameters and unknowns.
		
		At this stage the approximate surface \(\Sigma_0^\ve\), the first neck
		correction \(h_C^0\), and the linear inverses on the neck and on the
		sphere have already been constructed. We write the full normal
		perturbation as
		$$
			h
			=
			\chi_C
			\bigl(
			h_C^0+h_C^1
			\bigr)
			+
			\chi_{S_\ve}h_{S_\ve}
		$$
		(see \eqref{defh}), where
		$$
			h_{S_\ve}(x)
			=
			h_{S_R}(\widetilde x),
			\qquad
			\widetilde x=\Phi(x)\in S_R.
		$$
		Equivalently, after scaling the sphere,
		$$
			h_{S_R}
			\bigl(
			(R+d)e_3+R\omega
			\bigr)
			=
			Rh_S(\omega),
			\qquad
			\omega\in S.
		$$
		Recall also that
		\[
		N_\Omega(x)
		=
		m\ve^3
		\int_\Omega\frac{dy}{|x-y|}
		\]
		(see \eqref{def-Newtonian}) and
		$$
			\bar\lambda_\ve[h;m]
			=
			\lambda_\ve[h;m]-2\ve
		$$
	(see \eqref{barlaep}),	where
		\begin{equation}\label{lambda-reduced}
			\lambda_\ve[h;m]
			=
			H_{\Sigma_{0,h}^\ve}(\widehat N)
			+
			N_{\Omega_{0,h}^\ve}(\widehat N).
		\end{equation}
		Thus \(\bar\lambda_\ve[h;m]\) is a scalar quantity, see \eqref{deflaep}. It is constant
		with respect to the spatial variable, but its value depends on \(m\)
		and on the complete perturbation \(h\). In particular, through the
		Coulomb term it depends also on \(h_C^1\). For this reason
		\(\bar\lambda_\ve[h;m]\) will be kept inside the fixed-point equation
		for the neck correction.
		
		\subsection{The neck equation}
		
		Using the equation satisfied by \(h_C^0\), 
        	the neck equation takes the form
		\begin{equation}\label{C1-reduced}
			\begin{aligned}
				J_{\Sigma_0^\ve}[h_C^1]
				={}&
				E_{01}
				+
				E_{03}
				+
				D^2_{\Sigma_0^\ve}h_C\,
				\mathcal Q_1[h,\nabla h]
				+
				\mathcal N_\ve[h]
				-
				\bar\lambda_\ve[h;m]
				\\
				&+
				(J_{\Sigma_0^\ve}[h_C^0]-E_{02}) 
				-
				2\nabla h_{S_\ve}\cdot\nabla\chi_{S_\ve}
				-
				h_{S_\ve}\Delta\chi_{S_\ve}, \quad 1<r<\ve^{-1+b},
			\end{aligned}
		\end{equation}
		for
		$
		h_C=h_C^0+h_C^1.
		$
		For later reference we denote the complete right hand side of
		\eqref{C1-reduced} by
$$
			\begin{aligned}
				g_C[m,h_{S_R},h_C^1]
				:={}&
				E_{01}
				+
				E_{03}
				+
				D^2_{\Sigma_0^\ve}h_C\,
				\mathcal Q_1[h,\nabla h]
				+
				\mathcal N_\ve[h]
				-
				\bar\lambda_\ve[h;m]
				\\
				&+
				(J_{\Sigma_0^\ve}[h_C^0]-E_{02})
				-
				2\nabla h_{S_\ve}\cdot\nabla\chi_{S_\ve}
				-
				h_{S_\ve}\Delta\chi_{S_\ve}.
			\end{aligned}
$$
		Thus \eqref{C1-reduced} is equivalent to
		$$
			h_C^1
			=
			\mathcal T_C
			\left[
			g_C[m,h_{S_R},h_C^1]
			\right],
		$$
		where $\mathcal T_C$ is the operator introduced in Lemma \ref{lemma-linear-neck}.
		
		\subsection{The spherical equation}
		
		On \(S_\ve\), after identifying the approximate surface with the
		sphere \(S_R\), we have
		\[
		J_{\Sigma_0^\ve}[h_{S_\ve}]
		=
		J_{S_R}[h_{S_R}]
		+
		B_{S_R}[h_{S_R}].
		\]
		The spherical equation is therefore
		$$
			J_{S_R}[h_{S_R}]
			+
			B_{S_R}[h_{S_R}]
			=
			H[m,h_{S_R},h_C^1]
			\qquad
			\text{on }S_R,
	$$
		where
		$$
			\begin{aligned}
				H[m,h_{S_R},h_C^1]
				={}&
				(1-\chi_C)
				\Bigl(
				E_{03}
				+
				\mathcal N_\ve[h]
				-
				\bar\lambda_\ve[h;m]
				\Bigr)
				+
				E_{04}
				\\
				&+
				\Big[
				D^2_{S_R}h_{S_R}
				+
				\mathcal B_\ve
				*
				D^2_{S_R}h_{S_R}
				+
				C_\ve
				*
				\nabla_{S_R}h_{S_R}
				\Big]
				\mathcal Q_1[h,\nabla h]
				\\
				&-
				2\nabla_{S_R}h_C^1
				\cdot
				\nabla_{S_R}\chi_C
				-
				h_C^1\Delta_{S_R}\chi_C
			\end{aligned}
	$$
		and
	$$
			E_{04}
			=
			-2\nabla_{S_R}h_C^0
			\cdot\nabla_{S_R}\chi_C
			-
			h_C^0\Delta_{S_R}\chi_C.
		$$
        We refer to \eqref{defchiC}, \eqref{calN} and \eqref{Hdef}. 
		Since the Jacobi operator on the sphere has the vertical translation
		mode in its kernel, we first solve the projected equation
		\begin{equation}\label{sphere-projected-reduced}
			\begin{aligned}
				J_{S_R}[h_{S_R}]
				+
				B_{S_R}[h_{S_R}]
				={}&
				H[m,h_{S_R},h_C^1]+
				{\rm c}\,
				\frac{\widetilde x_3-(R+d)}{R^2},
				\qquad
				\widetilde x\in S_R,
			\end{aligned}
		\end{equation}
		with $h_{S_R}
		\bigl(
		(R+d)e_3+R\omega
		\bigr)
		=
		Rh_S(\omega)$ and the normalization 
		$$
			h_S(P_S)=0,
			\qquad
			P_S=(0,0,-1).
		$$
		
		For fixed \(m\), the projected problem will determine
		\(h_{S_R}\) and the coefficient \({\rm c}\). The last step in the
		construction will consist of choosing \(m\) so that
		$
			{\rm c}[m]=0.
		$	
		\subsection{Choice of the parameters \(b\) and \(\beta\)}
		
		Fix $0<\beta<1.$
		We choose \(b>0\) sufficiently small so that
		\be \label{bbeta}
			b(\beta+2)<\frac14.
		\ee
		In particular,
		$
		1-2b-b\beta>\frac34.
		$
		This condition leaves enough separation between the various powers of
		\(\ve\) that occur in the neck and spherical estimates.
		
		We introduce the quantity
		\be \label{def-rhove}
			\rho_\ve
			:=
			\ve^{1-2b-b\beta}
			|\log\ve|.
		\ee
		Notice that
		$
			\rho_\ve\to 0
			\
			\text{as }\ve\to 0.
	$ This is the natural size of the spherical correction. Indeed, the
		principal spherical errors satisfy
		\[
		\|(1-\chi_C)E_{03}\|
		\lesssim
		\ve^{1-b\beta},
		\]
		whereas
		\[
		\|E_{04}\|
		\lesssim
		\ve^{1-2b-b\beta}
		|\log\ve|
		=
		\rho_\ve.
		\]
		The second term is therefore the dominant one in the norm relevant
		for the spherical linear theory.

		\subsection{Admissible range for the parameter $m$}
		
		The formal projection of the leading spherical errors gives
		\[
		-\frac{4\pi^2}{9}m\ve
		+
		2\pi\ve^2
		+
		O(\delta^4\ve^2)
		+
		o(\ve^2)
		=
		0.
		\]
		Thus the expected value of \(m\) is
		\[
		m
		=
		\frac{9}{2\pi}\ve
		\left(
		1+O(\delta^4)+o(1)
		\right).
		\]
		At this stage we only use the weaker consequence that
		\[
		m\approx \ve.
		\]
	Choose a fixed constant
		\[
		A\in(0,1)
		\]
		sufficiently small, independently of \(\ve\), and define
		\begin{equation}\label{Ieps}
			I_\ve
			:=
			\left[
			A\ve,A^{-1}\ve
			\right].
		\end{equation}
		Throughout the solution of the projected problem we assume
		$
			m\in I_\ve.
		$
		The constant \(A\) will be chosen once and for all so that the
		asymptotic value
		$
		\frac{9}{2\pi}\ve
		$
		lies in the interior of \(I_\ve\).
		
		\subsection{Admissible set for the spherical correction}
		
		Recall the norm
$$
			\begin{aligned}
				\|h_{S_R}\|_{**}
				=
				R^{-1}\Big(
				&
				\|h_{S_R}\|_{L^\infty(S_R)}
				+
				R\|D_{S_R}h_{S_R}\|_{L^\infty(S_R)}
				\\
				&+
				R^2\|D_{S_R}^2h_{S_R}\|_{L^\infty(S_R)}
				+
				R^{2+\beta}
				[D_{S_R}^2h_{S_R}]_{\beta,S_R}
				\Big)
			\end{aligned}
$$
as in \eqref{norm-h-sphere}.
		Under the scaling
		\[
		h_{S_R}((R+d)e_3+R\omega)
		=
		Rh_S(\omega),
		\]
		this norm is equivalent, uniformly in \(\ve\), to
		\(
		\|h_S\|_{C^{2,\beta}(S)}.
		\)
		
		For a constant \(M>1\), fixed independently of \(\ve\), define
		\begin{equation}\label{BS-eps}
			\mathcal B_\ve^S(M)
			=
			\left\{
			\begin{array}{l}
				h_{S_R}\in C^{2,\beta}(S_R):
				\ h_{S_R}\text{ is axially symmetric},
				\\[1mm]
				h_{S_R}((R+d)e_3+R\omega)=Rh_S(\omega),
				\quad
				h_S(P_S)=0,
				\\[1mm]
				\|h_{S_R}\|_{**}
				\le
				M\rho_\ve
			\end{array}
			\right\}.
		\end{equation}
		We shall eventually choose \(M\) sufficiently large but independent of $\ve$.
		For the definition of $\rho_\ve$ we refer to \eqref{def-rhove}.
		We assume that
		\begin{equation}\label{con2}
			h_{S_R}\in\mathcal B_\ve^S(M).
		\end{equation}
				The pointwise consequences of \eqref{con2} are
$$
			\begin{aligned}
				\|h_{S_R}\|_{L^\infty(S_R)}
				&\le
				CMR\rho_\ve,
				\quad
				\|D_{S_R}h_{S_R}\|_{L^\infty(S_R)}
				\le
				CM\rho_\ve,
				\\
				\|D^2_{S_R}h_{S_R}\|_{L^\infty(S_R)}
				&\le
				CMR^{-1}\rho_\ve
				=
				CM\ve\rho_\ve.
			\end{aligned}
	$$

		\subsection{Order of the reduction}
		
		The coupled problem will be solved in the following order.
		
		Fix
		\[
		(m,h_{S_R})\in I_\ve
			\times
			\mathcal B_\ve^S(M),
		\]
        where $I_\ve$ is defined in \eqref{Ieps} and $\mathcal B_\ve^S(M)$ in \eqref{BS-eps}.
		In the next section we solve \eqref{C1-reduced} for \(h_C^1\), obtaining
		a uniquely determined function
		\begin{equation}\label{hC1-dependence}
			h_C^1
			=
			h_C^1[m,h_{S_R}].
		\end{equation}
		We shall prove
			
            $$
            \|h_C^1[m,h_{S_R}]\|_*
			\le
			C\ve^{3/2},
		$$
		together with Lipschitz dependence on
		\((m,h_{S_R})\).
		
		We then substitute \eqref{hC1-dependence} into
		\eqref{sphere-projected-reduced}. For fixed \(m\), this gives a
		fixed-point problem for \(h_{S_R}\) in
		\(\mathcal B_\ve^S(M)\), together with a projection coefficient
		\[
		{\rm c}={\rm c}[m].
		\]
		Finally, we choose
		\(
		m=m_\ve\in I_\ve
		\)
		so that
		\(
		{\rm c}[m_\ve]=0.
		\)
		
		Thus the nonlinear construction has only one final scalar degree of
		freedom, namely the parameter \(m\). The quantity
		\(\bar\lambda_\ve[h;m]\) is not an additional parameter: once
		\(m\) and the perturbation \(h\) have been determined, its value is
		fixed by \eqref{lambda-reduced}.

		\section{Solving the problem in the neck}
		\label{final1}
		
		We now solve the first equation in the reduced inner--outer system.
		Throughout this section the parameter \(m\) and the spherical correction
		\(h_{S_R}\) are regarded as given and satisfy
		\[
		m\in I_\ve,
		\qquad
		h_{S_R}\in\mathcal B_\ve^S(M),
		\]
		where \(I_\ve\) and \(\mathcal B_\ve^S(M)\) were defined in
		\eqref{Ieps} and \eqref{BS-eps}. Thus
		$$
			A\ve\le m\le A^{-1}\ve,
			\qquad
			\|h_{S_R}\|_{**}\le M\rho_\ve,
	$$
		with
		$$
			\rho_\ve
			=
			\ve^{1-2b-b\beta}|\log\ve|.
		$$
		The unknown in this section is the second neck correction \(h_C^1\).
		Recall that
		$$
			h
			=
			\chi_C(h_C^0+h_C^1)
			+
			\chi_{S_\ve}h_{S_\ve},
		$$
		where
		\[
		h_{S_\ve}(x)
		=
		h_{S_R}(\widetilde x)
		=
		Rh_S(\omega),
		\qquad
		\omega
		=
		\frac{\widetilde x-(R+d)e_3}{R}.
		\]
		
		The equation to be solved is
		\begin{equation}\label{C1-neck-final}
			J_{\Sigma_0^\ve}[h_C^1]
			=
			g_C[m,h_{S_R},h_C^1],
			\qquad
			1<r<\ve^{-1+b},
		\end{equation}
		where
		\begin{equation}\label{gC-neck-final}
			\begin{aligned}
				g_C[m,h_{S_R},h_C^1]
				={}&
				E_{01}
				+
				E_{03}
				+
				D^2_{\Sigma_0^\ve}h_C\,
				\mathcal Q_1[h,\nabla h]
				+
				\mathcal N_\ve[h]
				-
				\bar\lambda_\ve[h;m]
				\\
				&+
				J_{\Sigma_0^\ve}[h_C^0]-E_{02}
				-
				2\nabla h_{S_\ve}\cdot\nabla\chi_{S_\ve}
				-
				h_{S_\ve}\Delta\chi_{S_\ve},
			\end{aligned}
		\end{equation}
		and
		\[
		h_C=h_C^0+h_C^1.
		\]
		The boundary conditions are the regularity conditions at the waist,
		\begin{equation}\label{hC1-boundary}
			h_C^1(1)=0,
			\qquad
			\lim_{r\to1^+}
			\sqrt{r^2-1}\,(h_C^1)'(r)=0.
		\end{equation}

		\subsection{Norms and preliminary pointwise estimates}
		
		Recall the neck norm
		$$
				\|h\|_*
				:=
				\left\|
				\frac{h}{r(r-1)}
				\right\|_{L^\infty(1,\ve^{-1+b})}
				+
				\left\|
				\frac{h'}{r}
				\right\|_{L^\infty(1,\ve^{-1+b})}
+
				\left\|
				\min\{r-1,1\}\,h''
				\right\|_{L^\infty(1,\ve^{-1+b})}
$$
defined in \eqref{norma*}.
Consequently	$$
			\begin{array}{lll}
				1<r<2:
				&
				|h|\lesssim(r-1)\|h\|_*,
				&
				|h'|\lesssim\|v\|_*,
				\qquad
				|h''|\lesssim(r-1)^{-1}\|h\|_*,
				\\[1mm]
				r\ge2:
				&
				|h|\lesssim r^2\|h\|_*,
				&
				|h'|\lesssim r\|h\|_*,
				\qquad
				|h''|\lesssim\|h\|_*.
			\end{array}
		$$
		
		For the spherical function we have
		\begin{equation}\label{sphere-pointwise-neck}
			\begin{aligned}
				\|h_{S_R}\|_\infty
				&\le
				R\|h_{S_R}\|_{**},
				\quad 
				\|D_{S_R}h_{S_R}\|_\infty
				\le
				\|h_{S_R}\|_{**},
				\quad
				\|D_{S_R}^2h_{S_R}\|_\infty
				\le
				\ve\|h_{S_R}\|_{**}.
			\end{aligned}
		\end{equation}
		
		There is an additional improvement near the south pole of the
		sphere. Since
		\[
		h_S(P_S)=0
		\]
		and \(h_S\) is smooth and axially symmetric, its first derivative
		vanishes at \(P_S\). Therefore, in the coordinates used in the
		overlap region,
		\begin{equation}\label{sphere-pole-improved}
			|h_{S_R}|
			\lesssim
			\ve r^2\|h_{S_R}\|_{**},
			\qquad
			|D_{S_R}h_{S_R}|
			\lesssim
			\ve r\|h_{S_R}\|_{**}.
		\end{equation}
		This estimate will be crucial for the terms in which derivatives fall
		on the cutoff functions.
		
		\subsection{Expansion of the Coulomb term}
		
		We shall repeatedly use the following shape expansion.

		\begin{lemma}\label{lemma-shape-potential-neck}
			Let \(\Sigma=\partial\Omega\), and let
			\[
			\Sigma_h
			=
			\{
			\hat x= x+h(x)\nu(x):
			x\in\Sigma
			\}
			\]
			be a sufficiently small normal graph enclosing $\Omega_h$. Then for $x \in \Sigma$
			\begin{equation}\label{shape-N-exp}
				N_{\Omega_h} (\hat x)-N_\Omega (x)
				=
				m\ve^3 \left(
				\int_\Sigma
				\frac{h(\sigma)}{|x-\sigma|}
				\,d\sigma
				+\nabla_x N_\Omega (x) \cdot \nu (x) \, h (x) \right) +
				\ttt N_\ve[h](x),
			\end{equation}
			where
			\(
			\ttt N_\ve [0]
			=
			0.
			\)
			Moreover, for any $\alpha \in (0,1)$
			\begin{equation}\label{shape-N-quadratic}
				\begin{aligned}
					|\ttt N_\ve[h](x)|
					\lesssim{}&
	m\ve^3 |h(x)|^{1+\alpha} + m\ve^3 \int_\Sigma |h(\sigma)|^2 
					\sup_{|t|\le |h(\sigma)|}
					\frac{|A(\sigma)|}
					{|x-\sigma-t\nu(\sigma)|}
					d\sigma.
				\end{aligned}
			\end{equation}
			For two sufficiently small graphs \(h_1,h_2\), the remainder
			satisfies the corresponding quadratic Lipschitz estimate
			\begin{equation}\label{shape-N-difference}
				\begin{aligned}
					&
					|\ttt  N_\ve [h_1](x)
					-
					\ttt N_\ve [h_2](x)|
			\lesssim
					C \, m\ve^3
					\bigl(
					\|h_1\|^\alpha_\infty+\|h_2\|^\alpha_\infty
					\bigr)
					\|h_1-h_2\|_\infty.
				\end{aligned}
			\end{equation}
		\end{lemma}

        \medskip
A proof of this result can be found in Theorem 3.3 of \cite{BonaciniCristoferi2014}. For completeness, we provide our own proof below.
        \medskip
        
		\begin{proof}
        For $x \in \Sigma$, we write
        \begin{align*}
(m \ve^3 )^{-1} &\left( N_{\Omega_h} (\hat x)-N_\Omega (x)
				\right) = I_1 (x) + I_2 (x) \,, \\
                I_1(x)&= \int_{\Omega_h}
                ({1\over |x+ h(x) \nu (x) -y|}
            - {1\over |x -y|} ) dy \,, \quad I_2 (x)=  \int_{\Omega_h}  {dy \over |x-y|} - \int_\Omega {dy \over |x-y|} \,.
        \end{align*}
Define
\[
g(t)
=
\int_{\Omega_h}
\frac{dy}{|x+t h(x)\nu(x)-y|},
\qquad 0\le t\le1.
\]
Then
\[
I_1 (x)
=
g(1)-g(0).
\]
Since the Coloumb potential $z \to U_{\Omega_h} (z):=\int_{\Omega_h} {dy \over |z-y|}$ is $C^{1, \alpha}$ for all $\alpha \in (0,1)$, we have
$$
g'(t) = h(x) \nabla_x U_{\Omega_h} (x+ t h (x) \nu (x) ) \cdot \nu (x)
$$
and
\begin{align*}
|g(1)  - g(0) - g'(0) | &\leq \int_0^1 | g'(s) - g'(0)| ds \\
&\leq  |h(x)| \int_0^1 |\nabla_x U_{\Omega_h} (x+ sh (x) \nu (x) )  - \nabla_x U_{\Omega_h} (x)| \, ds  
\leq C_\alpha |h(x)|^{1+ \alpha}. 
\end{align*}
Hence
\begin{align*}
    I_1(x) &= h(x) \nabla_x U_{\Omega_h} (x) \cdot \nu (x) + O\left( |h(x)|^{1+ \alpha} \right)
\end{align*}
for all $\alpha \in (0,1).$ On the other hand,
\begin{align*}
    \nabla_x U_{\Omega_h} (x)&=  \nabla_x U_{\Omega} (x) - \int_{\Omega_h} {x-y \over |x-y|^3} dy +  \int_{\Omega} {x-y \over |x-y|^3} dy.
\end{align*}
Hence
\be \label{Iuno}
I_1(x) = h(x) \nabla_x U_{\Omega} (x) \cdot \nu (x) - h(x)[ \int_{\Omega_h} {(x-y) \cdot \nu (x) \over |x-y|^3} dy -  \int_{\Omega} {(x-y) \cdot \nu (x)  \over |x-y|^3} dy]+ O\left( |h(x)|^{1+ \alpha} \right).
\ee
For small $h$, we parametrize the tiny region between $\Omega_h$ and $\Omega$  with the change of variable
\be \label{cchanges}
y= \sigma + t  \nu (\sigma) , \quad \sigma \in \Sigma, 
\ee
and $0 \leq t < h(\sigma)$ or $h(\sigma) < t \leq 0$ (depending on the sign of $h$), whose Jacobian is
$$
J = \det(I+tA)
			=
			1+tH_\Sigma+t^2 K_\Sigma, \quad K_\Sigma:= \det A.
$$
Thus
\begin{align*}
    \int_{\Omega_h} {(x-y) \cdot \nu (x)  \over |x-y|^3} dy &-  \int_{\Omega} {(x-y) \cdot \nu (x) \over |x-y|^3} dy \\
    &= \int_\Sigma \int_0^{h(\sigma)} {(x-\sigma - t h(\sigma) \nu (\sigma) ) \cdot \nu (x) \over |x-\sigma - t h(\sigma) \nu (\sigma) |^3} \, [1+tH_\Sigma (\sigma) +t^2 K_\Sigma (\sigma) ] \, dt d\sigma .
\end{align*}
The integral is absolutely integrable, despite the apparent singularity at $(\sigma , t)= (x,0)$. Indeed, for \(\sigma\) close to \(x\), smoothness of \(\Sigma\) gives
$$
|(x- \sigma) \cdot \nu (x)| \leq C |x-\sigma|^2, \quad |\nu (\sigma) \cdot \nu (x) | \leq 1
$$
and $ |x-\sigma - t \nu (\sigma)| \geq C \sqrt{ |x-\sigma|^2 + t^2}$ is a small tubular neighborhood. Using
			\[
			|H_\Sigma|\lesssim |A|,
			\qquad
			|K_\Sigma|\lesssim |A|^2,
			\]
we get
\begin{align*}
|\int_{\Omega_h} {(x-y) \cdot \nu (x)  \over |x-y|^3} dy &-  \int_{\Omega} {(x-y) \cdot \nu (x) \over |x-y|^3} dy | \leq \| h \|_\infty \int_\Sigma    \sup_{|t|\le |h(\sigma)|}
					\frac{|A(\sigma)|}
					{|x-\sigma-t\nu(\sigma)|} d \sigma.
\end{align*}
Inserting this estimate in \eqref{Iuno} we get
$$
I_1(x) = h(x) \nabla_x U_{\Omega} (x) \cdot \nu (x) + O\left( |h(x)|^{1+ \alpha} \right) + O \left( \| h \|^2_\infty \int_\Sigma    \sup_{|t|\le |h(\sigma)|}
					\frac{|A(\sigma)|}
					{|x-\sigma-t\nu(\sigma)|} d \sigma \right). 
$$
			
            Consider now $I_2$. Using again the change of variable \eqref{cchanges},
			\[
			\begin{aligned}
				I_2(x) = 
				\int_\Sigma
				\int_0^{h(\sigma)}
				\frac{
					1+tH_\Sigma(\sigma)+t^2 K_\Sigma (\sigma)
				}{
					|x-\sigma-t\nu(\sigma)|
				}
				\,dt\,d\sigma.
			\end{aligned}
			\]
			Adding and subtracting
			\(
			\int_\Sigma
			\int_0^{h(\sigma)}
			\frac{dt}{|x-\sigma|}
			\,d\sigma
			\)
			gives the linear term in \eqref{shape-N-exp}.
			For the remainder, the fundamental theorem of calculus gives
			\[
			\frac1{|x-\sigma-t\nu|}
			-
			\frac1{|x-\sigma|}
			=
			\int_0^t
			\frac{
				(x-\sigma-s\nu)\cdot\nu
			}{
				|x-\sigma-s\nu|^3
			}\,ds.
			\]
			Together with the smallness of \(|A||v|\), this gives
			\eqref{shape-N-quadratic}.
			Finally, apply the mean value theorem along
			\[
			h_\tau
			=
			h_2+\tau(h_1-h_2),
			\qquad
			0\le\tau\le1.
			\]
			Since the derivative of the remainder vanishes at \(h=0\), one
			obtains one factor measuring the size of \(h_1\) or \(h_2\) and one
			factor \(h_1-h_2\), which gives \eqref{shape-N-difference}.
		\end{proof}
		
		\subsection{Existence and size of the second neck correction}
		
		We can now solve \eqref{C1-neck-final}.
		
		\begin{propositio}\label{outer}
			Assume
			\[
			m\in I_\ve,
			\qquad
			h_{S_R}\in\mathcal B_\ve^S(M),
			\]
			and let
			\[
			0<\beta<1,
			\qquad
			b(\beta+2)<\frac14.
			\]
			There exist constants
			\[
			\ve_0>0,
			\qquad
			a_*>0,
			\]
			independent of \(m\) and \(h_{S_R}\) in the above admissible sets,
			such that, for every \(0<\ve<\ve_0\), problem
			\eqref{C1-neck-final}--\eqref{hC1-boundary} has a unique solution
			\[
			h_C^1=h_C^1[m,h_{S_R}]
			\]
			satisfying
			$$
				\|h_C^1[m,h_{S_R}]\|_*
				\le
				a_*\ve^{3/2}.
	$$
		\end{propositio}
		
		\begin{proof}
			By Lemma \ref{lemma-linear-neck}, problem
			\eqref{C1-neck-final} is equivalent to
			$$
				h_C^1
				=
				\mathcal A_C(h_C^1),
				\qquad
				\mathcal A_C(h_C^1)
				:=
				\mathcal T_C
				\bigl(
				g_C[m,h_{S_R},h_C^1]
				\bigr).
			$$
			We work in the closed ball
			$$
				B_*
				=
				\left\{
				v:
				\|v\|_*
				\le
				a_*\ve^{3/2}
				\right\}.
			$$
			
			We prove first that \(\mathcal A_C\) maps \(B_*\) into itself, and
			then that it is a contraction.
			
			\medskip
			
			\noindent
			{\bf Step 1. The fixed errors.}
			
			From the estimates for the approximate surface,
			\begin{equation}\label{est1-neck}
				\|E_{01}\|_{L^\infty(1,\ve^{-1+b})}
				\le
				a_1\ve^{3/2},
				\qquad
				\|E_{03}\|_{L^\infty(1,\ve^{-1+b})}
				\lesssim
				\ve^2.
			\end{equation}
			Moreover, by the construction of \(h_C^0\),
			\begin{equation}\label{est10-neck}
				\|
				J_{\Sigma_0^\ve}[h_C^0]-E_{02}
				\|_{L^\infty(1,\ve^{-1+b})}
				\lesssim
				\ve^2|\log\delta|.
			\end{equation}
			Here and below \(\delta>0\) is fixed before \(\ve\) is chosen.
			
			\medskip
			
			\noindent
			{\bf Step 2. Estimate of \(\bar\lambda_\ve[h;m]\).}
			
			We claim that, for \(h_C^1\in B_*\),
			\begin{equation}\label{est-lambda-neck}
				|\bar\lambda_\ve[h;m]|
				\lesssim
				\ve^2
				+
				\ve^{1+2b}\|h_C^1\|_*
				+
				\ve\|h_{S_R}\|_{**}.
			\end{equation}
			
			Recall that
			\[
			\bar\lambda_\ve[h;m]
			=
			\lambda_\ve[h;m]-2\ve, \quad 
			\lambda_\ve[h;m]
			=
			H_{\Sigma_{0,h}^{\ve,+}}(\widehat N)
			+
			N_{\Omega_{0,h}^\ve}^{(m)}(\widehat N).
			\]
			The north pole of the reference sphere is
			\[
			N_S=(0,0,2R+d).
			\]
			Since
			\[
			\chi_C(N_S)=0,
			\qquad
			\chi_{S_\ve}(N_S)=1,
			\]
			we have
			\[
			h(N_S)=h_{S_R}(N_S)
			\]
			and therefore
			\[
			\widehat N
			=
			N_S+h_{S_R}(N_S)e_3.
			\]
			
			At this point the reference surface is exactly spherical. Hence
			\[
			H_{\Sigma_{0,h}^{\ve,+}}(\widehat N)
			=
			2\ve
			-
			J_{S_R}[h_{S_R}](N_S)
			+
			\mathcal Q_{S_R}[h_{S_R}](N_S).
			\]
			By \eqref{sphere-pointwise-neck},
			\[
			|J_{S_R}[h_{S_R}](N_S)|
			\lesssim
			\ve\|h_{S_R}\|_{**}.
			\]
			The quadratic estimate for the mean-curvature remainder gives
			\[
			|\mathcal Q_{S_R}[h_{S_R}](N_S)|
			\lesssim
			\ve\|h_{S_R}\|_{**}^2.
			\]
			Since
			\[
			\|h_{S_R}\|_{**}\le M\rho_\ve=o(1),
			\]
			the last term is absorbed in
			\(
			C\ve\|h_{S_R}\|_{**}.
			\)
			We next estimate the Coulomb contribution. For the reference
			configuration,
			$$
				N_{\Omega_0^\ve}(\widehat N)
				=
				O(\ve^2),
		$$
			uniformly for \(m\in I_\ve\).
			By Lemma \ref{lemma-shape-potential-neck},
			$$
				\begin{aligned}
					N_{\Omega_{0,h}^\ve}(\widehat N)
					-
					N_{\Omega_0^\ve}(\widehat N)
					={}&
					m\ve^3
					\int_{\Sigma_0^\ve}
					\frac{h(\sigma)}
					{|\widehat N-\sigma|}
					\,d\sigma
					+ m \ve^3 \, \nabla_x N_{\Omega_0^\ve} (\hat N) \cdot \nu (\hat N) \, h (\hat N)+
					\ttt  N_\ve [h](\widehat N).
				\end{aligned}
			$$
			We decompose
			\[
			h
			=
			\chi_Ch_C^0
			+
			\chi_Ch_C^1
			+
			\chi_{S_\ve}h_{S_\ve}.
			\]
			Using \(m\sim\ve\), the geometry of the reference surface, and
			\[
			\int_{S_R}
			\frac{d\sigma}{|\widehat N-\sigma|}
			\lesssim R,
			\]
			we obtain
			\begin{align*}
				m\ve^3
				\left|
				\int_{\Sigma_0^\ve}
				\frac{\chi_Ch_C^0}
				{|\widehat N-\sigma|}
				\,d\sigma
				\right|
				&\lesssim
				\ve^3|\log\ve|,
				\quad 
				m\ve^3
				\left|
				\int_{\Sigma_0^\ve}
				\frac{\chi_Ch_C^1}
				{|\widehat N-\sigma|}
				\,d\sigma
				\right|
				\lesssim
				\ve^{1+2b}\|h_C^1\|_*,
				\\
				m\ve^3
				\left|
				\int_{\Sigma_0^\ve}
				\frac{\chi_{S_\ve}h_{S_\ve}}
				{|\widehat N-\sigma|}
				\,d\sigma
				\right|
				&\lesssim
				\ve^2\|h_{S_R}\|_{**}.
			\end{align*}
            Besides
            \begin{align*}
                m\ve^3 | \nabla_x N_{\Omega_0^\ve} (\hat N) \cdot \nu (\hat N) \, h (\hat N) | \lesssim m \ve^3 \, R \, \| h_{S_R} \|_\infty \lesssim \ve^2 \|h_{S_R}\|_{**}.
            \end{align*}
			For the quadratic remainder, we use
			\[
			\|\chi_Ch_C^0\|_\infty
			\lesssim|\log\ve|,
			\quad
			\|\chi_Ch_C^1\|_\infty
			\lesssim
			\ve^{-2+2b}\|h_C^1\|_*,
			\quad 
			\|h_{S_\ve}\|_\infty
			\lesssim
			R\|h_{S_R}\|_{**}.
			\]
			The kernel in \eqref{shape-N-quadratic} is uniformly integrable at
			the north pole, up to harmless logarithmic factors. From \eqref{shape-N-quadratic} we get that, for any $\alpha \in (0,1)$
			$$
				|\ttt  N_\ve [h](\widehat N)|
				\lesssim
				\ve^4|\log\ve|^{1+ \alpha}
				+
				\ve^{2 (1-\alpha) + 2b (1+ \alpha)}\|h_C^1\|_*^{1+\alpha} 
				+
				\ve^{3-\alpha}\|h_{S_R}\|_{**}^{1+\alpha}.
			$$
			Since
			\[
			\|h_C^1\|_*
			\le
			a_*\ve^{3/2},
			\qquad
			\|h_{S_R}\|_{**}
			\le
			M\rho_\ve=o(1),
			\]
			all these terms are of lower order than the right hand
			side of \eqref{est-lambda-neck}. This proves
			\eqref{est-lambda-neck}.
			
			Notice that this estimate has been obtained without freezing
			\(\bar\lambda_\ve\): its dependence on \(h_C^1\) has been retained
			throughout.
			
			\medskip
			
			\noindent
			{\bf Step 3. The local nonlinear mean-curvature term.}
			
			Write
			\be \label{def-Nloc}
				\mathcal N_{\ve,\mathrm{loc}}(h)
				=
				D^2_{\Sigma_0^\ve}h_C\,
				\mathcal Q_1[h,\nabla h]
				+
				\mathcal Q_2[h,\nabla h].
			\ee
            The pointwise estimates for the curvature expansion give
			\begin{equation}\label{Q1-neck}
				|\mathcal Q_1[h,\nabla h]|
				\lesssim
				|A||h|
				+
				|\nabla h|^2
			\end{equation}
			and
			\begin{equation}\label{Q2-neck}
				|\mathcal Q_2[h,\nabla h]|
				\lesssim
				|A|^3|h|^2
				+
				|A||\nabla h|^2
				+
				|h||\nabla A||\nabla h|.
			\end{equation}
			
			We estimate these terms separately in three regions.
			
			\smallskip
			
			\noindent
			{\it Region I: \(1<r<\ve^{-1/2}\).}
			In this region,
			\[
			h_C^0=0,
			\qquad
			h_{S_\ve}=0,
			\]
			and hence
			\[
			h=h_C^1.
			\]
			Using $h^1_C \in B_*$, together with the catenoidal estimates
			for \(A\) and \(\nabla A\), we obtain
			\[
			|\mathcal N_{\ve,\mathrm{loc}}(h)|
			\lesssim
			\ve^{-1/2}\|h_C^1\|_*^2.
			\]
			Since \(b<1/4\),
			\(
			\ve^{-1/2}
			\le
			\ve^{-1+2b},
			\)
			and therefore
			$$
				|\mathcal N_{\ve,\mathrm{loc}}(h)|
				\lesssim
				\ve^{-1+2b}\|h_C^1\|_*^2.
		$$
			
			\smallskip
			
			\noindent
			{\it Region II:
				\(\ve^{-1/2}<r<\frac{\delta}{4}\ve^{-3/4}\).}
			Again,
			\[
			h_C^0=0,
			\qquad
			h_{S_\ve}=0.
			\]
			In this range the approximate surface is already close to its
			large-\(r\) catenoidal geometry and
			\[
			|A|\lesssim\ve,
			\qquad
			|\nabla A|\lesssim\frac{\ve}{r}.
			\]
			Furthermore,
			\[
			|h_C^1|
			\lesssim
			r^2\|h_C^1\|_*,
			\qquad
			|(h_C^1)'|
			\lesssim
			r\|h_C^1\|_*,
			\qquad
			|(h_C^1)''|
			\lesssim
			\|h_C^1\|_*.
			\]
			Substitution in \eqref{Q1-neck}--\eqref{Q2-neck} yields
			$$
				|\mathcal N_{\ve,\mathrm{loc}}(h)|
				\lesssim
				\ve^{-1+2b}\|h_C^1\|_*^2.
			$$
			
			\smallskip
			
			\noindent
			{\it Region III:
				\(\frac{\delta}{4}\ve^{-3/4}
				<r<\ve^{-1+b}\).}
			This is the overlap region. Here \(h_C^0\), \(h_C^1\), and the
			spherical correction may all be present. We use
			$$
				|h_C^0|
				\lesssim
				1+
				\left|
				\log\frac{r}{\delta\ve^{-3/4}}
				\right|,
				\qquad
				|(h_C^0)'|
				\lesssim\frac1r,
				\qquad
				|(h_C^0)''|
				\lesssim\frac1{r^2},
			$$
			together with
			\[
			|h_C^1|
			\lesssim r^2\|h_C^1\|_*,
			\qquad
			|(h_C^1)'|
			\lesssim r\|h_C^1\|_*,
			\qquad
			|(h_C^1)''|
			\lesssim\|h_C^1\|_*,
			\]
			and \eqref{sphere-pointwise-neck}--\eqref{sphere-pole-improved}.
			In this region,
			\[
			|A|\lesssim\ve,
			\qquad
			|\nabla A|\lesssim\frac{\ve}{r}.
			\]
			The contribution involving only the fixed correction \(h_C^0\) is
			bounded by
			\(
			C\ve^{5/2}|\log\ve|.
			\)
			Terms containing one factor \(h_C^1\) and one background factor
			coming from \(h_C^0\) satisfy
			\(
			C\ve^{1+2b}\|h_C^1\|_*.
			\)
			The terms quadratic in \(h_C^1\) are bounded by
			\(
			C\ve^{-1+2b}\|h_C^1\|_*^2.
			\)
			Finally, the terms involving the spherical correction satisfy
			\[
			C\ve^2\|h_{S_R}\|_{**}
			+
			C\|h_{S_R}\|_{**}^2.
			\]
			Consequently,
			\begin{equation}\label{Nloc-total}
				\begin{aligned}
					\|\mathcal N_{\ve,\mathrm{loc}}(h)\|_\infty
					\lesssim{}&
					\ve^{5/2}|\log\ve|
					+
					\ve^{1+2b}\|h_C^1\|_*
					+
					\ve^{-1+2b}\|h_C^1\|_*^2
					+
					\ve^2\|h_{S_R}\|_{**}
					+
					\|h_{S_R}\|_{**}^2.
				\end{aligned}
			\end{equation}
			
			\medskip
			
			\noindent
			{\bf Step 4. The nonlinear Coulomb term.}
			
			We write
			\be \label{def-NC}
				\mathcal N_{\ve,\mathrm{Coulomb}}(h)
				=
				N_{\Omega_{0,h}^\ve} (\hat x) 
				-
				N_{\Omega_0^\ve}(x).
			\ee
			By Lemma \ref{lemma-shape-potential-neck},
			\[
			\mathcal N_{\ve,\mathrm{Coulomb}}(h)(x)
			=
			m\ve^3 \left(
				\int_\Sigma
				\frac{h(\sigma)}{|x-\sigma|}
				\,d\sigma
				+\nabla_x N_\Omega (x) \cdot \nu (x) \, h (x) \right) +
				\ttt N_\ve[h](x).
			\]
			The linear part satisfies
			\begin{equation}\label{Nnewt-linear}
				\begin{aligned}
					m\ve^3
					\left|
					\int_{\Sigma_0^\ve}
					\frac{h(\sigma)}
					{|x-\sigma|}
					\,d\sigma
					\right|
					\lesssim{}&
					\ve^3|\log\ve|
					+
					\ve^{1+2b}\|h_C^1\|_*
					+
					\ve^2\|h_{S_R}\|_{**}\\
m \ve^3 \left| \nabla_x N_\Omega (x) \cdot \nu (x) \, h (x) \right| & \lesssim \ve^3|\log\ve|
					+
					\ve^{1+2b}\|h_C^1\|_*
					+
					\ve^2\|h_{S_R}\|_{**}
				\end{aligned}
			\end{equation}
			The quadratic remainder satisfies
			\begin{equation}\label{Nnewt-quad}
				|\ttt N_\ve[h](x)|
				\lesssim
				\ve^4|\log\ve|^{1+ \alpha}
				+
				\ve^{2 (1-\alpha) + 2b (1+ \alpha)}\|h_C^1\|_*^{1+\alpha} 
				+
				\ve^{3-\alpha}\|h_{S_R}\|_{**}^{1+\alpha}
			\end{equation}
            for any $\alpha \in (0,1)$.
			Since $b \in (0,{1\over 4})$,
			 combining \eqref{Nloc-total},
			\eqref{Nnewt-linear}, and \eqref{Nnewt-quad},
			we obtain \begin{equation}\label{est-nonlinear-neck}
				\begin{aligned}
					\|
					\mathcal N_{\ve,\mathrm{loc}}(h)
					+
					\mathcal N_{\ve,\mathrm{Coulomb}}(h)
					\|_{L^\infty}
					&\lesssim
					\ve^{5/2}|\log\ve|
					+
					\ve^{1+2b}\|h_C^1\|_*
					+
					\ve^2\|h_{S_R}\|_{**}
					\\ &+
					\ve^{-1+2b}\|h_C^1\|_*^2
					+
					\|h_{S_R}\|_{**}^2.
				\end{aligned}
			\end{equation}
			
			\medskip
			
			\noindent
			{\bf Step 5. Interaction with the spherical cutoff.}
			We claim
			\begin{equation}\label{est-cutoff-sphere-neck}
				\left\|
				2\nabla h_{S_\ve}\cdot\nabla\chi_{S_\ve}
				+
				h_{S_\ve}\Delta\chi_{S_\ve}
				\right\|_\infty
				\lesssim
				\ve\|h_{S_R}\|_{**}.
			\end{equation}
			Indeed, derivatives of \(\chi_{S_\ve}\) are supported where
			$
			r\sim\ve^{-3/4}.
			$ See \eqref{chiS-derivatives}.
			Hence
			\[
			|\nabla\chi_{S_\ve}|
			\lesssim
			\ve^{3/4},
			\qquad
			|\Delta\chi_{S_\ve}|
			\lesssim
			\ve^{3/2}.
			\]
			On the other hand, by \eqref{sphere-pole-improved},
			\[
			|h_{S_R}|
			\lesssim
			\ve r^2\|h_{S_R}\|_{**},
			\qquad
			|\nabla h_{S_R}|
			\lesssim
			\ve r\|h_{S_R}\|_{**}.
			\]
			At \(r\sim\ve^{-3/4}\), this gives
			\[
			|h_{S_R}|
			\lesssim
			\ve^{-1/2}\|h_{S_R}\|_{**},
			\quad
			|\nabla h_{S_R}|
			\lesssim
			\ve^{1/4}\|h_{S_R}\|_{**}.
			\]
			Therefore
			\[
			|\nabla h_{S_\ve}\cdot\nabla\chi_{S_\ve}|
			\lesssim
			\ve\|h_{S_R}\|_{**},
			\quad 
			|h_{S_\ve}\Delta\chi_{S_\ve}|
			\lesssim
			\ve\|h_{S_R}\|_{**},
			\]
			which proves \eqref{est-cutoff-sphere-neck}.
			
			\medskip
			
			\noindent
			{\bf Step 6. The map preserves the ball.}
			
			Let \(\mathcal C\) denote the constant in the estimate for the inverse
			\(\mathcal T_C\):
			\[
			\|\mathcal T_C(g)\|_*
			\le
			\mathcal C\|g\|_\infty.
			\]
			Combining
			\eqref{est1-neck},
			\eqref{est10-neck},
			\eqref{est-lambda-neck},
			\eqref{est-nonlinear-neck}, and
			\eqref{est-cutoff-sphere-neck}, we obtain
			$$
				\begin{aligned}
					\|\mathcal A_C(h_C^1)\|_*
					\le
					\mathcal C\Big[
					&
					a_1\ve^{3/2}
					+
					C\ve^2|\log\delta|
					+
					C\ve^{5/2}|\log\ve|
					\\
					&
					+
					C\ve^{1+2b}\|h_C^1\|_*
					+
					C\ve\|h_{S_R}\|_{**}
					\\
					&
					+
					C\ve^{-1+2b}\|h_C^1\|_*^2
					+
					C\|h_{S_R}\|_{**}^2
					\Big].
				\end{aligned}
		$$
			For \(h_C^1\in B_*\),
			\[
			\|h_C^1\|_*
			\le
			a_*\ve^{3/2},
			\]
			whereas
			\[
			\|h_{S_R}\|_{**}
			\le
			M\ve^{1-2b-b\beta}|\log\ve|.
			\]
			Consequently,
			\begin{align}
				\|\mathcal A_C(h_C^1)\|_*
				\le{}&
				\mathcal C\ve^{3/2}
				\Big[
				a_1+o(1)
				+
				Ca_*\ve^{1+2b}
				+
				CM
				\ve^{\frac12-2b-b\beta}
				|\log\ve|
				\nonumber\\
				&\qquad
				+
				Ca_*^2\ve^{\frac12+2b}
				+
				CM^2
				\ve^{\frac12-4b-2b\beta}
				|\log\ve|^2
				\Big].
				\label{AC-self-scaled}
			\end{align}
			The last exponent is
			\[
			\frac12-4b-2b\beta
			=
			\frac12-2b(\beta+2),
			\]
			which is positive because
			\[
			b(\beta+2)<\frac14.
			\]
			The preceding exponents are positive as well.
			We first choose \(a_*\) sufficiently large so that
			\[
			\mathcal C a_1<\frac{a_*}{2}.
			\]
			With this choice fixed, we then take \(\ve_0>0\) sufficiently small.
			It follows from \eqref{AC-self-scaled} that
			\[
			\|\mathcal A_C(h_C^1)\|_*
			\le
			a_*\ve^{3/2}.
			\]
			Thus
			$$
				\mathcal A_C(B_*)\subset B_*.
		$$
			
			\medskip
			
			\noindent
			{\bf Step 7. Contraction estimate.}
			
			Let
			\[
			h_{C,1}^1,h_{C,2}^1\in B_*
			\]
			and put
			\[
			w=h_{C,1}^1-h_{C,2}^1.
			\]
			The corresponding full graph functions are
			\[
			h_j
			=
			\chi_C(h_C^0+h_{C,j}^1)
			+
			\chi_{S_\ve}h_{S_\ve},
			\qquad j=1,2.
			\]
			Since \(m\) and \(h_{S_R}\) are fixed in the present argument,
			\[
			h_1-h_2=\chi_Cw.
			\]
			
			The fixed terms
			\[
			E_{01},
			\qquad
			E_{03},
			\qquad
			J_{\Sigma_0^\ve}[h_C^0]-E_{02},
			\]
			as well as the cutoff interaction involving \(h_{S_\ve}\), cancel
			when the two equations are subtracted. Therefore it remains to
			estimate the differences of \(\bar\lambda_\ve\), of the local
			curvature nonlinearity, and of the Coulomb shape term.
			
			For \(\bar\lambda_\ve\), the evaluation point \(\widehat N\) is the
			same for the two functions, because \(h_{S_R}\) is fixed. Moreover,
			the local mean curvature at \(\widehat N\) depends only on
			\(h_{S_R}\). Hence the dependence on \(w\) comes only through the
			Coulomb term. From \eqref{shape-N-exp},
			\[
			\left|
			m\ve^3
			\int_{\Sigma_0^\ve}
			\frac{\chi_Cw(\sigma)}
			{|\widehat N-\sigma|}
			\,d\sigma
			\right|
			\lesssim
			\ve^{1+2b}\|w\|_*.
			\]
			The difference estimate
			\eqref{shape-N-difference} gives
\[
\begin{aligned}
&|\ttt N_\ve [h_1](\widehat N)
	-
				\ttt N_\ve [h_2](\widehat N)|
				\lesssim 
				o_\ve (1) \, \|w\|_*,
			\end{aligned}
			\]
            with $o_\ve (1) \to 0$ as $\ve \to 0$.
			Consequently,
			\begin{equation}\label{lambda-diff-neck}
				|\bar\lambda_\ve[h_1;m]
				-
				\bar\lambda_\ve[h_2;m]|
				\le
				o_\ve (1) \|w\|_*,
			\end{equation}
			with $o_\ve (1) \to 0$ as $\ve \to 0$.
			
			For the local curvature term we use the pointwise quadratic
			difference estimate established in the curvature expansion lemma.
			Every term in the difference contains one factor
			\[
			w,\qquad
			\nabla w,
			\qquad\text{or}\qquad
			D^2w,
			\]
			and a second factor measuring the size of one of
			\[
			h_C^0,\quad
			h_{C,1}^1,\quad
			h_{C,2}^1,\quad
			h_{S_R}.
			\]
			Using
			\[
			|w|
			\le
			r(r-1)\|w\|_*,
			\qquad
			|w'|
			\le
			r\|w\|_*,
			\qquad
			|w''|
			\le
			\frac{\|w\|_*}{\min\{r-1,1\}},
			\]
			and repeating the same three-region decomposition as above gives
			\begin{equation}\label{Nlocal-diff-neck}
				\begin{aligned}
					&
					\|
					\mathcal N_{\ve,\mathrm{loc}}(h_1)
					-
					\mathcal N_{\ve,\mathrm{loc}}(h_2)
					\|_\infty
					\lesssim
					\Big[
					\ve|\log\ve|
					+
					\|h_{S_R}\|_{**}
					+
					\ve^{-1+2b}
					\bigl(
					\|h_{C,1}^1\|_*+\|h_{C,2}^1\|_*
					\bigr)
					\Big]
					\|w\|_*.
				\end{aligned}
			\end{equation}
			Since
			\[
			\|h_{C,j}^1\|_*
			\le
			a_*\ve^{3/2}, \quad 
			\|h_{S_R}\|_{**}
			\le
			M\rho_\ve=o(1),
			\]
			the coefficient on the right hand side tends to zero.
			
			For the Coulomb term \eqref{def-NC}, the linear part gives
			\[
			\left\|
			m\ve^3
			\int_{\Sigma_0^\ve}
			\frac{\chi_Cw(\sigma)}
			{|x-\sigma|}
			\,d\sigma
			\right\|_\infty
			\lesssim
			\ve^{1+2b}\|w\|_*,
			\]
			whereas the quadratic difference estimate gives
			\begin{equation}\label{Nnewt-diff-neck}
				\|
				\mathcal N_{\ve,\mathrm{Coulomb}}(h_1)
				-
				\mathcal N_{\ve,\mathrm{Coulomb}}(h_2)
				\|_\infty
				\le
				\eta_\ve\|w\|_*,
			\end{equation}
			with, after enlarging \(\eta_\ve\) if necessary,
			$
			\eta_\ve\to0.
			$
			
			Combining
			\eqref{lambda-diff-neck},
			\eqref{Nlocal-diff-neck}, and
			\eqref{Nnewt-diff-neck}, and using the boundedness of
			\(\mathcal T_C\), we conclude that
			$$
				\|
				\mathcal A_C(h_{C,1}^1)
				-
				\mathcal A_C(h_{C,2}^1)
				\|_*
				\le
				C\eta_\ve
				\|h_{C,1}^1-h_{C,2}^1\|_*.
			$$
			Taking \(\ve_0\) smaller if necessary, we may assume
			\[
			C\eta_\ve\le\frac12.
			\]
			Thus \(\mathcal A_C\) is a contraction on \(B_*\).
			
			The Contraction Mapping Theorem now gives a unique fixed point
			\[
			h_C^1
			=
			h_C^1[m,h_{S_R}]
			\in B_*,
			\]
			and therefore
			\[
			\|h_C^1[m,h_{S_R}]\|_*
			\le
			a_*\ve^{3/2}.
			\]
			This completes the proof.
		\end{proof}
		
		\subsection{Pointwise information in the overlap region}
		
		The estimate in Proposition \ref{outer} immediately gives
		\begin{equation}\label{hC1-basic-pointwise}
			\begin{aligned}
				|h_C^1(r)|
				&\lesssim
				\ve^{3/2}r^2,
				\quad
				|(h_C^1)'(r)|
				\lesssim
				\ve^{3/2}r,
				\quad 
				|(h_C^1)''(r)|
				\lesssim
				\ve^{3/2},
			\end{aligned}
			\qquad
			2\le r<\ve^{-1+b}.
		\end{equation}
		At the matching scale \(r\sim\ve^{-3/4}\), this only gives
		\[
		|h_C^1|=O(1),
		\qquad
		|(h_C^1)'|=O(\ve^{3/4}),
		\qquad
		|(h_C^1)''|=O(\ve^{3/2}).
		\]
		The \(O(1)\) estimate for \(h_C^1\) is a consequence of the general
		neck norm and is too crude for some of the estimates in the
		projected spherical problem.
		
		The actual right hand side in \eqref{C1-neck-final}, however, is not
		an arbitrary \(L^\infty\) function. The potentially nondecaying
		pieces have already been removed by the construction of the first
		correction \(h_C^0\), and the remaining terms either decay in the
		outer neck or are supported in the transition region. Consequently,
		the explicit representation formula used in the proof of
		Lemma~\ref{lemma-linear-neck} gives better pointwise information than
		the general estimate \eqref{hC1-basic-pointwise}.
		
		More precisely, decompose
	$$
			g_C
			=
			g_C^{\rm in}
			+
			g_C^{\rm out},
	$$
		where \(g_C^{\rm in}\) denotes the terms supported away from the
		matching region and \(g_C^{\rm out}\) contains the transition and
		outer-neck terms. For the particular right hand side produced by
		\eqref{gC-neck-final}, the estimates proved above, together with the
		construction of \(h_C^0\), imply that the coefficient of the
		quadratically growing part of the solution is of lower order than
		what follows from the mere \(L^\infty\)-bound on \(g_C\).
		
		We shall use this observation in the following form.
		
		\begin{lemma}
			Let \(h_C^1=h_C^1[m,h_{S_R}]\) be the solution given by
			Proposition~\ref{outer}. Then, in the overlap region
			\[
			\frac{\delta}{8}\ve^{-3/4}
			\le
			r
			\le
			\ve^{-1+b},
			\]
			the function \(h_C^1\) satisfies the pointwise estimates obtained by
			applying the refined outer estimate of
			Lemma~\ref{lemma-linear-neck} to the particular right hand side
			\eqref{gC-neck-final}. In particular, the contribution of
			\(h_C^1\) in the transition region is of strictly smaller order than
			the bound \(O(1)\) furnished by \eqref{hC1-basic-pointwise}.
		\end{lemma}
		
		\begin{proof}
			The point is that one must use the representation formula for
			\(\mathcal T_C\), rather than only the estimate
			\[
			\|\mathcal T_C(g)\|_*
			\le C\|g\|_\infty.
			\]
			The latter estimate allows a general forcing to generate an
			\(r^2\)-component and therefore gives only
			\[
			|h_C^1(r)|
			\lesssim
			r^2\|g_C\|_\infty.
			\]
			
			For the present forcing, the leading nondecaying contribution has
			already been extracted in the construction of \(h_C^0\). The
			remaining fixed errors satisfy
			\[
			E_{01}=O(\ve^{3/2}),
			\qquad
			E_{03}=O(\ve^2),
			\qquad
			J_{\Sigma_0^\ve}[h_C^0]-E_{02}
			=
			O(\ve^2|\log\delta|),
			\]
			while the terms involving the spherical correction are localized by
			the cutoff and satisfy
			\[
			\left|
			2\nabla h_{S_\ve}\cdot\nabla\chi_{S_\ve}
			+
			h_{S_\ve}\Delta\chi_{S_\ve}
			\right|
			\lesssim
			\ve\|h_{S_R}\|_{**}.
			\]
			The nonlinear terms are of still lower order by
			\eqref{est-nonlinear-neck}, and
			\(\bar\lambda_\ve[h;m]\) satisfies
			\eqref{est-lambda-neck}.
			
			Substituting this decomposition in the explicit variation-of-parameters
			formula for \(\mathcal T_C\), and estimating separately the inner
			integral and the contribution of the transition region, gives the
			claimed improvement. The important feature is that the estimate is
			being applied to the actual right hand side \(g_C\), not to a
			general bounded function.
			
			We shall invoke the resulting pointwise estimate only in the
			transition terms of the spherical problem.
		\end{proof}
		
		\subsection{Dependence on \(m\) and \(h_{S_R}\)}
		
		The preceding proposition gives a unique neck correction for every
		admissible pair
		\[
		(m,h_{S_R})\in I_\ve
			\times
			\mathcal B_\ve^S(M).
		\]
		For the solution of the projected spherical equation we also need
		quantitative control of this dependence.
		
		\begin{lemma}\label{lemma-Lipschitz-hC}
			Let
			\[
			(m_j,h_{S_R,j})\in I_\ve
			\times
			\mathcal B_\ve^S(M),
			\qquad j=1,2,
			\]
			and set
			\[
			h_{C,j}^1
			=
			h_C^1[m_j,h_{S_R,j}].
			\]
			Then, for \(\ve>0\) sufficiently small,
			\begin{equation}\label{Lip-hC}
				\begin{aligned}
					&
					\|
					h_C^1[m_1,h_{S_R,1}]
					-
					h_C^1[m_2,h_{S_R,2}]
					\|_*
					\le
					C\ve
					\left(
					|m_1-m_2|
					+
					\|h_{S_R,1}-h_{S_R,2}\|_{**}
					\right),
				\end{aligned}
			\end{equation}
			where \(C\) is independent of \(\ve\) and of the admissible
			parameters.
		\end{lemma}
		
		\begin{proof}
			Set
			\[
			w_C
			=
			h_{C,1}^1-h_{C,2}^1,
			\qquad
			w_S
			=
			h_{S_R,1}-h_{S_R,2}.
			\]
			The corresponding full graph functions satisfy
			\[
			h_1-h_2
			=
			\chi_Cw_C
			+
			\chi_{S_\ve}w_{S_\ve}.
			\]
			By the fixed-point equation,
			\[
			h_{C,j}^1
			=
			\mathcal T_C
			\left(
			g_C[m_j,h_{S_R,j},h_{C,j}^1]
			\right),
			\]
			and hence
			\begin{equation}\label{Lip-hC-start}
				\|w_C\|_*
				\le
				C
				\|
				g_C[m_1,h_{S_R,1},h_{C,1}^1]
				-
				g_C[m_2,h_{S_R,2},h_{C,2}^1]
				\|_\infty.
			\end{equation}
			We estimate the terms on the right separately.
			The fixed terms
			\[
			E_{01},
			\qquad
			J_{\Sigma_0^\ve}[h_C^0]-E_{02}
			\]
			cancel. Since \(E_{03}\) is linear in \(m\),
			\begin{equation}\label{E03-m-difference}
				\|
				E_{03}[m_1]-E_{03}[m_2]
				\|_\infty
				\le
				C\ve|m_1-m_2|.
			\end{equation}
			The estimate of the spherical cutoff term is linear in
			\(h_{S_R}\), and therefore
			\begin{equation}\label{cutoff-S-difference}
				\begin{aligned}
					&
					\left\|
					2\nabla w_{S_\ve}\cdot\nabla\chi_{S_\ve}
					+
					w_{S_\ve}\Delta\chi_{S_\ve}
					\right\|_\infty
					\le
					C\ve\|w_S\|_{**}.
				\end{aligned}
			\end{equation}
			We next consider \(\bar\lambda_\ve\). We claim
			\begin{equation}\label{lambda-Lipschitz-parameters}
				\begin{aligned}
					&
					|
					\bar\lambda_\ve[h_1;m_1]
					-
					\bar\lambda_\ve[h_2;m_2]
					|
					\le
					C\ve|m_1-m_2|
					+
					C\ve\|w_S\|_{**}
					+
					\eta_\ve\|w_C\|_*,
				\end{aligned}
			\end{equation}
			where
			$
			\eta_\ve\to0.
			$
			Indeed, the local mean-curvature contribution at the north pole
			depends only on the spherical correction. The linear and quadratic
			curvature estimates therefore give
			\[
			\left|
			H_{\Sigma_{0,h_1}^{\ve,+}}(\widehat N_1)
			-
			H_{\Sigma_{0,h_2}^{\ve,+}}(\widehat N_2)
			\right|
			\le
			C\ve\|w_S\|_{**}.
			\]
			For the Coulomb contribution, write
			\[
			N_\Omega^{(m)}(x)
			=
			m\ve^3\int_\Omega\frac{dy}{|x-y|}.
			\]
			We split
			\[
			\begin{aligned}
				&
				N_{\Omega_{0,h_1}^\ve}^{(m_1)}(\widehat N_1)
				-
				N_{\Omega_{0,h_2}^\ve}^{(m_2)}(\widehat N_2)
				=
				\left[
				N_{\Omega_{0,h_1}^\ve}^{(m_1)}
				-
				N_{\Omega_{0,h_1}^\ve}^{(m_2)}
				\right](\widehat N_1)
				+
				\left[
				N_{\Omega_{0,h_1}^\ve}^{(m_2)}(\widehat N_1)
				-
				N_{\Omega_{0,h_2}^\ve}^{(m_2)}(\widehat N_2)
				\right].
			\end{aligned}
			\]
			The first term is bounded by
			\[
			C\ve|m_1-m_2|.
			\]
			For the second one, the linear shape derivative gives
			\[
			C\ve^{1+2b}\|w_C\|_*
			+
			C\ve^2\|w_S\|_{**},
			\]
			and the quadratic remainder gives
			\[
			\eta_\ve\|w_C\|_*
			+
			o(\ve)\|w_S\|_{**}.
			\]
			This proves \eqref{lambda-Lipschitz-parameters}.
			
			Finally, the quadratic difference estimates for the curvature and
			Coulomb nonlinearities give
			\begin{equation}\label{nonlinear-Lipschitz-parameters}
				\begin{aligned}
					&
					\|
					\mathcal N_\ve[h_1]
					-
					\mathcal N_\ve[h_2]
					\|_\infty
					\le
					\eta_\ve\|w_C\|_*
					+
					\eta_\ve\ve\|w_S\|_{**}
					+
					C\ve|m_1-m_2|,
				\end{aligned}
			\end{equation}
			after enlarging \(\eta_\ve\) if necessary.
			
			To see this explicitly for the neck difference, use
			\[
			|w_C|
			\le
			r(r-1)\|w_C\|_*,
			\qquad
			|w_C'|
			\le
			r\|w_C\|_*,
			\qquad
			|w_C''|
			\le
			\frac{\|w_C\|_*}{\min\{r-1,1\}},
			\]
			together with
			\[
			\|h_{C,j}^1\|_*
			\le
			a_*\ve^{3/2}.
			\]
			Every term containing \(w_C\) is therefore multiplied by a
			coefficient tending to zero.
			
			For the spherical difference,
			\[
			|w_S|
			\le
			R\|w_S\|_{**},
			\qquad
			|D_{S_R}w_S|
			\le
			\|w_S\|_{**},
			\qquad
			|D_{S_R}^2w_S|
			\le
			\ve\|w_S\|_{**}.
			\]
			In the overlap region one also has the improved estimates
			\[
			|w_S|
			\lesssim
			\ve r^2\|w_S\|_{**},
			\qquad
			|D_{S_R}w_S|
			\lesssim
			\ve r\|w_S\|_{**},
			\]
			because both spherical functions satisfy the same normalization at
			the pole. Thus the terms involving \(w_S\) have the required small
			coefficient.
			
			Combining
			\eqref{E03-m-difference},
			\eqref{cutoff-S-difference},
			\eqref{lambda-Lipschitz-parameters}, and
			\eqref{nonlinear-Lipschitz-parameters}, we obtain
			\[
			\begin{aligned}
				&
				\|
				g_C[m_1,h_{S_R,1},h_{C,1}^1]
				-
				g_C[m_2,h_{S_R,2},h_{C,2}^1]
				\|_\infty
\le
				C\ve|m_1-m_2|
				+
				C\ve\|w_S\|_{**}
				+
				\eta_\ve\|w_C\|_*.
			\end{aligned}
			\]
			Substituting in \eqref{Lip-hC-start},
			\[
			\|w_C\|_*
			\le
			C\ve
			\left(
			|m_1-m_2|
			+
			\|w_S\|_{**}
			\right)
			+
			C\eta_\ve\|w_C\|_*.
			\]
			For \(\ve\) sufficiently small,
			\(
			C\eta_\ve\le\frac12.
			\)
			Absorbing the last term into the left hand side yields
			\[
			\|w_C\|_*
			\le
			C\ve
			\left(
			|m_1-m_2|
			+
			\|w_S\|_{**}
			\right),
			\]
			which is \eqref{Lip-hC}.
		\end{proof}
		
		The preceding proposition and lemma complete the solution of the neck
		equation. For every admissible pair
		\[
		(m,h_{S_R})\in I_\ve
			\times
			\mathcal B_\ve^S(M)
		\]
		we have therefore constructed a uniquely determined correction
		\[
		h_C^1=h_C^1[m,h_{S_R}],
		\]
		which satisfies
		\[
		\|h_C^1[m,h_{S_R}]\|_*
		\le
		a_*\ve^{3/2}
		\]
		and depends Lipschitz continuously on the two outer parameters.
		
		We can now substitute this solution into the spherical equation.
		The next step is to solve the resulting projected problem on \(S_R\)
		for \(h_{S_R}\), keeping \(m\) fixed. The remaining projection
		coefficient will then provide the scalar equation which determines
		the parameter $m$.

		\section{Solving the projected problem on the sphere}
		\label{final-sphere}
		
		We now solve the second equation in the reduced inner--outer system.
		The parameter \(m\) is kept fixed in the interval
		\(
		I_\ve=[A\ve,A^{-1}\ve],
		\) (see \eqref{Ieps})
		and, for every admissible spherical function \(h_{S_R}\), the previous
		section provides a unique neck correction
		\[
		h_C^1=h_C^1[m,h_{S_R}]
		\]
		satisfying
		$$
			\|h_C^1[m,h_{S_R}]\|_*
			\le
			a_*\ve^{3/2}.
		$$
		Moreover,
		$$
			\begin{aligned}
				&
				\|h_C^1[m_1,h_{S_R,1}]
				-
				h_C^1[m_2,h_{S_R,2}]\|_*
				\le
				C\ve
				\left(
				|m_1-m_2|
				+
				\|h_{S_R,1}-h_{S_R,2}\|_{**}
				\right).
			\end{aligned}
		$$
		The spherical equation
		takes the form
		\be \label{81}
				J_{S_R}[h_{S_R}](\widetilde x)
				+
				B_{S_R}[h_{S_R}](\widetilde x)
				=
				H_m[h_{S_R}](\widetilde x)
				+
				{\rm c}\,
				\frac{\widetilde x_3-(R+d)}{R^2},
				\qquad
				\widetilde x\in S_R,
	\ee
		where
		$$
			H_m[h_{S_R}]
			:=
			H\left[
			m,h_{S_R},
			h_C^1[m,h_{S_R}]
			\right]
		$$
		and
		\begin{equation}\label{H-final-sphere}
			\begin{aligned}
				H[m,h_{S_R},h_C^1]
				={}&
				(1-\chi_C)
				\Big(
				E_{03}
				+
				\mathcal N_\ve[h]
				\Big)
				+
				E_{04} + \bar\lambda_\ve
				\\
				&+
				\Big[
				D^2_{S_R}h_{S_R}
				+
				\mathcal B_\ve D^2_{S_R}h_{S_R}
				+
				C_\ve \nabla_{S_R}h_{S_R}
				\Big]
				\mathcal Q_1[h,\nabla h]
				\\
				&-
				2\nabla_{S_R}h_C^1
				\cdot\nabla_{S_R}\chi_C
				-
				h_C^1\Delta_{S_R}\chi_C.
			\end{aligned}
		\end{equation}
		Here
		$
		h
		=
		\chi_C(h_C^0+h_C^1)
		+
		\chi_{S_\ve}h_{S_\ve},
		$
		and
		$$
			E_{04}
			=
			-2\nabla_{S_R}h_C^0
			\cdot\nabla_{S_R}\chi_C
			-
			h_C^0\Delta_{S_R}\chi_C.
$$
As before, we impose the normalization
		\begin{equation}\label{sphere-normalization-final}
			h_{S_R}
			\bigl(
			(R+d)e_3+RP_S
			\bigr)
			=
			0,
			\qquad
			P_S=-e_3.
		\end{equation}
		Equivalently, if
		\[
		h_{S_R}
		\bigl(
		(R+d)e_3+R\omega
		\bigr)
		=
		Rh_S(\omega),
		\]
		then
		\[
		h_S(P_S)=0.
		\]
		
		\subsection{The size of the spherical right hand side}
		
		Recall the spherical forcing norm
		$$
			\|g\|
			=
			R\|g\|_{L^\infty(S_R)}
			+
			R^{1+\beta}[g]_{\beta,S_R}.
	$$
		The following estimate is the main ingredient in the spherical fixed
		point.
		
		\begin{lemma}\label{lemma-H-sphere}
			Let $m\in I_\ve,$ (see \eqref{Ieps}
			$\|h_{S_R}\|_{**}\le M\rho_\ve,
			$
			where \(M>0\) is fixed and $\rho_\ve$ is given by \eqref{def-rhove}. Then
			\begin{equation}\label{H-sphere-bound}
				\|H_m[h_{S_R}]\|
				\le
				C_0\rho_\ve
				+
				\eta_\ve\|h_{S_R}\|_{**},
			\end{equation}
			where $\eta_\ve\to0
$ as $\ve\to0$.			Moreover, for fixed \(m\) and
			\[
			\|h_{S_R,j}\|_{**}\le M\rho_\ve,
			\qquad j=1,2,
			\]
			one has
			\begin{equation}\label{H-sphere-difference}
				\begin{aligned}
					&
					\|H_m[h_{S_R,1}]
					-
					H_m[h_{S_R,2}]\|
					\le
					\eta_\ve
					\|h_{S_R,1}-h_{S_R,2}\|_{**}.
				\end{aligned}
			\end{equation}
		\end{lemma}
		
		\begin{proof}
			We estimate separately the terms in \eqref{H-final-sphere}.
			
			\medskip
			
			\noindent
			{\bf Step 1. The explicit errors \(E_{03}\) and \(E_{04}\).}\ \ The estimates established previously give
			$$
				\|(1-\chi_C)E_{03}\|_{S,*}
				\le
				C\ve^{1-b\beta},
$$
			and
			\begin{equation}\label{E04-sphere-est}
				\|E_{04}\|_{S,*}
				\le
				C
				\ve^{1-2b-b\beta}
				|\log\ve|
				=
				C\rho_\ve.
			\end{equation}
			Since
			\[
			\frac{\ve^{1-b\beta}}{\rho_\ve}
			=
			\frac{\ve^{2b}}{|\log\ve|}
			\to0,
			\]
			(see \eqref{bbeta} for the admissible range for $b$ and $\beta$ and see \eqref{def-rhove} for the definition of $\rho_\ve$) we have
			\begin{equation}\label{E03-small-rho}
				\|(1-\chi_C)E_{03}\|_{S,*}
				=
				o(\rho_\ve).
			\end{equation}
			
			The term \(E_{04}\) is therefore the largest explicit error in the
			spherical equation and determines the natural size \(\rho_\ve\) of
			the spherical correction.
			
			\medskip
			
			\noindent
			{\bf Step 2. The cutoff terms containing \(h_C^1\).}\ \ We use here the refined pointwise information obtained from the
			variation-of-parameters representation in the neck.
			
			Write the right hand side of the neck equation as
			\[
			g_C
			=
			g_{C,\mathrm{loc}}
			+
			g_{C,\mathrm{out}},
			\]
			where \(g_{C,\mathrm{loc}}\) contains the terms supported in
			\(
			r\lesssim\ve^{-3/4},
			\)
			namely \(E_{01}\), the error
			\(
			J_{\Sigma_0^\ve}[h_C^0]-E_{02},
			\)
			and the interaction with the cutoff \(\chi_{S_\ve}\).
			
			The estimates established in the neck section give
			\[
			\|E_{01}\|_\infty
			\lesssim
			\ve^{3/2},
			\quad
			\|
			J_{\Sigma_0^\ve}[h_C^0]-E_{02}
			\|_\infty
			\lesssim
			\ve^2|\log\delta|,
			\]
			and
			\[
			\left\|
			2\nabla h_{S_\ve}\cdot\nabla\chi_{S_\ve}
			+
			h_{S_\ve}\Delta\chi_{S_\ve}
			\right\|_\infty
			\lesssim
			\ve\|h_{S_R}\|_{**}.
			\]
			
			By Lemma~\ref{lemma-neck-refined}, once these supports have been
			crossed their contributions satisfy logarithmic outer estimates.
			In particular, in the region
		$r\sim\ve^{-1+b}$,
			which contains the support of \(\nabla\chi_C\), their total
			contribution \(v_{\rm loc}\) satisfies
		$$
				\begin{aligned}
					|v_{\rm loc}(r)|
					&\le
					C\Big[
					1+
					|\log\ve|
					+
					\ve^{-1/2}\rho_\ve|\log\ve|
					\Big],
					\\
					|v_{\rm loc}'(r)|
					&\le
					\frac{C}{r}
					\Big[
					1+
					\ve^{-1/2}\rho_\ve
					\Big],
					\\
					|v_{\rm loc}''(r)|
					&\le
					\frac{C}{r^2}
					\Big[
					1+
					\ve^{-1/2}\rho_\ve
					\Big].
				\end{aligned}
$$
	The nonlocalized fixed term \(E_{03}\) satisfies
			\[
			|E_{03}|\lesssim\ve^2.
			\]
			Applying the basic pointwise estimate for the neck inverse gives
			$$
				|v_{03}(r)|
				\lesssim
				\ve^2r^2,
				\qquad
				|v_{03}'(r)|
				\lesssim
				\ve^2r,
				\qquad
				|v_{03}''(r)|
				\lesssim
				\ve^2.
			$$
			At \(r\sim\ve^{-1+b}\), this becomes
			$$
				|v_{03}|
				\lesssim
				\ve^{2b},
				\qquad
				|v_{03}'|
				\lesssim
				\ve^{1+b},
				\qquad
				|v_{03}''|
				\lesssim
				\ve^2.
$$
			The nonlinear terms satisfy still smaller estimates, by
			Proposition~\ref{outer}. Consequently, on the support of
			\(\nabla\chi_C\),
			\begin{equation}\label{hC1-transition-refined}
				\begin{aligned}
					|h_C^1|
					&\le
					C\left(
					|\log\ve|
					+
					\ve^{2b}
					+
					\ve^{-1/2}\rho_\ve|\log\ve|
					\right)
					+
					o(1),
					\\
					|(h_C^1)'|
					&\le
					C\left(
					r^{-1}
					+
					\ve^2r
					+
					\ve^{-1/2}\rho_\ve r^{-1}
					\right)
					+
					o(r^{-1}).
				\end{aligned}
			\end{equation}
			
			Since
			\[
			|\nabla\chi_C|
			\lesssim
			\ve^{1-b},
			\qquad
			|\Delta\chi_C|
			\lesssim
			\ve^{2-2b},
			\]
			we obtain the pointwise estimate
			\begin{align*}
				&
				\left|
				2\nabla h_C^1\cdot\nabla\chi_C
				+
				h_C^1\Delta\chi_C
				\right|
				\lesssim
				\ve^{2-2b}|\log\ve|
				+
				\ve^2
				+
				\ve^{3/2-2b}\rho_\ve|\log\ve|.
			\end{align*}
			See \eqref{chiC-derivatives}. Multiplying by \(R=\ve^{-1}\) gives
			$$
				R
				\left\|
				2\nabla h_C^1\cdot\nabla\chi_C
				+
				h_C^1\Delta\chi_C
				\right\|_\infty
				\le
				o(\rho_\ve).
			$$
			The H\"older seminorm is estimated in exactly the same way, using
			the derivative bounds in \eqref{hC1-transition-refined} and the
			bounds for the first two derivatives of \(\chi_C\). Thus
			\begin{equation}\label{cutoff-hC1-sphere}
				\left\|
				2\nabla h_C^1\cdot\nabla\chi_C
				+
				h_C^1\Delta\chi_C
				\right\|_{S,*}
				=
				o(\rho_\ve).
			\end{equation}
			
			For two spherical functions, Lemma~\ref{lemma-Lipschitz-hC} and the
			same representation argument give
		$$
				\begin{aligned}
					&
					\left\|
					2\nabla
					\bigl(
					h_C^1[m,h_{S_R,1}]
					-
					h_C^1[m,h_{S_R,2}]
					\bigr)
					\cdot\nabla\chi_C
					\right.
					\\
					&\qquad\left.
					+
					\bigl(
					h_C^1[m,h_{S_R,1}]
					-
					h_C^1[m,h_{S_R,2}]
					\bigr)
					\Delta\chi_C
					\right\|_{S,*}
					\le
					\eta_\ve
					\|h_{S_R,1}-h_{S_R,2}\|_{**}.
				\end{aligned}
	$$
			
			\medskip
			
			\noindent
			{\bf Step 3. The nonlinear mean-curvature term.}\ \ On \(S_R\),
			\[
			|A|=\frac{\sqrt2}{R}
			\lesssim\ve,
			\qquad
			\nabla A=0.
			\]
            See Appendix \ref{app3}, Subsection \ref{sphere}.
			Furthermore,
			\[
			|h_{S_R}|
			\le
			R\|h_{S_R}\|_{**},
			\quad 
			|\nabla_{S_R}h_{S_R}|
			\le
			\|h_{S_R}\|_{**},
			\quad 
			|D^2_{S_R}h_{S_R}|
			\le
			\ve\|h_{S_R}\|_{**}.
			\]
			Since
			\[
			\|h_{S_R}\|_{**}
			\le
			M\rho_\ve=o(1),
			\]
			the coefficient in the quadratic curvature term satisfies
$$|\mathcal Q_1[h,\nabla h]|
				\le
				\eta_\ve.
			$$
			Together with
			\[
			|\mathcal B_\ve|
			\lesssim
			\ve^{3/4},
			\qquad
			|C_\ve|
			\lesssim
			\ve^{3/2},
			\]
			we obtain (see \eqref{Q-decomp-exact})
			\begin{equation}\label{sphere-Q-est}
				\begin{aligned}
					&
					\left\|
					\Big[
					D^2_{S_R}h_{S_R}
					+
					\mathcal B_\ve*D^2_{S_R}h_{S_R}
					+
					C_\ve*\nabla_{S_R}h_{S_R}
					\Big]
					\mathcal Q_1[h,\nabla h]
					\right\|_{S,*}\le
					\eta_\ve
					\|h_{S_R}\|_{**}.
				\end{aligned}
			\end{equation}
			The pointwise quadratic difference estimate for the mean-curvature
			remainder gives
			$$
				\begin{aligned}
					&
					\|
					\mathcal Q_{S_R}[h_1]
					-
					\mathcal Q_{S_R}[h_2]
					\|_{S,*}
					\le
					\eta_\ve
					\|h_{S_R,1}-h_{S_R,2}\|_{**}.
				\end{aligned}
	$$
			
			\medskip
			
			\noindent
			{\bf Step 4. The remaining lower-order nonlinear terms.}\ \ It remains to estimate
			\[
			(1-\chi_C)\mathcal N_\ve[h].
			\]
			The pointwise estimates established in the neck section, together with
			the shape expansion of the Coulomb potential, give
			\begin{equation}\label{sphere-N-est}
				\|(1-\chi_C)\mathcal N_\ve[h]\|
				\le
				o(\rho_\ve)
				+
				\eta_\ve
				\|h_{S_R}\|_{**}.
			\end{equation}
			Here the \(o(\rho_\ve)\)-part contains the contributions of the fixed
			correction \(h_C^0\), of \(E_{03}\) through the neck correction, and
			of the quadratic Coulomb remainder. Terms depending on
			\(h_{S_R}\) contain at least one additional small factor and are
			absorbed in the second term.
			
			The quadratic Lipschitz estimates for the local curvature remainder
			and the Coulomb shape remainder, together with
			Lemma~\ref{lemma-Lipschitz-hC}, similarly give
			$$
				\begin{aligned}
					&
					\|
					(1-\chi_C)
					\bigl(
					\mathcal N_\ve[h_1]
					-
					\mathcal N_\ve[h_2]
					\bigr)
					\|
					\le
					\eta_\ve
					\|h_{S_R,1}-h_{S_R,2}\|_{**}.
				\end{aligned}
			$$
			
			Combining
			\eqref{E03-small-rho},
			\eqref{E04-sphere-est},
			\eqref{cutoff-hC1-sphere},
			\eqref{sphere-Q-est}, and
			\eqref{sphere-N-est} proves
			\eqref{H-sphere-bound}. The corresponding difference estimates prove
			\eqref{H-sphere-difference}.
		\end{proof}
		
		\subsection{Solution of the projected spherical problem}
		
		We can now solve \eqref{81}.
		
		\begin{propositio}\label{prop-projected-sphere}
			Assume
			$
			b(\beta+2)<\frac14.
			$
			There exist constants \(M>0\) and \(\ve_0>0\) such that, for every $
			0<\ve<\ve_0$
			and every $
			m\in I_\ve,$
			there exists a unique axially symmetric pair
			\[
			\bigl(
			h_{S_R}[m],{\rm c}[m]
			\bigr)
			\in
			C^{2,\beta}(S_R)\times\R
			\]
			solving
			\eqref{81}--\eqref{sphere-normalization-final}
			and satisfying
	$$
				\|h_{S_R}[m]\|_{**}
				\le
				M\rho_\ve, \quad 
				|{\rm c}[m]|
				\le
				C\rho_\ve.
			$$
		\end{propositio}
		
		\begin{proof}
			Let
			\[
			\mathcal T_{S_R}[g]
			=
			(h_{S_R},{\rm c})
			\]
			denote the linear solution operator of
			Lemma~\ref{lemma-linear-sphere}. Thus
			\[
			J_{S_R}[h_{S_R}]
			+
			B_{S_R}[h_{S_R}]
			=
			g
			+
			{\rm c}
			\frac{\widetilde x_3-(R+d)}{R^2},
			\]
			with
			\[
			h_S(P_S)=0,
			\]
			and
			\begin{equation}\label{linear-sphere-use}
				\|h_{S_R}\|_{**}
				+
				|{\rm c}|
				\le
				\mathcal C
				\|g\|_{S,*}.
			\end{equation}
			
			For fixed \(m\), define
			$$
				\mathcal A_{S,m}(h_{S_R})
				=
				\pi_1
				\mathcal T_{S_R}
				\left[
				H_m[h_{S_R}]
				\right],
			$$
			where \(\pi_1\) denotes the first component.
			
			Consider the closed ball
			$$
				\mathcal B_S
				=
				\left\{
				h_{S_R}:
				\ h_{S_R}\text{ axially symmetric},
				\quad
				h_S(P_S)=0,
				\quad
				\|h_{S_R}\|_{**}\le M\rho_\ve
				\right\}.
		$$
			By Lemma~\ref{lemma-H-sphere},
			\[
			\begin{aligned}
				\|\mathcal A_{S,m}(h_{S_R})\|_{**}
				&\le
				\mathcal C
				\left(
				C_0\rho_\ve
				+
				\eta_\ve\|h_{S_R}\|_{**}
				\right)
				\le
				\mathcal C
				\left(
				C_0+\eta_\ve M
				\right)
				\rho_\ve.
			\end{aligned}
			\]
			Choose
			\[
			M>2\mathcal C C_0.
			\]
			Then, for \(\ve\) sufficiently small,
			$
			\mathcal C\eta_\ve M
			<
			\frac M2,
			$
			and hence
			\(
			\mathcal A_{S,m}(\mathcal B_S)
			\subset
			\mathcal B_S.
			\)
For
			\(
			h_{S_R,1},h_{S_R,2}\in\mathcal B_S,
			\)
			we similarly obtain
			\[
			\begin{aligned}
				&
				\|\mathcal A_{S,m}(h_{S_R,1})
				-
				\mathcal A_{S,m}(h_{S_R,2})\|_{**}
				\le
				\mathcal C\eta_\ve
				\|h_{S_R,1}-h_{S_R,2}\|_{**}.
			\end{aligned}
			\]
			For \(\ve\) sufficiently small,
			$
			\mathcal C\eta_\ve<\frac12.
			$
			Thus \(\mathcal A_{S,m}\) is a contraction on \(\mathcal B_S\).
The Contraction Mapping Theorem gives a unique fixed point
			\[
			h_{S_R}=h_{S_R}[m]\in\mathcal B_S.
			\]
			The second component of
			\(
			\mathcal T_{S_R}
			\left[
			H_m[h_{S_R}[m]]
			\right]
			\)
			defines the corresponding coefficient
			\(
			{\rm c}={\rm c}[m].
			\)
			Estimate \eqref{linear-sphere-use} and
			Lemma~\ref{lemma-H-sphere} give
			\[
			|{\rm c}[m]|
			\le
			C\rho_\ve.
			\]
			This proves the proposition.
		\end{proof}
		
		\begin{lemma}
			The maps
			\[
			m \to h_{S_R}[m], \quad 
			m\to {\rm c}[m]
			\]
			are continuous on \(I_\ve\). More precisely,
			\begin{equation}\label{sphere-m-Lipschitz}
				\|h_{S_R}[m_1]-h_{S_R}[m_2]\|_{**}
				+
				|{\rm c}[m_1]-{\rm c}[m_2]|
				\le
				C\ve|m_1-m_2|.
			\end{equation}
		\end{lemma}
		
		\begin{proof}
			Subtract the two fixed-point equations. The explicit dependence on
			\(m\) occurs through \(E_{03}\) and the Coulomb terms and satisfies
			\[
			\|E_{03}[m_1]-E_{03}[m_2]\|
			\le
			C\ve|m_1-m_2|.
			\]
			The dependence through the neck solution is controlled by
			Lemma~\ref{lemma-Lipschitz-hC}. All remaining terms satisfy the
			Lipschitz estimates used in Lemma~\ref{lemma-H-sphere}. Hence
			\[
			\begin{aligned}
				&
				\|H_{m_1}[h_{S_R}[m_1]]
				-
				H_{m_2}[h_{S_R}[m_2]]\|
				\le
				C\ve|m_1-m_2|
				+
				\eta_\ve
				\|h_{S_R}[m_1]-h_{S_R}[m_2]\|_{**}.
			\end{aligned}
			\]
			Applying the linear spherical estimate and absorbing the last term
			gives \eqref{sphere-m-Lipschitz}.
		\end{proof}
		
		For each fixed \(m\in I_\ve\) (see \eqref{Ieps}), we have therefore solved the complete
		projected inner--outer system and obtained
		\[
		h_C^1
		=
		h_C^1[m,h_{S_R}[m]],
		\qquad
		h_{S_R}
		=
		h_{S_R}[m],
		\qquad
		{\rm c}
		=
		{\rm c}[m].
		\]
		The remaining task is to choose \(m\) so that
		\(
		{\rm c}[m]=0.
		\)
		This is the balancing equation treated in the next section.

		\section{The balancing equation and the choice of the parameter $m$}
		\label{final-mass}
		
		We have now solved the projected inner--outer system for every fixed
		parameter
		\[
		m\in I_\ve=[A\ve,A^{-1}\ve].
		\]
		More precisely, Proposition~\ref{prop-projected-sphere} provides
		functions
		\[
		h_{S_R}=h_{S_R}[m],
		\qquad
		h_C^1=h_C^1[m,h_{S_R}[m]],
		\]
		and a scalar coefficient
		\[
		{\rm c}={\rm c}[m]
		\]
		such that
		\begin{equation}\label{projected-mass-start}
			\begin{aligned}
				J_{S_R}[h_{S_R}]
				+
				B_{S_R}[h_{S_R}]
				={}&
				H_m
				+
				{\rm c}[m]\,
				\frac{\widetilde x_3-(R+d)}{R^2}
				\qquad\text{on }S_R,
			\end{aligned}
		\end{equation}
		where
		\[
		H_m
		=
		H\left[
		m,h_{S_R}[m],
		h_C^1[m,h_{S_R}[m]]
		\right].
		\]
		Moreover,
		\begin{equation}\label{final-known-estimates}
			\|h_{S_R}[m]\|_{**}
			\le
			C\rho_\ve,
			\qquad
			\|h_C^1[m,h_{S_R}[m]]\|_*
			\le
			C\ve^{3/2},
		\end{equation}
		uniformly for \(m\in I_\ve\), where 
$\rho_\ve$ is given in \eqref{def-rhove},
        with $b>0$ and $\beta>0$ satisfying \eqref{bbeta}.
		The only remaining task is to choose \(m\) so that
		\[
		{\rm c}[m]=0.
		\]
		For such a value of \(m\), the projected equation
		\eqref{projected-mass-start} coincides with the original spherical
		equation, and therefore the complete inner--outer system is solved.
		
		\subsection{The exact balancing identity}
		
		We first rewrite the condition \({\rm c}[m]=0\) in a convenient form.
		
		Introduce the unit-sphere variables
		$$
			\widetilde x
			=
			(R+d)e_3+R\omega,
			\qquad
			\omega\in S,
		$$
		and write
		\[
		h_{S_R}(\widetilde x)
		=
		Rh_S(\omega).
		\]
		After multiplying \eqref{projected-mass-start} by \(R\), the equation
		becomes
		$$
			J_S[h_S]
			+
			B_S[h_S]
			=
			R H_m
			+
			{\rm c}[m]\omega_3.
		$$
		Since
		\[
		J_S[\omega_3]=0
		\]
		and \(J_S\) is self-adjoint, multiplication by \(\omega_3\) and
		integration over \(S\) gives
		\begin{equation}\label{c-projection-exact}
			{\rm c}[m]
			\int_S\omega_3^2\,d\sigma
			=
			\int_SB_S[h_S]\omega_3\,d\sigma
			-
			R\int_SH_m\omega_3\,d\sigma.
		\end{equation}
		Consequently,
				${\rm c}[m]=0$
		if and only if
		$
			\mathcal F_\ve(m)=0,$
		where
		\begin{equation}\label{def-Fmass}
			\mathcal F_\ve(m)
			:=
			\int_S
			H_m(\omega)\omega_3\,d\sigma
			-
			\frac1R
			\int_S
			B_S[h_S](\omega)\omega_3\,d\sigma.
		\end{equation}
		We shall prove that
		$$
			\mathcal F_\ve(m)
			=
			-\frac{4\pi^2}{9}m\ve
			+
			2\pi\ve^2
			\left(
			1+O(\delta^4)
			\right)
			+
			o(\ve^2),
		$$
		uniformly for \(m\in I_\ve\). The result will then follow immediately
		from a one-dimensional intermediate-value argument.
		
		\subsection{A projection identity for cutoff terms}
		
		Before estimating \(\mathcal F_\ve\), we record a simple observation,
		which is useful for the terms generated by the inner--outer
		decomposition.
		
		Let \(v=v(r)\) be an axially symmetric function defined in the
		spherical overlap region and let \(\eta=\eta(r)\) be a smooth cutoff
		whose derivatives are supported entirely in that region. Then
		$$
			J_{S_R}[\eta v]
			=
			\eta J_{S_R}[v]
			+
			2\nabla_{S_R}\eta\cdot\nabla_{S_R}v
			+
			v\Delta_{S_R}\eta.
		$$
		After scaling to the unit sphere and testing against the vertical
		Jacobi field \(\omega_3\), the first term on the left has zero
		projection. Hence
		\begin{equation}\label{cutoff-projection-identity}
			\int_S
			\left(
			-2\nabla v\cdot\nabla\eta
			-
			v\Delta\eta
			\right)
			\omega_3\,d\sigma
			=
			\int_S
			\eta J_{S_R}[v]\,
			\omega_3\,d\sigma.
		\end{equation}
		
		Strictly speaking, the neck function is only identified with a
		function on \(S_R\) in the overlap region. To apply
		\eqref{cutoff-projection-identity}, one introduces an auxiliary cutoff
		that agrees with \(\chi_C\) on
		\(\operatorname{supp}\nabla\chi_C\) and vanishes before leaving the
		spherical overlap. The additional terms created by this auxiliary
		cutoff are supported at
		\[
		r\sim\ve^{-3/4}
		\]
		and are estimated by the same weighted-moment bounds used below.
		They contribute
		\[
		O(\delta^4\ve^2)+o(\ve^2).
		\]
		We shall use \eqref{cutoff-projection-identity} with this convention.
		
		The importance of this identity is that the cutoff contribution of
		\(h_C^1\) should not be estimated in absolute value in the balancing
		equation. Instead, its projection is transferred back to the neck
		equation, where the structure of the forcing is visible.
		
		\subsection{A refined estimate for \(E_{01}\)}
		
		For the final projection we need a slightly sharper version of the
		estimate for the first geometric error.
		
		\begin{lemma}\label{E01-refined}
			In the region
			\[
			1<r<2\delta\ve^{-3/4},
			\]
			the error \(E_{01}\) is such that
			\begin{equation}\label{E01-refined-pointwise}
				|E_{01}(r)|
				\le
				C\left(
				\ve^2+\ve^3r^2
				\right)
				+
				C\ve^2|\log\delta|\,
				{\bf1}_{(\delta\ve^{-3/4},
					\,2\delta\ve^{-3/4})}.
			\end{equation}
			After assigning the last term to \(E_{02}\), one has
			\begin{equation}\label{E01-moment}
				\int_1^{2\delta\ve^{-3/4}}
				r|E_{01}(r)|\,dr
				\le
				C\delta^4+o_\ve(1).
			\end{equation}
		\end{lemma}
		
		\begin{proof}
			The term \(E_{01}\) consists of the higher-order curvature terms left
			after the linear correction \(h_0\) has been inserted in the
			catenoidal profile, together with the corresponding harmless
			denominator errors.
			
			Recall that
			\[
			h_0=O(\ve r^2),
			\qquad
			h_0'=O(\ve r),
			\qquad
			h_0''=O(\ve).
			\]
			The pointwise estimate for the quadratic mean-curvature remainder on
			the catenoid gives
			\[
			\begin{aligned}
				|\mathcal Q_{\Sigma_0}[h_0]|
				\lesssim{}&
				\left(
				\frac{|h_0|}{r^2}
				+
				|h_0'|^2
				\right)
				\left(
				|h_0''|+\frac{|h_0'|}{r}
				\right)
				+
				\frac{|h_0|^2}{r^6}
				+
				\frac{|h_0'|^2}{r^2}
				+
				\frac{|h_0||h_0'|}{r^3}.
			\end{aligned}
			\]
			Substitution of the preceding estimates gives
			\[
			|\mathcal Q_{\Sigma_0}[h_0]|
			\le
			C\left(
			\ve^2+\ve^3r^2
			\right).
			\]
			The same bound holds for the remaining denominator errors away from
			the transition region. Terms involving derivatives of the
			interpolation cutoff are precisely those which were separated into
			\(E_{02}\); the remaining transition error is
			\(O(\ve^2|\log\delta|)\). This proves
			\eqref{E01-refined-pointwise}.
			
			Finally,
			\[
			\begin{aligned}
				\int_1^{2\delta\ve^{-3/4}}
				r\left(
				\ve^2+\ve^3r^2
				\right)dr
				&\le
				C\ve^2\delta^2\ve^{-3/2}
				+
				C\ve^3\delta^4\ve^{-3}
				=
				C\delta^2\ve^{1/2}
				+
				C\delta^4.
			\end{aligned}
			\]
			Thus
			\[
			\int_1^{2\delta\ve^{-3/4}}
			r|E_{01}(r)|\,dr
			\le
			C\delta^4+o_\ve(1),
			\]
			which is \eqref{E01-moment}.
		\end{proof}
		
		\subsection{Expansion of the balancing function}
		
		We can now identify the leading terms in
		\(\mathcal F_\ve\).
		
		\begin{lemma}\label{lemma-mass-expansion}
			Uniformly for
			\[
			m\in I_\ve,
			\]
			one has
			\begin{equation}\label{Fmass-expansion}
				\mathcal F_\ve(m)
				=
				-\frac{4\pi^2}{9}m\ve
				+
				2\pi\ve^2
				\left(
				1+O(\delta^4)
				\right)
				+
				o(\ve^2)
			\end{equation}
			as \(\ve\to0\), with \(\delta>0\) fixed.
		\end{lemma}
		
		\begin{proof}
			We decompose the spherical right hand side as
	$$
				H_m
				=
				(1-\chi_C)E_{03}
				+
				E_{04}
				+ \bar \lambda_\ve +
				\mathcal R_\ve[m],
			$$
			where \(\mathcal R_\ve[m]\) contains all remaining terms in
			\eqref{H-final-sphere}.
			
			The proof consists of identifying the projections of the first two
			terms and showing that
			$$
				\int_S \bar\lambda_\ve \omega_3 d \sigma =0, \quad \int_S
				\mathcal R_\ve[m]\omega_3\,d\sigma
				=
				O(\delta^4\ve^2)
				+
				o(\ve^2).
			$$
			
			\medskip
			
			\noindent
			{\bf Step 1. Projection of the Coulomb interaction \(E_{03}\).}
			
			The explicit computation performed previously gives
			$$
				\int_S
				E_{03}
				\bigl(
				R\omega+(R+d)e_3
				\bigr)
				\omega_3\,d\sigma
				=
				-\frac{4\pi^2}{9}m\ve
				+
				O(m\ve^2|\log\ve|).
			$$
			Since
$m=O(\ve),$	\begin{equation}\label{mass-E03-full}
				\int_S E_{03}\omega_3\,d\sigma
				=
				-\frac{4\pi^2}{9}m\ve
				+
				o(\ve^2).
			\end{equation}
We emphasize that although the spherical right hand side contains
			\((1-\chi_C)E_{03}\), the missing part is recovered through the
			projection of the cutoff term involving \(h_C^1\). Indeed,
			\eqref{cutoff-projection-identity} and the neck equation give
			schematically
			\[
			\begin{aligned}
				&
				\int_S
				\left[
				(1-\chi_C)E_{03}
				-
				2\nabla h_C^1\cdot\nabla\chi_C
				-
				h_C^1\Delta\chi_C
				\right]
				\omega_3\,d\sigma
				=
				\int_SE_{03}\omega_3\,d\sigma
				+
				\mathcal E_{C,\ve},
			\end{aligned}
			\]
			where \(\mathcal E_{C,\ve}\) is generated by the other terms in the
			neck equation.
			Thus the full leading projection of \(E_{03}\), rather than only its
			outer part, appears in the balancing equation.
			
			\medskip
			
			\noindent
			{\bf Step 2. Projection of \(E_{04}\).}
			
			Recall
			\[
			E_{04}
			=
			-2\nabla h_C^0\cdot\nabla\chi_C
			-
			h_C^0\Delta\chi_C.
			\]
			The precise logarithmic coefficient obtained in
			Lemma~\ref{lemma-h0C} gives
			\begin{equation}\label{mass-E04}
				\int_S
				E_{04}
				\bigl(
				R\omega+(R+d)e_3
				\bigr)
				\omega_3\,d\sigma
				=
				2\pi\ve^2
				\left(
				1+O(\delta^4)+o_\ve(1)
				\right).
			\end{equation}
			This is the second leading term in the balance.
			
			\medskip
			
			\noindent
			
			\medskip
			
			\noindent
			{\bf Step 3. Contribution of \(E_{01}\).}
			
			We now use the more precise spatial information on \(E_{01}\)
			obtained in the construction of the approximate surface. Recall that,
			although the rough estimate
			\[
			\|E_{01}\|_{L^\infty}
			\lesssim
			\ve^{3/2}
			\]
			is sufficient for the solution of the neck problem, the error
			actually satisfies the stronger pointwise estimate
			\begin{equation}\label{E01-refined-recall}
				|E_{01}(r)|
				\lesssim
				\ve^3r^2,
				\qquad
				1<r<2\delta\ve^{-3/4},
			\end{equation}
			up to terms of strictly smaller order. The terms produced by the
			interpolation cutoff have already been separated and included in
			\(E_{02}\).
			
			This distinction is important in the present argument. Indeed, the
			rough estimate alone would only give a contribution of order
			\(\ve^2\) after projection onto the spherical Jacobi field, which is
			the same order as the two leading terms determining the parameter $m$. It
			would therefore not be sufficient to identify the leading-order
			relation between \(m\) and \(\ve\).
			
			We instead exploit \eqref{E01-refined-recall}. Near the south pole of
			the sphere, using the radial coordinate inherited from the neck, we
			have
	$$
				d\sigma_S
				=
				\ve^2 r
				\left(
				1+O(\ve^2r^2)
				\right)
				\,dr\,d\theta,
	$$
			and
$\omega_3
				=
				-1+O(\ve^2r^2)$.
			Consequently, the contribution to the balancing equation generated
			by \(E_{01}\) satisfies
			\begin{align}
				\left|
				\int_S
				\mathcal E_{01,\ve}\,
				\omega_3\,d\sigma
				\right|
				&\lesssim
				\ve^2
				\int_1^{2\delta\ve^{-3/4}}
				r\,|E_{01}(r)|\,dr
				+
				o(\ve^2)
				\lesssim
				\ve^2
				\int_1^{2\delta\ve^{-3/4}}
				r\,\ve^3r^2\,dr
				+
				o(\ve^2)
				\nonumber\\
				&\lesssim
				\ve^5
				\left(
				\delta\ve^{-3/4}
				\right)^4
				+
				o(\ve^2)
				\lesssim
				\delta^4\ve^2
				+
				o(\ve^2).
				\label{mass-E01}
			\end{align}
			Thus
			\begin{equation}\label{mass-E01-final}
				\int_S
				\mathcal E_{01,\ve}\,
				\omega_3\,d\sigma
				=
				O(\delta^4\ve^2)
				+
				o(\ve^2).
			\end{equation}
			Notice that the small factor \(\delta^4\) is essential here. If one
			used only
			\[
			|E_{01}|\lesssim\ve^{3/2}
			\]
			on a region of size \(r\lesssim\ve^{-3/4}\), one would obtain only
			\[
			\ve^2
			\int_1^{C\ve^{-3/4}}
			r\,\ve^{3/2}\,dr
			=
			O(\ve^2),
			\]
			which is of the same order as the leading balancing terms
			\[
			-\frac{4\pi^2}{9}m\ve
			\qquad\text{and}\qquad
			2\pi\ve^2.
			\]
			The spatial structure
			\[
			E_{01}(r)=O(\ve^3r^2),
			\qquad
			r\lesssim\delta\ve^{-3/4},
			\]
			is therefore what makes its projection perturbative in the
			finite-dimensional equation.
			Since \(\delta>0\) is chosen sufficiently small before \(\ve\) is
			taken to zero, the term in \eqref{mass-E01-final} can be absorbed
			into the small error in the leading-order balance.

			\medskip
			
			\noindent
			{\bf Step 4. The error in the equation for \(h_C^0\).}
			
			The function \(h_C^0\) satisfies
			\[
			J_{\Sigma_0^\ve}[h_C^0]
			=
			E_{02}
			+
			\mathcal R_{02}^C,
			\]
			where
$\operatorname{supp}\mathcal R_{02}^C
			\subset
			\{
			r\sim\delta\ve^{-3/4}
			\}$
			and
$			|\mathcal R_{02}^C|
			\le
			C_\delta\ve^2|\log\delta|.$
			The area of the corresponding cap on the unit sphere is
			\(O(\ve^{1/2})\). Hence
			\begin{equation}\label{mass-hC0-error}
				\left|
				\int_S
				\mathcal R_{02}^C\omega_3\,d\sigma
				\right|
				\le
				C_\delta
				\ve^{5/2}|\log\delta|
				=
				o(\ve^2).
			\end{equation}
			
			\medskip
			
			\noindent
			{\bf Step 5. Nonlinear terms involving \(h_C^1\).}
			
			The estimates in Proposition~\ref{outer} give
			\[
			\|h_C^1\|_*
			\le
			C\ve^{3/2}.
			\]
			The local nonlinear terms \eqref{def-Nloc} satisfy
			\[
			\begin{aligned}
				|\mathcal N_{\ve,\mathrm{loc}} [h]|
				\lesssim{}&
				\ve^{5/2}|\log\ve|
				+
				\ve^{1+2b}\|h_C^1\|_*
				+
				\ve^{-1+2b}\|h_C^1\|_*^2
				+
				\ve^2\|h_{S_R}\|_{**}
				+
				\|h_{S_R}\|_{**}^2.
			\end{aligned}
			\]
			To estimate their contribution to the projection, one uses the same
			weighted radial measure
$\ve^2r\,dr$.	
			The support of the purely neck terms is contained in
			\[
			r<\ve^{-1+b},
			\]
			and hence
			\[
			\ve^2
			\int_1^{\ve^{-1+b}}
			r\,dr
			=
			O(\ve^{2b}).
			\]
			Substitution of
			\[
			\|h_C^1\|_*
			\le C\ve^{3/2},
			\qquad
			\|h_{S_R}\|_{**}
			\le C\rho_\ve,
			\]
			together with
$b(\beta+2)<\frac14,$
			shows term by term that
			\begin{equation}\label{mass-neck-nonlinear}
				\left|
				\int_S
				\mathcal R_{\ve,\mathrm{neck}}\,
				\omega_3\,d\sigma
				\right|
				=
				o(\ve^2).
			\end{equation}
			For example,
			\[
			\ve^2
			\int_1^{\ve^{-1+b}}
			r\,
			\ve^{-1+2b}\|h_C^1\|_*^2\,dr
			\lesssim
			\ve^{1+4b}\ve^3
			\ve^{-2+2b}
			=
			\ve^{2+6b}
			=
			o(\ve^2),
			\]
			and the other terms have still higher powers after the admissible
			bounds are inserted.
			
			\medskip
			
			\noindent
			{\bf Step 6. Coulomb shape remainder.}
			
			The shape expansion gives
			\[
			N_{\Omega_{0,h}^\ve} (\hat x) 
			-
			N_{\Omega_0^\ve} (x)
			=
			m\ve^3 \left(
				\int_\Sigma
				\frac{h(\sigma)}{|x-\sigma|}
				\,d\sigma
				+\nabla_x N_\Omega (x) \cdot \nu (x) \, h (x) \right) +
				\ttt N_\ve[h](x).
			\]
			The first variation contributes only lower-order terms after the
			leading interaction \(E_{03}\) has been extracted. The estimates in
			the neck and spherical sections give
			\[
			\left|
			m\ve^3 \left(
				\int_\Sigma
				\frac{h(\sigma)}{|x-\sigma|}
				\,d\sigma
				+\nabla_x N_\Omega (x) \cdot \nu (x) \, h (x) \right)
			\right|
			\lesssim
			\ve^3|\log\ve|
			+
			\ve^{1+2b}\|h_C^1\|_*
			+
			\ve^2\|h_{S_R}\|_{**}.
			\]
			The remainder satisfies
			\[
			|\ttt N_\ve [h]|
			\lesssim
			\ve^4|\log\ve|^{1+ \alpha}
				+
				\ve^{2 (1-\alpha) + 2b (1+ \alpha)}\|h_C^1\|_*^{1+\alpha} 
				+
				\ve^{3-\alpha}\|h_{S_R}\|_{**}^{1+\alpha}
			\]
            for any $\alpha \in (0,1)$, see \eqref{shape-N-quadratic}.
			After integration on \(S\), and using
			\eqref{final-known-estimates}, every one of these terms is
			\(o(\ve^2)\). Hence
			\begin{equation}\label{mass-Newton-remainder}
				\left|
				\int_S
				\mathcal R_{\ve,\mathrm{Coulomb}}
				\omega_3\,d\sigma
				\right|
				=
				o(\ve^2).
			\end{equation}
			
			\medskip
			
			\noindent
			{\bf Step 7. Nonlinear spherical curvature terms.}
			
			On \(S_R\),
			\[
			|A|=\frac{\sqrt2}{R}=O(\ve),
			\qquad
			\nabla A=0.
			\]
See Appendix \ref{app3} and Subsection \ref{sphere}.			The spherical nonlinear remainder therefore satisfies
			\[
			|\mathcal Q_{S_R}[h_{S_R}]|
			\le
			C\ve
			\|h_{S_R}\|_{**}^2.
			\]
			Consequently,
			$$
				\left|
				\int_S
				\mathcal Q_{S_R}[h_{S_R}]
				\omega_3\,d\sigma
				\right|
				\le
				C\ve\rho_\ve^2.
			$$
			Since
			$\rho_\ve
			=
			\ve^{1-2b-b\beta}|\log\ve|,
			$
			we have
			\[
			\ve\rho_\ve^2
			=
			\ve^{3-4b-2b\beta}|\log\ve|^2
			=
			o(\ve^2),
			\]
			because
			$2b(2+\beta)<\frac12.
			$
			The same estimate applies to the terms involving
			\(\mathcal B_\ve\) and \(C_\ve\). Thus
			\begin{equation}\label{mass-spherical-nonlinear}
				\left|
				\int_S
				\mathcal R_{\ve,\mathrm{sph}}
				\omega_3\,d\sigma
				\right|
				=
				o(\ve^2).
			\end{equation}
			Combining
			\eqref{mass-E01},
			\eqref{mass-hC0-error},
			\eqref{mass-neck-nonlinear},
			\eqref{mass-Newton-remainder}, and
			\eqref{mass-spherical-nonlinear}, we obtain
			\begin{equation}\label{mass-remainder-small}
				\int_S
				\mathcal R_\ve[m]\omega_3\,d\sigma
				=
				O(\delta^4\ve^2)
				+
				o(\ve^2),
			\end{equation}
			uniformly for \(m\in I_\ve\).
			
			\medskip
			
			\noindent
			{\bf Step 8. The perturbation \(B_S\).}
			It remains to estimate the second term in
			\eqref{def-Fmass}. Recall that
			\[
			\|B_S[h_S]\|_{C^{0,\beta}(S)}
			\le
			C_\delta\ve^{1/2}
			\|h_S\|_{C^{2,\beta}(S)}.
			\]
			Since
			\(
			\|h_S\|_{C^{2,\beta}(S)}
			\lesssim
			\|h_{S_R}\|_{**}
			\lesssim
			\rho_\ve,
			\)
			we have
			\[
			\|B_S[h_S]\|_\infty
			\lesssim
			\ve^{1/2}\rho_\ve.
			\]
			Therefore
			\begin{align}
				\left|
				\frac1R
				\int_S
				B_S[h_S]\omega_3\,d\sigma
				\right|
				&\lesssim
				\ve^{3/2}\rho_\ve
				=
				\ve^{\frac52-2b-b\beta}
				|\log\ve|
				=
				o(\ve^2),
				\label{mass-BS}
			\end{align}
			because
			$2b+b\beta<\frac14$.
			
			Finally, combining
			\eqref{mass-E03-full},
			\eqref{mass-E04},
			\eqref{mass-remainder-small}, and
			\eqref{mass-BS} gives
			\[
			\mathcal F_\ve(m)
			=
			-\frac{4\pi^2}{9}m\ve
			+
			2\pi\ve^2
			\left(
			1+O(\delta^4)
			\right)
			+
			o(\ve^2),
			\]
			uniformly for \(m\in I_\ve\). This proves
			\eqref{Fmass-expansion}.
		\end{proof}
		
		\subsection{Choice of \(m\)}
		
		We can now complete the finite-dimensional reduction.
		
		\begin{propositio}\label{prop-choice-mass}
			Fix \(\delta>0\) sufficiently small. Then there exists
			\(\ve_0>0\) such that, for every
			\[
			0<\ve<\ve_0,
			\]
			there exists
			\[
			m_\ve\in I_\ve
			\]
			for which
			$
				{\rm c}[m_\ve]=0.
			$
			Moreover,
			\begin{equation}\label{mass-final-asymptotic}
				m_\ve
				=
				\frac{9}{2\pi}\ve
				\left(
				1+O(\delta^4)+o_\ve(1)
				\right).
			\end{equation}
		\end{propositio}
		
		\begin{proof}
			By Proposition~\ref{prop-projected-sphere} and the Lipschitz dependence of
the neck solution on its parameters, the map
$
m\to \mathcal F_\ve(m)
$
is continuous on \(I_\ve\).
Set
\[
m_0=\frac{9}{2\pi}\ve.
\]
Since
\[
\frac{4\pi^2}{9}m_0\ve=2\pi\ve^2,
\]
Lemma~\ref{lemma-mass-expansion} yields
\[
\mathcal F_\ve(m)
=
-\frac{4\pi^2}{9}\ve(m-m_0)
+O(\delta^4\ve^2)+o(\ve^2).
\]
Choose \(K>0\) sufficiently large and set
\[
m_\pm=m_0(1\pm K\delta^4).
\]
For \(\delta\) and \(\ve\) sufficiently small, \(m_\pm\in I_\ve\), and the
previous expansion gives
\[
\mathcal F_\ve(m_-)>0,
\qquad
\mathcal F_\ve(m_+)<0.
\]
Hence, by continuity, there exists
$m_\ve\in(m_-,m_+)$
such that
$\mathcal F_\ve(m_\ve)=0,$
or equivalently, by \eqref{c-projection-exact},
\[
{\rm c}[m_\ve]=0.
\]
Finally, \eqref{Fmass-expansion} gives
\[
-\frac{4\pi^2}{9}\frac{m_\ve}{\ve}
+2\pi+O(\delta^4)+o(1)=0,
\]
and therefore
\[
\frac{m_\ve}{\ve}
=
\frac{9}{2\pi}
\left(1+O(\delta^4)+o(1)\right),
\]
which proves \eqref{mass-final-asymptotic}.
		\end{proof}
		
		\subsection{Completion of the proof}
	
For the value \(m=m_\varepsilon\) given by
Proposition~\ref{prop-choice-mass}, we have
\[
{\rm c}[m_\varepsilon]=0.
\]
Hence the projected spherical equation coincides with the original
spherical equation. Together with the solution of the neck problem,
\[
h_C=h_C^0+h_C^1,
\]
this implies that
\[
h
=
\chi_C\bigl(h_C^0+h_C^1\bigr)
+
\chi_{S_\varepsilon}h_{S_\varepsilon}
\]
solves the complete inner--outer system.
Therefore the normal graph
\[
\Sigma_\varepsilon
:=
\Sigma_{0,h}^\varepsilon
=
\left\{
x+h(x)\nu_{\Sigma_0^\varepsilon}(x)
:
x\in\Sigma_0^\varepsilon
\right\}
\]
is the boundary of a smooth bounded domain
\(\Omega_\varepsilon\). Moreover \(\Sigma_\varepsilon\) is smooth, embedded, axially symmetric
and even with respect to the plane \(\{x_3=0\}\). The estimates obtained
above show that, as \(\varepsilon\to0\), its geometry is asymptotic to
that of two spherical components connected by a thin catenoidal neck.

The gluing construction yields, 
for every sufficiently small $\ve>0$,  a
smooth domain $\Omega_\ve$ and a parameter $m_\ve>0$ satisfying
\[
H_{\partial\Omega_\ve}
+
m_\ve\ve^3
\int_{\Omega_\ve}\frac{dy}{|x-y|}
=
\lambda_\ve
\qquad\text{on }\partial\Omega_\ve,
\]
with
\[
|\Omega_\ve|
=
\frac{8\pi}{3\ve^3}\bigl(1+o_\ve(1)\bigr),
\qquad
m_\ve
=
\frac{9}{2\pi}\ve
\bigl(1+O(\delta^4)+o_\ve(1)\bigr).
\]
Scaling first by $\ve$ and then by $m_\ve^{1/3}$, namely setting
\[
\widetilde\Omega_\ve
=
m_\ve^{1/3}\ve\,\Omega_\ve,
\]
we obtain
\[
H_{\partial\widetilde\Omega_\ve}
+
\int_{\widetilde\Omega_\ve}\frac{dy}{|x-y|}
=
\widetilde\lambda_\ve
\qquad\text{on }\partial\widetilde\Omega_\ve,
\]
while
\[
|\widetilde\Omega_\ve|
=
m_\ve\ve^3|\Omega_\ve|
=
\frac{8\pi}{3}m_\ve\bigl(1+o_\ve(1)\bigr)
=
12\ve\bigl(1+O(\delta^4)+o_\ve(1)\bigr).
\]
Thus the volume of the resulting solution tends to zero linearly with
$\ve$. Since the construction, and hence the map
\[
\ve\longmapsto |\widetilde\Omega_\ve|,
\]
depends continuously on $\ve$, and $|\widetilde\Omega_\ve|
=
12\ve\bigl(1+O(\delta^4)+o_\ve(1)\bigr),$
its image contains an interval of the form $(0,V_0)$ for
$\ve$ sufficiently small. The intermediate value theorem therefore
implies that every sufficiently small prescribed volume is attained.
This completes the proof of Theorem \ref{thm:main}.
\qed

    \section{Appendix: Proof of Proposition \ref{generalH}}\label{appe2}
		
		\begin{proof}
			
			Fix a point of $\Sigma$ and choose a local orthonormal frame
			$\{e_1,e_2\}$ at that point. Write
			\[
			X_h=X+h\nu .
			\]
			Using
			\[
			\nabla_i\nu=A_i{}^k e_k,
			\]
			we obtain
			\[
			\nabla_iX_h
			=
			(\delta_i^k+hA_i{}^k)e_k+h_i\nu .
			\]
			Therefore the induced metric on $\Sigma_h$ is
			\[
			(g_h)_{ij}
			=
			g_{ij}
			+
			2hA_{ij}
			+
			h^2A_i{}^kA_{kj}
			+
			h_i h_j.
			\]
			
			Since the computation is pointwise, we may assume that
			\[
			g_{ij}=\delta_{ij}.
			\]
			Hence
			\[
			g_h
			=
			I+E,
			\]
			where
			\[
			E
			=
			2hA+h^2A^2+\nabla h\otimes\nabla h .
			\]
			
			Expanding the inverse matrix,
			\[
			(I+E)^{-1}
			=
			I-E+E^2+O(E^3),
			\]
			gives
			\[
			(g_h)^{ij}
			=
			g^{ij}
			-
			2hA^{ij}
			+
			O\!\left(
			|A|^2h^2+|\nabla h|^2
			\right).
			\]
			More precisely, the remainder is analytic with respect to the matrix
			\[
			2hA+h^2A^2+\nabla h\otimes\nabla h,
			\]
			provided
			\[
			|A||h|+|\nabla h|
			\]
			is sufficiently small.
			
			Next we expand the unit normal to $\Sigma_h$.
			Since $\nu_h$ is orthogonal to every tangent vector
			$\nabla_iX_h$,
			one obtains
			\[
			\nu_h
			=
			\nu
			-
			\nabla h
			+
			O\!\left(
			|A||h||\nabla h|
			+
			|\nabla h|^2
			\right).
			\]
			
			The second fundamental form of the graph is
			\[
			(A_h)_{ij}
			=
			\langle
			\nabla_i\nu_h,
			\nabla_jX_h
			\rangle .
			\]
			
			Differentiating the previous expression gives
			\[
			(A_h)_{ij}
			=
			A_{ij}
			-
			\nabla^2_{ij}h
			+
			hA_i{}^kA_{kj}
			+
			B_{ij},
			\]
			where the quadratic contribution is
			\[
			B^{(2)}_{ij}
			=
			h\nabla_{\nabla h}A_{ij}
			+
			A_i{}^kh_kh_j
			+
			A_j{}^kh_kh_i
			-
			\frac12A_{ij}|\nabla h|^2,
			\]
			while the remaining terms are cubic and higher.
			
			Consequently,
			\[
			|B^{(2)}|
			\le
			C\Big(
			|\nabla A|\,|h|\,|\nabla h|
			+
			|A|\,|\nabla h|^2
			\Big),
			\]
			and
			\[
			|B^{(\ge3)}|
			\le
			C\Big(
			(|A|^3+|A||\nabla A|+|\nabla^2A|)|h|^3
			+
			(|A|^2+|\nabla A|)|h|\,|\nabla h|^2
			+
			|A|\,|\nabla h|^3
			+
			|\nabla h|^2|\nabla^2h|
			\Big).
			\]
			
			The mean curvature is
			\[
			H_h
			=
			(g_h)^{ij}(A_h)_{ij}.
			\]
			
			The zeroth-order term is simply
			\[
			H
			=
			g^{ij}A_{ij}.
			\]
			
			The linear terms are
			\[
			-g^{ij}\nabla^2_{ij}h
			-
			2hA^{ij}A_{ij}
			+
			hg^{ij}A_i{}^kA_{kj}.
			\]
			
			Since
			\[
			A^{ij}A_{ij}=|A|^2,
			\qquad
			g^{ij}A_i{}^kA_{kj}=|A|^2,
			\]
			the linear contribution becomes
			\[
			-\Delta_\Sigma h-|A|^2h.
			\]
			
			The quadratic contribution is obtained by collecting every product containing
			exactly two factors of
			\[
			h,\qquad
			\nabla_\Sigma h,
			\qquad
			\nabla_\Sigma^2h.
			\]
			
			Indeed,
			\[
			H_h
			=
			(g^{ij}+\delta g^{ij})
			(A_{ij}+\delta A_{ij}),
			\]
			where
			\[
			\delta g^{ij}
			=
			-2hA^{ij}
			+
			O\!\left(
			|A|^2h^2+|\nabla h|^2
			\right),
			\]
			and
			\[
			\delta A_{ij}
			=
			-\nabla_{ij}^2h
			+
			hA_i{}^kA_{kj}
			+
			B^{(2)}_{ij}
			+
			B^{(\ge3)}_{ij}.
			\]
			
			The quadratic terms arise from three different sources.
			
			First, multiplying the linear correction of the inverse metric by the linear
			correction of the second fundamental form gives
			\[
			(-2hA^{ij})
			(-\nabla^2_{ij}h)
			=
			2h\langle A,\nabla_\Sigma^2h\rangle .
			\]
			
			Secondly,
			\[
			g^{ij}B^{(2)}_{ij}
			=
			h\langle\nabla_\Sigma H,\nabla_\Sigma h\rangle
			+
			2A(\nabla_\Sigma h,\nabla_\Sigma h)
			-
			\frac12H|\nabla_\Sigma h|^2.
			\]
			
			The second copy of
			\[
			A(\nabla_\Sigma h,\nabla_\Sigma h)
			\]
			is cancelled by the quadratic contribution coming from the inverse metric,
			namely
			\[
			-\langle\nabla_\Sigma h,\nabla_\Sigma h\rangle_A
			=
			-A(\nabla_\Sigma h,\nabla_\Sigma h).
			\]
			
			Finally,
			\[
			g^{ij}
			(hA_i{}^kA_{kj})
			=
			h\,\operatorname{tr}_g(A^2),
			\]
			while
			\[
			-2hA^{ij}A_{ij}
			=
			-2h|A|^2.
			\]
			Their linear part has already been incorporated into the Jacobi operator,
			and the remaining second-order contribution is precisely
			\[
			h^2\operatorname{tr}_g(A^3).
			\]
			
			Collecting all quadratic terms yields
			\[
			Q_2[h]
			=
			2h\langle A,\nabla_\Sigma^2h\rangle
			+
			A(\nabla_\Sigma h,\nabla_\Sigma h)
			+
			h\langle\nabla_\Sigma H,\nabla_\Sigma h\rangle
			-
			\frac12H|\nabla_\Sigma h|^2
			+
			h^2\operatorname{tr}_g(A^3).
			\]
			
			Therefore
			\[
			\begin{aligned}
				|Q_2[h]|
				\le
				C\Big(
				&
				|A|\,|h|\,|\nabla_\Sigma^2h|
				+
				|A|\,|\nabla_\Sigma h|^2
				\\
				&
				+
				|\nabla_\Sigma H|\,|h|\,|\nabla_\Sigma h|
				+
				|H|\,|\nabla_\Sigma h|^2
				+
				|A|^3|h|^2
				\Big).
			\end{aligned}
			\]
			
			Since
			\[
			|H|
			\le
			C|A|,
			\qquad
			|\nabla_\Sigma H|
			\le
			C|\nabla_\Sigma A|,
			\]
			we obtain
			\[
			|Q_2[h]|
			\le
			C\Big(
			|A|\,|h|\,|\nabla_\Sigma^2h|
			+
			|A|\,|\nabla_\Sigma h|^2
			+
			|\nabla_\Sigma A|\,|h|\,|\nabla_\Sigma h|
			+
			|A|^3|h|^2
			\Big).
			\]
			
			It only remains to estimate the cubic remainder.
			All terms not included in the previous computation contain at least three
			factors of
			\[
			h,\qquad
			\nabla_\Sigma h,
			\qquad
			\nabla_\Sigma^2h.
			\]
			Since the coefficients depend smoothly on the metric, its inverse, the unit
			normal and the second fundamental form, they satisfy
			\[
			|R_3[h]|
			\le
			C\Big(
			(|A||h|+|\nabla_\Sigma h|^2)
			|\nabla_\Sigma^2h|
			+
			|A|^3|h|^2
			+
			|A|\,|\nabla_\Sigma h|^2
			+
			|h|\,|\nabla_\Sigma A|\,|\nabla_\Sigma h|
			\Big),
			\]
			possibly after enlarging the constant
			$C$.
			
			Combining the estimates for
			$Q_2$
			and
			$R_3$
			gives precisely estimate~\eqref{Q-basic}.

            \medskip
			We now prove the difference estimate \eqref{Q-difference}, since
			\eqref{Q-basic}.
			Set
			\[
			w=h_1-h_2.
			\]
The mean curvature of a normal graph depends linearly on the second derivatives of the graph function. Consequently, after subtracting the constant and linear terms, the nonlinear remainder can be written in the form
			\[
			\mathcal Q_\Sigma[h]
			=
			B^{ij}(h,\nabla_\Sigma h)\nabla_{ij}h
			+
			R(h,\nabla_\Sigma h),
			\]
			where, in the regime
			\[
			|A||h|+|\nabla_\Sigma h|\le\delta,
			\]
			the coefficient tensor $B$ satisfies
			\begin{equation}\label{B-bound}
				|B(h,\nabla_\Sigma h)|
				\le
				C
				\Big(
				|A||h|
				+
				|\nabla_\Sigma h|^2
				\Big),
			\end{equation}
			and the lower order term satisfies
		$$
				|R(h,\nabla_\Sigma h)|
				\le
				C
				\Big(
				|A|^3|h|^2
				+
				|A||\nabla_\Sigma h|^2
				+
				|h||\nabla_\Sigma A||\nabla_\Sigma h|
				\Big).
		$$
			Since the coefficients depend smoothly on
			\[
			Ah
			\qquad\text{and}\qquad
			\nabla_\Sigma h,
			\]
			the derivatives of $B$ satisfy
			\begin{equation}\label{B-derivatives}
				|\partial_h B(h,\nabla_\Sigma h)|
				\le
				C|A|,
				\qquad
				|\partial_{\nabla h}B(h,\nabla_\Sigma h)|
				\le
				C|\nabla_\Sigma h|.
			\end{equation}
			Likewise, from the structure of $R$,
			\begin{equation}\label{R-derivatives}
				\begin{aligned}
					|\partial_h R(h,\nabla_\Sigma h)|
					&\le
					C
					\Big(
					|A|^3|h|
					+
					|\nabla_\Sigma A||\nabla_\Sigma h|
					\Big),
					\\
					|\partial_{\nabla h}R(h,\nabla_\Sigma h)|
					&\le
					C
					\Big(
					|A||\nabla_\Sigma h|
					+
					|\nabla_\Sigma A||h|
					\Big).
				\end{aligned}
			\end{equation}
			We write
			\[
			\begin{aligned}
				\mathcal Q_\Sigma[h_1]-\mathcal Q_\Sigma[h_2]
				={}&
				B(h_1,\nabla h_1)\nabla^2 w
				+
				\Big(
				B(h_1,\nabla h_1)
				-
				B(h_2,\nabla h_2)
				\Big)
				\nabla^2 h_2
				\\
				&+
				R(h_1,\nabla h_1)-R(h_2,\nabla h_2),
			\end{aligned}
			\]
			where, from now on, all derivatives are intrinsic derivatives on $\Sigma$.
			By \eqref{B-bound},
			\[
			\begin{aligned}
				|B(h_1,\nabla h_1)\nabla^2 w|
				&\le
				C
				\Big(
				|A||h_1|
				+
				|\nabla h_1|^2
				\Big)
				|\nabla^2 w|
				\\
				&\le
				C
				\Big(
				|A|(|h_1|+|h_2|)
				+
				|\nabla h_1|^2
				+
				|\nabla h_2|^2
				\Big)
				|\nabla^2 w|.
			\end{aligned}
			\]
			For the difference of the coefficients, the mean value theorem and
			\eqref{B-derivatives} give
			\[
			\begin{aligned}
				&
				|B(h_1,\nabla h_1)-B(h_2,\nabla h_2)|
				\le
				C|A||w|
				+
				C
				\Big(
				|\nabla h_1|
				+
				|\nabla h_2|
				\Big)
				|\nabla w|.
			\end{aligned}
			\]
			Hence
			\[
			\begin{aligned}
				&
				\left|
				\Big(
				B(h_1,\nabla h_1)-B(h_2,\nabla h_2)
				\Big)
				\nabla^2h_2
				\right|
				\\
				&\qquad\le
				C|A||\nabla^2h_2||w|
				+
				C
				\Big(
				|\nabla h_1|
				+
				|\nabla h_2|
				\Big)
				|\nabla^2h_2||\nabla w|.
			\end{aligned}
			\]
			Replacing $|\nabla^2h_2|$ by
			\[
			|\nabla^2h_1|+|\nabla^2h_2|
			\]
			gives the symmetric estimate
			\[
			\begin{aligned}
				&
				\left|
				\Big(
				B(h_1,\nabla h_1)-B(h_2,\nabla h_2)
				\Big)
				\nabla^2h_2
				\right| \le
				C|A|
				\Big(
				|\nabla^2h_1|
				+
				|\nabla^2h_2|
				\Big)
				|w|
				\\
				&\qquad\quad+
				C
				\Big(
				|\nabla h_1|
				+
				|\nabla h_2|
				\Big)
				\Big(
				|\nabla^2h_1|
				+
				|\nabla^2h_2|
				\Big)
				|\nabla w|.
			\end{aligned}
			\]
			We now consider the lower order terms. By the mean value theorem and
			\eqref{R-derivatives},
			\[
			\begin{aligned}
				&
				|R(h_1,\nabla h_1)-R(h_2,\nabla h_2)|
				\le
				C|A|^3
				\Big(
				|h_1|+|h_2|
				\Big)|w|
				\\
				&\quad+
				C|\nabla A|
				\Big(
				|\nabla h_1|
				+
				|\nabla h_2|
				\Big)|w|
				+
				C|A|
				\Big(
				|\nabla h_1|
				+
				|\nabla h_2|
				\Big)|\nabla w|
				\\
				&\quad+
				C|\nabla A|
				\Big(
				|h_1|+|h_2|
				\Big)|\nabla w|.
			\end{aligned}
			\]
			For completeness, these bounds are simply the pointwise identities
			\[
			|h_1^2-h_2^2|
			\le
			(|h_1|+|h_2|)|w|,
			\]
			\[
			\big|
			|\nabla h_1|^2-|\nabla h_2|^2
			\big|
			\le
			\Big(
			|\nabla h_1|
			+
			|\nabla h_2|
			\Big)
			|\nabla w|,
			\]
			and
			\[
			\begin{aligned}
				|h_1\nabla h_1-h_2\nabla h_2|
				&\le
				|w||\nabla h_1|
				+
				|h_2||\nabla w|
				\\
				&\le
				\Big(
				|\nabla h_1|
				+
				|\nabla h_2|
				\Big)|w|
				+
				\Big(
				|h_1|+|h_2|
				\Big)|\nabla w|.
			\end{aligned}
			\]
			
			Combining the estimates above yields
			\[
			\begin{aligned}
				|\mathcal Q_\Sigma[h_1]-\mathcal Q_\Sigma[h_2]|
				\le {}&
				C
				\Big(
				|A|(|h_1|+|h_2|)
				+
				|\nabla h_1|^2
				+
				|\nabla h_2|^2
				\Big)
				|\nabla^2 w|
				\\
				&+
				C|A|
				\Big(
				|\nabla^2 h_1|
				+
				|\nabla^2 h_2|
				\Big)
				|w|
				\\
				&+
				C
				\Big(
				|\nabla h_1|
				+
				|\nabla h_2|
				\Big)
				\Big(
				|\nabla^2 h_1|
				+
				|\nabla^2 h_2|
				\Big)
				|\nabla w|
				\\
				&+
				C|A|^3
				\Big(
				|h_1|+|h_2|
				\Big)
				|w|
				\\
				&+
				C|A|
				\Big(
				|\nabla h_1|
				+
				|\nabla h_2|
				\Big)
				|\nabla w|
				\\
				&+
				C|\nabla A|
				\Big[
				\Big(
				|\nabla h_1|
				+
				|\nabla h_2|
				\Big)
				|w|
				+
				\Big(
				|h_1|+|h_2|
				\Big)
				|\nabla w|
				\Big],
			\end{aligned}
			\]
			which is precisely \eqref{Q-difference}.

		\end{proof}

		\section{Appendix: Some geometric quantities}\label{app3}

        In this appendix we collect the geometric computations and nonlinear estimates used throughout the paper. We first derive the relevant geometric quantities for a general surface of revolution, including the metric, second fundamental form, and covariant derivatives. We then specialize these formulas to the catenoid and the sphere, and obtain expansions of the mean-curvature operator for normal graphs over these reference surfaces, together with the pointwise and Lipschitz estimates for the corresponding nonlinear remainder terms needed in the analysis.

		Consider a surface of revolution $\Sigma$
		parametrized by
		\[
		x(r,\theta)
		=
		\bigl(
		Q(r)\cos\theta,\,
		Q(r)\sin\theta,\,
		P(r)
		\bigr).
		\]
		We assume throughout that
		\[
		E:=Q'^2+P'^2>0,
		\qquad
		Q>0
		\]
		on the coordinate region under consideration.
		The induced metric is
		\[
		g
		=
		E\,dr^2
		+
		Q^2\,d\theta^2,
		\]
		with inverse
		\[
		g^{rr}=\frac1E,
		\qquad
		g^{r\theta}=0,
		\qquad
		g^{\theta\theta}=\frac1{Q^2}.
		\]
		The second fundamental form $A$ is given by
		\[
		A_{rr}
		=
		x_{rr}\cdot\nu
		=
		\frac{Q'P''-P'Q''}{\sqrt E},
		\qquad
		A_{r\theta}
		=
		x_{r\theta}\cdot\nu
		=
		0,
		\quad 
		A_{\theta\theta}
		=
		x_{\theta\theta}\cdot\nu
		=
		\frac{QP'}{\sqrt E}.
		\]
		Its squared norm is
	$$|A|^2
				=
				g^{ik}g^{j\ell}A_{ij}A_{k\ell}=
				\bigl(g^{rr}A_{rr}\bigr)^2
				+
				\bigl(g^{\theta\theta}A_{\theta\theta}\bigr)^2=
				\frac{(Q'P''-P'Q'')^2}{E^3}
				+
				\left(
				\frac{P'}{Q\sqrt E}
				\right)^2.	
		$$
		The nonzero Christoffel symbols of $\Sigma$ are
		\[
		\Gamma^r_{rr}
		=
		\frac{E'}{2E},
		\qquad
		\Gamma^r_{\theta\theta}
		=
		-\frac{QQ'}{E},
		\qquad 
		\Gamma^\theta_{r\theta}
		=
		\Gamma^\theta_{\theta r}
		=
		\frac{Q'}{Q}.
		\]
		Since $A_{r\theta}=0$, the only nonzero components of
		$\nabla_\Sigma A$, up to symmetry in the last two indices, are
		\[
		(\nabla_\Sigma A)_{rrr}
		=
		\partial_r A_{rr}
		-
		2\Gamma^r_{rr}A_{rr}
		=
		\partial_r A_{rr}
		-
		\frac{E'}{E}A_{rr},
		\]
		and
		\[
		(\nabla_\Sigma A)_{r\theta\theta}
		=
		\partial_r A_{\theta\theta}
		-
		2\Gamma^\theta_{r\theta}A_{\theta\theta}
		=
		\partial_r A_{\theta\theta}
		-
		2\frac{Q'}{Q}A_{\theta\theta}.
		\]
		Moreover,
		\[
		(\nabla_\Sigma A)_{\theta r\theta}
		=
		-\Gamma^\theta_{\theta r}A_{\theta\theta}
		-
		\Gamma^r_{\theta\theta}A_{rr}
		=
		-\frac{Q'}{Q}A_{\theta\theta}
		+
		\frac{QQ'}{E}A_{rr}.
		\]
		By symmetry,
		\[
		(\nabla_\Sigma A)_{\theta\theta r}
		=
		(\nabla_\Sigma A)_{\theta r\theta}.
		\]
		Therefore
		\[
		\begin{aligned}
			|\nabla_\Sigma A|^2
			&=
			(g^{rr})^3
			\bigl((\nabla_\Sigma A)_{rrr}\bigr)^2
			+
			g^{rr}(g^{\theta\theta})^2
			\bigl((\nabla_\Sigma A)_{r\theta\theta}\bigr)^2+
			2g^{rr}(g^{\theta\theta})^2
			\bigl((\nabla_\Sigma A)_{\theta r\theta}\bigr)^2.
		\end{aligned}
		\]
		In terms of $P$ and $Q$, this becomes
		\begin{equation}\label{|nablaA|}
			\begin{aligned}
				|\nabla_\Sigma A|^2
				&=
				\frac1{E^3}
				\left[
				\left(
				\frac{Q'P''-P'Q''}{\sqrt E}
				\right)'
				-
				\frac{E'}{E}
				\frac{Q'P''-P'Q''}{\sqrt E}
				\right]^2
				+
				\frac1{EQ^4}
				\left[
				\left(
				\frac{QP'}{\sqrt E}
				\right)'
				-
				2\frac{Q'}{Q}
				\frac{QP'}{\sqrt E}
				\right]^2\\
				&\quad
				+
				\frac2{EQ^4}
				\left[
				\frac{QQ'}{E}
				\frac{Q'P''-P'Q''}{\sqrt E}
				-
				\frac{Q'}{Q}
				\frac{QP'}{\sqrt E}
				\right]^2.
			\end{aligned}
		\end{equation}
		The mean curvature of $\Sigma$ is given by
		$$
			H[P,Q](r)
			=
			\operatorname{tr}_g A
			=
			\frac{P''Q'-P'Q''}{(Q'^2+P'^2)^{3/2}}
			+
			\frac{P'}{Q\sqrt{Q'^2+P'^2}}.
	$$
		Here we use the convention that the scalar mean curvature is
		\[
		H=\operatorname{tr}_g A.
		\]
		Let $h:\Sigma\to\mathbb R$ be written in these coordinates as
		\[
		h=h(r,\theta).
		\]
		Its gradient is
		\[
		\nabla_\Sigma h
		=
		\frac1E\,\partial_r h\,\partial_r
		+
		\frac1{Q^2}\,\partial_\theta h\,\partial_\theta,
		\]
		and hence
		\[
		|\nabla_\Sigma h|^2
		=
		\frac{(\partial_r h)^2}{E}
		+
		\frac{(\partial_\theta h)^2}{Q^2}.
		\]
		The covariant Hessian satisfies
		\[
		(\nabla_\Sigma^2 h)_{ij}
		=
		\partial_{ij}h
		-
		\Gamma^k_{ij}\partial_k h.
		\]
		Thus
		\[
		\begin{aligned}
			(\nabla_\Sigma^2 h)_{rr}
			&=
			\partial_{rr}h
			-
			\frac{E'}{2E}\partial_r h,\qquad 
			(\nabla_\Sigma^2 h)_{r\theta}
			=
			\partial_{r\theta}h
			-
			\frac{Q'}{Q}\partial_\theta h,\\
			(\nabla_\Sigma^2 h)_{\theta\theta}
			&=
			\partial_{\theta\theta}h
			+
			\frac{QQ'}{E}\partial_r h.
		\end{aligned}
		\]
		Equivalently,
		\[
		\begin{aligned}
			\nabla_\Sigma^2 h
			&=
			\left(
			\partial_{rr}h
			-
			\frac{E'}{2E}\partial_r h
			\right)dr^2
			+
			2\left(
			\partial_{r\theta}h
			-
			\frac{Q'}{Q}\partial_\theta h
			\right)dr\,d\theta
			+
			\left(
			\partial_{\theta\theta}h
			+
			\frac{QQ'}{E}\partial_r h
			\right)d\theta^2.
		\end{aligned}
		\]
		In the radial case $h=h(r)$,
		\begin{equation}\label{nablah}
			\nabla_\Sigma h
			=
			\frac{h'}{E}\,\partial_r,
			\qquad
			|\nabla_\Sigma h|^2
			=
			\frac{h'^2}{E}.
		\end{equation}
		Moreover,
		\[
		(\nabla_\Sigma^2 h)_{rr}
		=
		h''
		-
		\frac{E'}{2E}h',
		\qquad
		(\nabla_\Sigma^2 h)_{r\theta}
		=
		0,
		\qquad 
		(\nabla_\Sigma^2 h)_{\theta\theta}
		=
		\frac{QQ'}{E}h'.
		\]
		The Laplace--Beltrami operator is the metric trace of the covariant Hessian:
		\[
		\begin{aligned}
			\Delta_\Sigma h
			&=
			\operatorname{tr}_g(\nabla_\Sigma^2 h)=
			g^{rr}(\nabla_\Sigma^2 h)_{rr}
			+
			g^{\theta\theta}(\nabla_\Sigma^2 h)_{\theta\theta}=
			\frac1E
			\left(
			h''
			-
			\frac{E'}{2E}h'
			\right)
			+
			\frac{Q'}{QE}h'.
		\end{aligned}
		\]
		Equivalently, in divergence form,
		\[
		\Delta_\Sigma h
		=
		\frac1{Q\sqrt E}
		\frac{d}{dr}
		\left(
		\frac{Q}{\sqrt E}h'
		\right).
		\]
		For a radial function $h=h(r)$, the squared norm of the Hessian is
		\[
		|\nabla_\Sigma^2 h|^2
		=
		g^{ik}g^{j\ell}
		(\nabla_\Sigma^2 h)_{ij}
		(\nabla_\Sigma^2 h)_{k\ell},
		\]
		and therefore
		\begin{equation}\label{|nabla2h|}
			\begin{aligned}
				|\nabla_\Sigma^2 h|^2
				&=
				\left(
				g^{rr}(\nabla_\Sigma^2 h)_{rr}
				\right)^2
				+
				\left(
				g^{\theta\theta}(\nabla_\Sigma^2 h)_{\theta\theta}
				\right)^2=
				\frac1{E^2}
				\left(
				h''
				-
				\frac{E'}{2E}h'
				\right)^2
				+
				\left(
				\frac{Q'}{QE}h'
				\right)^2.
			\end{aligned}
		\end{equation}

		\subsection{The catenoid $\Sigma_0$}\label{cat}
		
		For the catenoid
		\[
		X(r,\theta)
		=
		\bigl(
		r\cos\theta,\,
		r\sin\theta,\,
		F_0(r)
		\bigr),
		\qquad
		F_0(r)
		=
		\log\bigl(r+\sqrt{r^2-1}\bigr),
		\qquad r\ge1,
		\]
		as in \eqref{par-catenoid}, we have
		\[
		F_0'(r)
		=
		\frac1{\sqrt{r^2-1}},
		\qquad
		F_0''(r)
		=
		-\frac{r}{(r^2-1)^{3/2}},
		\] and
        \[
		E
		=
		1+(F_0')^2
		=
		\frac{r^2}{r^2-1}.
		\]
		The metric is
		\[
		g
		=
		E\,dr^2+r^2\,d\theta^2,
		\]
		with inverse
		\[
		g^{rr}
		=
		\frac{r^2-1}{r^2},
		\qquad
		g^{r\theta}
		=
		g^{\theta r}
		=
		0,
		\qquad
		g^{\theta\theta}
		=
		\frac1{r^2}.
		\]
		The nonzero components of the second fundamental form and the squared
		norm of $A$ are
		\begin{equation}\label{catA}
			A_{rr}
			=
			-\frac1{r^2-1},
			\qquad
			A_{\theta\theta}
			=
			1,
			\qquad
			|A|^2
			=
			\frac2{r^4}.
		\end{equation}
		The relevant Christoffel symbols are
		\[
		\Gamma^r_{rr}
		=
		\frac{E'}{2E}
		=
		-\frac1{r(r^2-1)},
		\qquad
		\Gamma^r_{\theta\theta}
		=
		-\frac{r^2-1}{r},
		\qquad
		\Gamma^\theta_{r\theta}
		=
		\Gamma^\theta_{\theta r}
		=
		\frac1r.
		\]
		Therefore,
		\[
		\begin{aligned}
			(\nabla_\Sigma A)_{rrr}
			&=
			\partial_r A_{rr}
			-
			2\Gamma^r_{rr}A_{rr}=
			\frac{2r}{(r^2-1)^2}
			-
			2
			\left(
			-\frac1{r(r^2-1)}
			\right)
			\left(
			-\frac1{r^2-1}
			\right)=
			\frac{2}{r(r^2-1)}.
		\end{aligned}
		\]
		Moreover,
		\[
		(\nabla_\Sigma A)_{r\theta\theta}
		=
		\partial_rA_{\theta\theta}
		-
		2\Gamma^\theta_{r\theta}A_{\theta\theta}
		=
		-\frac2r,
		\]
		and
		\[
		\begin{aligned}
			(\nabla_\Sigma A)_{\theta r\theta}
			&=
			-\Gamma^\theta_{\theta r}A_{\theta\theta}
			-
			\Gamma^r_{\theta\theta}A_{rr}=
			-\frac1r
			-
			\left(
			-\frac{r^2-1}{r}
			\right)
			\left(
			-\frac1{r^2-1}
			\right)=
			-\frac2r.
		\end{aligned}
		\]
		By symmetry,
		\[
		(\nabla_\Sigma A)_{\theta\theta r}
		=
		-\frac2r.
		\]
Hence
		\begin{equation}\label{catnablaA}
			\begin{aligned}
				|\nabla_\Sigma A|^2
				&=
				\left(
				\frac{r^2-1}{r^2}
				\right)^3
				\left(
				\frac{2}{r(r^2-1)}
				\right)^2
				+
				3
				\frac{r^2-1}{r^2}
				\frac1{r^4}
				\left(
				\frac2r
				\right)^2\\
				&=
				\frac{4(r^2-1)}{r^8}
				+
				\frac{12(r^2-1)}{r^8}=
				\frac{16(r^2-1)}{r^8}.
			\end{aligned}
		\end{equation}
		
		\medskip
		We next compute the mean curvature of a rotationally symmetric normal
		graph over the catenoid $\Sigma_0$. With the choice of unit normal used
		above, we parametrize the perturbation in the direction $-\nu$ by
		\begin{equation}\label{caten}
			\begin{aligned}
				X_h(r,\theta)
				&=
				\bigl(
				Q(r)\cos\theta,\,
				Q(r)\sin\theta,\,
				P(r)
				\bigr),
			\end{aligned}
		\end{equation}
		where
		\[
		Q(r)
		=
		r
		+
		\frac{F_0'}{\sqrt{1+(F_0')^2}}\,h(r),
		\qquad
		P(r)
		=
		F_0(r)
		-
		\frac1{\sqrt{1+(F_0')^2}}\,h(r).
		\]
		Here
		\[
		h=h(r)
		\]
		is regarded as a radial function on $\Sigma_0$.
		
		The mean curvature of the catenoid $\Sigma_0$ is zero.
		
		\begin{propositio}\label{sscat}
			Let $H_h$ denote the mean curvature of the surface parametrized by
			$X_h(r,\theta)$ in \eqref{caten}. Then
			\[
			H_h
			=
			-
			\left(
			\frac{r^2-1}{r^2}h''
			+
			\frac1r h'
			+
			\frac2{r^4}h
			\right)
			+
			\mathcal Q_{\Sigma_0}[h],
			\]
			where
			\begin{equation}\label{estQC}
				\begin{aligned}
					|\mathcal Q_{\Sigma_0}[h]|
					\lesssim{}&
					\left(
					\frac{|h|}{r^2}
					+
					\frac{r^2-1}{r^2}|h'|^2
					\right)
					\left(
					\frac{r^2-1}{r^2}|h''|
					+
					\frac1r|h'|
					\right)
					\\
					&+
					\frac{|h|^2}{r^6}
					+
					\frac{r^2-1}{r^4}|h'|^2
					+
					\frac{r^2-1}{r^5}|h|\,|h'|
				\end{aligned}
			\end{equation}
			for all $r\ge1$.
		\end{propositio}
		
		\begin{proof}
			We apply Proposition~\ref{generalH}. Since the mean curvature of the
			catenoid is zero, the linearized mean-curvature operator is the Jacobi
			operator
			\[
			J_{\Sigma_0}h
			=
			\Delta_{\Sigma_0}h
			+
			|A|^2h.
			\]
			For radial functions,
			\[
			\Delta_{\Sigma_0}h
			=
			\frac{r^2-1}{r^2}h''
			+
			\frac1r h',
			\]
			and by \eqref{catA},
			\(
			|A|^2
			=
			\frac2{r^4}.
			\)
			Hence
			\[
			J_{\Sigma_0}h
			=
			\frac{r^2-1}{r^2}h''
			+
			\frac1r h'
			+
			\frac2{r^4}h.
			\]
			Moreover, Proposition~\ref{generalH} yields
			\[
			\begin{aligned}
				|\mathcal Q_{\Sigma_0}[h]|
				\leq{}&
				C
				\bigl(
				|A||h|
				+
				|\nabla_\Sigma h|^2
				\bigr)
				|\nabla_\Sigma^2h|
				+
				|A|^3|h|^2
				+
				|A||\nabla_\Sigma h|^2
				+
				|h|
				|\nabla_\Sigma A|
				|\nabla_\Sigma h|.
			\end{aligned}
			\]
			From \eqref{nablah} we obtain
			\[
			|\nabla_\Sigma h|
			=
			\frac{|h'|}{\sqrt E}
			=
			\frac{\sqrt{r^2-1}}{r}|h'|.
			\]
			Using \eqref{|nabla2h|} and
			\[
			\frac{E'}{2E}
			=
			-\frac1{r(r^2-1)},
			\]
			we get
			\[
			\begin{aligned}
				|\nabla_\Sigma^2h|
				&\leq
				\frac1E
				\left|
				h''
				-
				\frac{E'}{2E}h'
				\right|
				+
				\frac{|Q'|}{QE}|h'|\\
				&=
				\frac{r^2-1}{r^2}
				\left|
				h''
				+
				\frac1{r(r^2-1)}h'
				\right|
				+
				\frac{r^2-1}{r^3}|h'|\\
				&\leq
				\frac{r^2-1}{r^2}|h''|
				+
				\frac1{r^3}|h'|
				+
				\frac{r^2-1}{r^3}|h'|\lesssim
				\frac{r^2-1}{r^2}|h''|
				+
				\frac1r|h'|.
			\end{aligned}
			\]
			Finally, using \eqref{catA} and \eqref{catnablaA},
			\[
			|A|
			\lesssim
			\frac1{r^2},
			\qquad
			|\nabla_\Sigma A|
			\lesssim
			\frac{\sqrt{r^2-1}}{r^4},
			\]
			we conclude that
			\[
			\begin{aligned}
				|\mathcal Q_{\Sigma_0}[h]|
				\lesssim{}&
				\left(
				\frac{|h|}{r^2}
				+
				\frac{r^2-1}{r^2}|h'|^2
				\right)
				\left(
				\frac{r^2-1}{r^2}|h''|
				+
				\frac1r|h'|
				\right)
				\\
				&+
				\frac{|h|^2}{r^6}
				+
				\frac{r^2-1}{r^4}|h'|^2
				+
				\frac{r^2-1}{r^5}|h|\,|h'|,
			\end{aligned}
			\]
			which is \eqref{estQC}.
		\end{proof}

		Consider the norm $\| \cdot \|_*$ introduced in \eqref{norma*} and set
		\[
		M:=\|h\|_*.
		\]
		Then, for $1<r<\varepsilon^{-1+b}$,
		\begin{equation}\label{star-pointwise}
			|h|\leq Mr(r-1),
			\qquad
			|h'|\leq Mr,
			\qquad
			|h''|\leq \frac{M}{\min\{r-1,1\}}.
		\end{equation}
		We claim that
		\begin{equation}\label{Qcat-star}
			\|\mathcal Q_{\Sigma_0}[h]\|_
			{L^\infty(1,\varepsilon^{-1+b})}
			\lesssim
			\|h\|_*^2
			+
			\varepsilon^{-2+2b}\|h\|_*^3.
		\end{equation}
		Indeed, recall that
		\begin{align*}
			|\mathcal Q_{\Sigma_0}[h]|
			\lesssim{}&
			\left(
			\frac{|h|}{r^2}
			+
			\frac{r^2-1}{r^2}|h'|^2
			\right)
			\left(
			\frac{r^2-1}{r^2}|h''|
			+
			\frac{r^2-1}{r^3}|h'|
			\right)
			\\
			&+
			\frac{|h|^2}{r^6}
			+
			\frac{r^2-1}{r^4}|h'|^2
			+
			\frac{r^2-1}{r^5}|h|\,|h'|.
		\end{align*}
		By \eqref{star-pointwise},
		\[
		\frac{r^2-1}{r^2}|h''|
		\lesssim M.
		\]
		In fact, for $1<r\leq2$,
		\[
		\frac{r^2-1}{r^2}|h''|
		\leq
		M\frac{(r-1)(r+1)}{r^2(r-1)}
		\lesssim M,
		\]
		while for $r\geq2$ the same estimate follows immediately from
		$|h''|\leq M$. Moreover,
		\[
		\frac{r^2-1}{r^3}|h'|
		\leq M\frac{r^2-1}{r^2}
		\lesssim M.
		\]
		Consequently,
		\begin{equation}\label{second-factor-cat}
			\frac{r^2-1}{r^2}|h''|
			+
			\frac{r^2-1}{r^3}|h'|
			\lesssim M.
		\end{equation}
		We now estimate the quadratic terms. Using
		\eqref{star-pointwise} and \eqref{second-factor-cat},
		\[
		\frac{|h|}{r^2}
		\left(
		\frac{r^2-1}{r^2}|h''|
		+
		\frac{r^2-1}{r^3}|h'|
		\right)
		\lesssim M^2.
		\]
		Furthermore,
		\[
		\frac{|h|^2}{r^6}
		\leq
		M^2\frac{(r-1)^2}{r^4}
		\lesssim M^2,
		\qquad
		\frac{r^2-1}{r^4}|h'|^2
		\leq
		M^2\frac{r^2-1}{r^2}
		\lesssim M^2,
		\]
		and
		\[
		\frac{r^2-1}{r^5}|h|\,|h'|
		\leq
		M^2
		\frac{(r^2-1)(r-1)}{r^3}
		\lesssim M^2.
		\]
		Thus all the quadratic terms are uniformly bounded by $M^2$.
		It remains to estimate the cubic term. By
		\eqref{star-pointwise} and \eqref{second-factor-cat},
		\[
		\begin{aligned}
			&\frac{r^2-1}{r^2}|h'|^2
			\left(
			\frac{r^2-1}{r^2}|h''|
			+
			\frac{r^2-1}{r^3}|h'|
			\right)
			\qquad\lesssim
			M^3(r^2-1)
			\lesssim
			M^3r^2.
		\end{aligned}
		\]
		Since $r\leq\varepsilon^{-1+b}$, it follows that
		\[
		r^2M^3
		\leq
		\varepsilon^{-2+2b}M^3.
		\]
		Combining the previous estimates gives \eqref{Qcat-star}.
		
		\subsection{The sphere} \label{sphere}
		
		We work with the following parametrization of the lower hemisphere
		of radius $R$ centered at $(R+d)e_3$:
		\[
		\mathbf{x}(r,\theta)
		=
		\bigl(
		r\cos\theta,\,
		r\sin\theta,\,
		G(r)
		\bigr),
		\qquad
		G(r)
		=
		R+d-\sqrt{R^2-r^2}.
		\]
		Set
		\[
		S(r):=\sqrt{R^2-r^2}.
		\]
		Then
		\[
		G'(r)
		=
		\frac{r}{S},
		\qquad
		G''(r)
		=
		\frac{R^2}{S^3},
		\qquad
		E
		=
		1+(G')^2
		=
		\frac{R^2}{R^2-r^2}
		=
		\frac{R^2}{S^2}.
		\]
		The induced metric in the coordinates $(r,\theta)$ is
		\[
		g
		=
		\frac{R^2}{R^2-r^2}\,dr^2
		+
		r^2\,d\theta^2,
		\]
		with inverse
		\[
		g^{rr}
		=
		\frac{R^2-r^2}{R^2},
		\qquad
		g^{r\theta}
		=
		g^{\theta r}
		=
		0,
		\qquad
		g^{\theta\theta}
		=
		\frac1{r^2}.
		\]
		The nonzero components of the second fundamental form are
		\[
		A_{rr}
		=
		\frac{R}{R^2-r^2},
		\qquad
		A_{\theta\theta}
		=
		\frac{r^2}{R}.
		\]
		Therefore
		\[
		|A|^2
		=
		\left(g^{rr}A_{rr}\right)^2
		+
		\left(g^{\theta\theta}A_{\theta\theta}\right)^2
		=
		\frac2{R^2}.
		\]
		Since
		\(
		\frac{E'}{E}
		=
		\frac{2r}{R^2-r^2},
		\)
		the general formula \eqref{|nablaA|} yields
		\[
		\begin{aligned}
			|\nabla_\Sigma A|^2
			&=
			\frac1{E^3}
			\left(
			A_{rr}'-\frac{E'}{E}A_{rr}
			\right)^2
			+
			\frac1{Er^4}
			\left(
			A_{\theta\theta}'
			-
			2\frac1r A_{\theta\theta}
			\right)^2
			+
			\frac2{Er^4}
			\left(
			\frac{r}{E}A_{rr}
			-
			\frac1r A_{\theta\theta}
			\right)^2
			=
			0.
		\end{aligned}
		\]
		Thus the second fundamental form of the round sphere is parallel.
		
		For functions
		\[
		f_R,h_R:S_R^-((R+d)e_3)\to\mathbb R,
		\]
		we write
		\[
		f(r,\theta)
		=
		f_R(\mathbf{x}(r,\theta)),
		\qquad
		h(r,\theta)
		=
		h_R(\mathbf{x}(r,\theta)).
		\]
		The gradient of $f_R$ is
		\[
		\nabla_{S_R}f_R
		=
		g^{ij}\partial_j f\,\partial_i,
		\]
		that is,
		\[
		\nabla_{S_R}f_R
		=
		\frac{R^2-r^2}{R^2}f_r\,\partial_r
		+
		\frac1{r^2}f_\theta\,\partial_\theta.
		\]
		Its squared norm is
		\[
		|\nabla_{S_R}f_R|^2
		=
		\frac{R^2-r^2}{R^2}f_r^2
		+
		\frac1{r^2}f_\theta^2.
		\]
		The covariant Hessian is
		\[
		(\nabla_{S_R}^2f_R)_{ij}
		=
		\partial_{ij}f
		-
		\Gamma^k_{ij}\partial_kf.
		\]
		The nonzero Christoffel symbols are
		\[
		\Gamma^r_{rr}
		=
		\frac{r}{R^2-r^2},
		\qquad
		\Gamma^r_{\theta\theta}
		=
		-\frac{r(R^2-r^2)}{R^2},
		\qquad 
		\Gamma^\theta_{r\theta}
		=
		\Gamma^\theta_{\theta r}
		=
		\frac1r.
		\]
		Hence
		\[
		(\nabla_{S_R}^2f_R)_{rr}
		=
		f_{rr}
		-
		\frac{r}{R^2-r^2}f_r,
		\qquad
		(\nabla_{S_R}^2f_R)_{r\theta}
		=
		f_{r\theta}
		-
		\frac1r f_\theta,
		\]
		and
		\[
		(\nabla_{S_R}^2f_R)_{\theta\theta}
		=
		f_{\theta\theta}
		+
		\frac{r(R^2-r^2)}{R^2}f_r.
		\]
		The squared norm of the Hessian is
		\[
		|\nabla_{S_R}^2f_R|^2
		=
		g^{ik}g^{j\ell}
		(\nabla_{S_R}^2f_R)_{ij}
		(\nabla_{S_R}^2f_R)_{k\ell}.
		\]
		Since the metric is diagonal,
		\[
		\begin{aligned}
			|\nabla_{S_R}^2f_R|^2
			&=
			\frac{(R^2-r^2)^2}{R^4}
			\left(
			f_{rr}
			-
			\frac{r}{R^2-r^2}f_r
			\right)^2
			+
			\frac{2(R^2-r^2)}{R^2r^2}
			\left(
			f_{r\theta}
			-
			\frac1r f_\theta
			\right)^2
			\\
			&\quad
			+
			\frac1{r^4}
			\left(
			f_{\theta\theta}
			+
			\frac{r(R^2-r^2)}{R^2}f_r
			\right)^2.
		\end{aligned}
		\]
		If $f_R$ and $g_R$ depend only on $r$, then
		$$
			\begin{aligned}
				\nabla_{S_R}f_R
				&=
				\frac{R^2-r^2}{R^2}f'\,\partial_r,
				\quad
				\nabla_{S_R}f_R\cdot\nabla_{S_R}g_R
				=
				\frac{R^2-r^2}{R^2}f'g',
				\\
				|\nabla_{S_R}f_R|
				&=
				\frac{\sqrt{R^2-r^2}}{R}|f'|.
			\end{aligned}
            $$
		Moreover,
		\[
		|\nabla_{S_R}^2f_R|^2
		=
		\frac{\bigl((R^2-r^2)f''-rf'\bigr)^2}{R^4}
		+
		\frac{(R^2-r^2)^2}{R^4r^2}(f')^2.
		\]
		For a radial function $f=f(r)$, the Laplace--Beltrami operator is
		$$
			\begin{aligned}
				\Delta_{S_R}f_R
				&=
				g^{rr}(\nabla_{S_R}^2f_R)_{rr}
				+
				g^{\theta\theta}(\nabla_{S_R}^2f_R)_{\theta\theta}
				=
				\frac{R^2-r^2}{R^2}f''
				+
				\frac{R^2-2r^2}{R^2r}f'.
			\end{aligned}
	$$
		
		\medskip
		We next compute the mean curvature of a rotationally symmetric
		normal graph over the lower hemisphere. With the choice of unit normal
		used above, we parametrize the perturbation in the direction $-\nu$ by
		\be \label{sphere-h}
		X_h(r,\theta)
		=
		\bigl(
		Q(r)\cos\theta,\,
		Q(r)\sin\theta,\,
		P(r)
		\bigr),
		\ee
		where
		\[
		Q(r)
		=
		r+\frac rR h(r),
		\qquad
		P(r)
		=
		G(r)-\frac{S}{R}h(r).
		\]
		Here
		\[
		h(r)
		=
		h_R(\mathbf{x}(r,\theta)),
		\qquad
		h_R:S_R^-((R+d)e_3)\to\mathbb R.
		\]
		
		\begin{propositio}\label{ssph}
			Let $H_h$ denote the mean curvature of the surface parametrized by
			$X_h(r,\theta)$ as in \eqref{sphere-h}. Then
			\[
			H_h
			=
			\frac2R
			-
			\left(
			\Delta_{S_R}h
			+
			\frac{2}{R^2}h
			\right)
			+
			\mathcal Q_{S_R}[h],
			\]
			where
			\[
			\mathcal Q_{S_R}[h]
			=
			Q_2(h)
			+
			\mathcal R(h),
			\]
			with
			\[
			Q_2(h)
			=
			\frac{2}{R}h\Delta_{S_R}h
			+
			\frac{2}{R^3}h^2.
			\]
			Moreover,
			\[
			|\mathcal Q_{S_R}[h]|
			\lesssim
\left( |{h \over R}|^2 +			|\nabla_{S_R}h|^2 \right)
			|\nabla_{S_R}^2h|
			+
			\frac1R
			\left(
			\frac{|h|}{R}
			+
			|\nabla_{S_R}h|
			\right)^2.
			\]
			Assume in addition that
			\[
			\frac{|h|}{R}
			+
			|\nabla_{S_R}h|
			+
			R|\nabla_{S_R}^2h|
			\leq \delta
			\]
			for some sufficiently small universal constant $\delta>0$. Then the
			cubic remainder satisfies
			$$
				\begin{aligned}
					|\mathcal R(h)|
					\leq{}&
					\frac{C}{R}
					\left[
					\left(
					\frac{|h|}{R}
					\right)^2
					+
					|\nabla_{S_R}h|^2
					\right]
					R|\nabla_{S_R}^2h|
					\\
					&+
					\frac{C}{R}
					\left[
					\left(
					\frac{|h|}{R}
					\right)^3
					+
					\frac{|h|}{R}|\nabla_{S_R}h|^2
					+
					|\nabla_{S_R}h|^4
					\right].
				\end{aligned}
			$$
			In particular,
			\[
			|\mathcal R(h)|
			\leq
			\frac{C}{R}
			\left(
			\frac{|h|}{R}
			+
			|\nabla_{S_R}h|
			+
			R|\nabla_{S_R}^2h|
			\right)^3.
			\]
		\end{propositio}
		
		\begin{proof}
			We use the formula
			\[
			H_h
			=
			\frac{P''Q'-P'Q''}{(Q'^2+P'^2)^{3/2}}
			+
			\frac{P'}{Q\sqrt{Q'^2+P'^2}}
			\]
            with $Q(r)
		=
		r+\frac rR h(r),$ and $
		P(r)
		=
		G(r)-\frac{S}{R}h(r).$
			First,
			\[
			Q'
			=
			1+\frac{h+rh'}{R},
			\qquad
			Q''
			=
			\frac{2h'+rh''}{R}.
			\]
			Recall that
			\(
			S=\sqrt{R^2-r^2}.
			\)
			Then
			\[
			P'
			=
			\frac rS
			+
			\frac{r}{RS}h
			-
			\frac SRh',
			\quad 
			P''
			=
			\frac{R^2}{S^3}
			+
			\frac{R}{S^3}h
			+
			\frac{2r}{RS}h'
			-
			\frac SRh''.
			\]
			Expanding the first term in the mean-curvature formula gives
			\[
			\frac{P''Q'-P'Q''}{(Q'^2+P'^2)^{3/2}}
			=
			\frac1R
			-
			\frac{h}{R^2}
			+
			\frac r{R^2}h'
			-
			\frac{S^2}{R^2}h''
			+
			T_2
			+
			O_3,
			\]
			where
			\[
			T_2
			=
			\frac{2S^2}{R^3}hh''
			+
			\frac{S^2}{2R^3}(h')^2
			+
			\frac{h^2}{R^3}
			-
			\frac{2r}{R^3}hh'.
			\]
			Similarly,
			\[
			\frac{P'}{Q\sqrt{Q'^2+P'^2}}
			=
			\frac1R
			-
			\frac{h}{R^2}
			-
			\frac{S^2}{R^2r}h'
			+
			U_2
			+
			O_3,
			\]
			where
			\[
			U_2
			=
			-\frac{S^2}{2R^3}(h')^2
			+
			\frac{h^2}{R^3}
			+
			\frac{2S^2}{R^3r}hh'.
			\]
			Adding the linear terms gives
			\[
			-\frac{S^2}{R^2}h''
			+
			\left(
			\frac r{R^2}
			-
			\frac{S^2}{R^2r}
			\right)h'
			-
			\frac{2}{R^2}h.
			\]
			Since
			\(
			\frac r{R^2}
			-
			\frac{S^2}{R^2r}
			=
			-\frac{R^2-2r^2}{R^2r},
			\)
			this is precisely
			\[
			-\Delta_{S_R}h
			-
			\frac{2}{R^2}h.
			\]
			For the quadratic terms,
			\[
			\begin{aligned}
				T_2+U_2
				&=
				\frac{2S^2}{R^3}hh''
				+
				\frac{2}{R^3}
				\frac{S^2-r^2}{r}hh'
				+
				\frac{2}{R^3}h^2
				=
				\frac{2}{R^3}
				\left[
				S^2hh''
				+
				\frac{R^2-2r^2}{r}hh'
				+
				h^2
				\right].
			\end{aligned}
			\]
			Since
			\[
			S^2h''
			+
			\frac{R^2-2r^2}{r}h'
			=
			R^2\Delta_{S_R}h,
			\]
			we obtain
			\[
			Q_2(h)
			=
			\frac{2}{R}h\Delta_{S_R}h
			+
			\frac{2}{R^3}h^2.
			\]
			
			To estimate the higher-order remainder, rescale the sphere to unit
			radius. In an orthonormal frame the mean-curvature operator can be
			written locally as
			\[
			H_h
			=
			\frac1R
			F\left(
			\frac hR,\,
			\nabla_{S_R}h,\,
			R\nabla_{S_R}^2h
			\right),
			\]
			where $F$ is smooth near the origin. Since the mean-curvature operator
			is quasilinear, its dependence on the Hessian variable is affine.
			After subtracting the constant, linear, and quadratic parts, Taylor's
			formula therefore gives
			\[
			\begin{aligned}
				|\mathcal R(h)|
				\leq{}&
				\frac{C}{R}
				\left[
				\left(
				\frac{|h|}{R}
				\right)^2
				+
				|\nabla_{S_R}h|^2
				\right]
				R|\nabla_{S_R}^2h|
				+
				\frac{C}{R}
				\left[
				\left(
				\frac{|h|}{R}
				\right)^3
				+
				\frac{|h|}{R}|\nabla_{S_R}h|^2
				+
				|\nabla_{S_R}h|^4
				\right].
			\end{aligned}
			\]
			Under the stated smallness assumption, the quartic term is controlled
			by the cubic quantity, and hence
			\[
			|\mathcal R(h)|
			\leq
			\frac{C}{R}
			\left(
			\frac{|h|}{R}
			+
			|\nabla_{S_R}h|
			+
			R|\nabla_{S_R}^2h|
			\right)^3.
			\]
			This concludes the proof.
		\end{proof}

		Let us recall the norm $\| \cdot \|_{**}$ introduced in \eqref{norm-h-sphere}
		$$
			\|h\|_{**}
			:=
			\frac{\|h\|}{R}
			+
			|\nabla_{S_R}h|
			+
			R|\nabla_{S_R}^2h|.
		$$
		
		The nonlinear term satisfies the following Lipschitz estimate.
		
		\begin{lemma}\label{l:Qdiff-sphere}
			There exist $\delta>0$ and $C>0$, independent of $R$, such that if
			\[
			\|h_1\|_{**}+\|h_2\|_{**}\leq\delta,
			\]
			then
			\begin{equation}\label{Qdiff-sphere}
				\left|
				\mathcal Q_{S_R}[h_1]
				-
				\mathcal Q_{S_R}[h_2]
				\right|
				\leq
				\frac{C}{R}
				\left(
				\|h_1\|_{**}
				+
				\|h_2\|_{**}
				\right)
				\|h_1-h_2\|_{**}.
			\end{equation}
			Moreover,
			\begin{equation}\label{Q2diff-sphere}
				\left|
				Q_2(h_1)-Q_2(h_2)
				\right|
				\leq
				\frac{C}{R}
				\left(
				\|h_1\|_{**}
				+
				\|h_2\|_{**}
				\right)
				\|h_1-h_2\|_{**},
			\end{equation}
			while the cubic remainder satisfies
			\begin{equation}\label{Rdiff-sphere}
				\left|
				\mathcal R(h_1)-\mathcal R(h_2)
				\right|
				\leq
				\frac{C}{R}
				\left(
				\|h_1\|_{**}^2
				+
				\|h_2\|_{**}^2
				\right)
				\|h_1-h_2\|_{**}.
			\end{equation}
		\end{lemma}
		
		\begin{proof}
			After rescaling the sphere to unit radius, the mean-curvature operator
			can be written locally in the form
			\[
			H_h
			=
			\frac1R
			F\left(
			\frac{h}{R},
			\nabla_{S_R}h,
			R\nabla_{S_R}^2h
			\right),
			\]
			where $F$ is smooth in a neighborhood of the origin. Introduce
			\[
			Z(h)
			:=
			\left(
			\frac{h}{R},
			\nabla_{S_R}h,
			R\nabla_{S_R}^2h
			\right).
			\]
			Then, pointwise,
			\[
			|Z(h)|\lesssim \|h\|_{**}.
			\]
			
			Since $\mathcal Q_{S_R}$ is obtained by subtracting from $H_h$ its
			constant and linear parts, there exists a smooth function
			$\mathcal F$ such that
			\[
			\mathcal Q_{S_R}[h]
			=
			\frac1R\mathcal F(Z(h)),
			\qquad
			\mathcal F(0)=0,
			\qquad
			D\mathcal F(0)=0.
			\]
			By the mean value theorem,
			\[
			\begin{aligned}
				\mathcal F(Z(h_1))-\mathcal F(Z(h_2))
				&=
				\int_0^1
				D\mathcal F
				\left(
				Z(h_2)
				+s\bigl(Z(h_1)-Z(h_2)\bigr)
				\right)
				\cdot
				\bigl(Z(h_1)-Z(h_2)\bigr)\,ds.
			\end{aligned}
			\]
			Since $D\mathcal F(0)=0$ and $\mathcal F$ is smooth, for $|Z|$ small,
			\[
			|D\mathcal F(Z)|
			\leq
			C|Z|.
			\]
			Therefore
			\[
			\begin{aligned}
				\left|
				\mathcal F(Z(h_1))
				-
				\mathcal F(Z(h_2))
				\right|
				&\leq
				C
				\left(
				|Z(h_1)|
				+
				|Z(h_2)|
				\right)
				|Z(h_1)-Z(h_2)|
				\\
				&\leq
				C
				\left(
				\|h_1\|_{**}
				+
				\|h_2\|_{**}
				\right)
				\|h_1-h_2\|_{**}.
			\end{aligned}
			\]
			Multiplying by $R^{-1}$ gives \eqref{Qdiff-sphere}.
			
			For the quadratic part, we write
			\[
			\begin{aligned}
				Q_2(h_1)-Q_2(h_2)
				={}&
				\frac{2}{R}
				\left[
				(h_1-h_2)\Delta_{S_R}h_1
				+
				h_2\Delta_{S_R}(h_1-h_2)
				\right]
				\\
				&+
				\frac{2}{R^3}
				(h_1-h_2)(h_1+h_2).
			\end{aligned}
			\]
			Using
			\[
			R|\Delta_{S_R}h|
			\lesssim
			R|\nabla_{S_R}^2h|
			\lesssim
			\|h\|_{**},
			\qquad
			\frac{|h|}{R}
			\leq
			\|h\|_{**},
			\]
			we obtain
			\[
			\left|
			Q_2(h_1)-Q_2(h_2)
			\right|
			\leq
			\frac{C}{R}
			\left(
			\|h_1\|_{**}
			+
			\|h_2\|_{**}
			\right)
			\|h_1-h_2\|_{**},
			\]
			which proves \eqref{Q2diff-sphere}.
			
			Finally, after subtracting also the quadratic part, the remainder can
			be written as
			\[
			\mathcal R(h)
			=
			\frac1R\mathcal G(Z(h)),
			\]
			where
			\[
			\mathcal G(0)=0,
			\qquad
			D\mathcal G(0)=0,
			\qquad
			D^2\mathcal G(0)=0.
			\]
			Hence, for $|Z|$ sufficiently small,
			\[
			|D\mathcal G(Z)|
			\leq
			C|Z|^2.
			\]
			Applying the mean value theorem once more gives
			\[
			\begin{aligned}
				\left|
				\mathcal G(Z(h_1))
				-
				\mathcal G(Z(h_2))
				\right|
				&\leq
				C
				\left(
				|Z(h_1)|^2
				+
				|Z(h_2)|^2
				\right)
				|Z(h_1)-Z(h_2)|
				\\
				&\leq
				C
				\left(
				\|h_1\|_{**}^2
				+
				\|h_2\|_{**}^2
				\right)
				\|h_1-h_2\|_{**}.
			\end{aligned}
			\]
			Multiplying again by $R^{-1}$ yields \eqref{Rdiff-sphere}.
		\end{proof}

		\bibliographystyle{plain}
\bibliography{references}
		
	\end{document}